\documentclass[11pt]{article}
\usepackage{amsthm, amsmath, amssymb, amsfonts, url, booktabs, tikz,
setspace, fancyhdr, hyperref, todonotes, caption}
\usepackage[margin = 1in]{geometry}

\usepackage[color, final]{showkeys} 

\definecolor{refkey}{gray}{.75}
\definecolor{labelkey}{gray}{.5}

\newtheorem{theorem}{Theorem}[section]
\newtheorem{proposition}[theorem]{Proposition}
\newtheorem{lemma}[theorem]{Lemma}
\newtheorem{corollary}[theorem]{Corollary}

\newtheorem*{thm}{Theorem}

\theoremstyle{definition}
\newtheorem{definition}[theorem]{Definition}
\newtheorem{example}[theorem]{Example}
\newtheorem{setup}[theorem]{Setup}

\theoremstyle{remark}
\newtheorem*{remark}{Remark}

\DeclareMathOperator{\DISC}{DISC}
\newcommand{\spa}{\mathrm{sp}}

\newcommand{\abs}[1]{\left\lvert#1\right\rvert}

\newcommand{\floor}[1]{\left\lfloor #1 \right\rfloor}
\newcommand{\ceil}[1]{\left\lceil #1 \right\rceil}
\newcommand{\paren}[1]{\left( #1 \right)}
\newcommand{\sqb}[1]{\left[ #1 \right]}
\newcommand{\set}[1]{\left\{ #1 \right\}}
\newcommand{\setcond}[2]{\left\{ #1 \;\middle\vert\; #2 \right\}}
\newcommand{\cond}[2]{\left( #1 \;\middle\vert\; #2 \right)}
\newcommand{\sqcond}[2]{\left[ #1 \;\middle\vert\; #2 \right]}
\newcommand{\wt}{\widetilde}

\newcommand{\x}{\times}
\newcommand{\approxmod}{\mathop{\approx}}

\newcommand{\CC}{\mathbb{C}}
\newcommand{\EE}{\mathbb{E}}

\newcommand{\RR}{\mathbb{R}}
\newcommand{\bx}{\mathbf{x}}
\newcommand{\by}{\mathbf{y}}
\newcommand{\bz}{\mathbf{z}}

\newcommand{\D}{\Delta}
\newcommand{\e}{\epsilon}

\usetikzlibrary{decorations.pathmorphing}
\usetikzlibrary{decorations.pathreplacing}
\usetikzlibrary{decorations.markings}
\usetikzlibrary{arrows}

\tikzstyle{p}+=[fill=black, circle, minimum width = 1pt, inner sep =
1pt]
\tikzstyle{w}+=[fill=white, draw, circle, minimum width = 1pt, inner sep =
1.5pt]
\tikzstyle{l}+=[fill=white,inner sep = 1pt,font=\footnotesize]
\tikzstyle{Gamma}+=[decoration={snake, segment length=1.3mm,
  amplitude=.4mm}]
\tikzstyle{dense}+=[line width=1mm]
\newcommand{\LargeLeftarrow}{\scalebox{2}{$\Leftarrow$}}
\newcommand{\LargeUparrow}{\scalebox{2}{$\Uparrow$}}

\begin{document}

\title{Extremal results in sparse pseudorandom graphs}

\author{David Conlon\thanks{Mathematical Institute, Oxford OX1 3LB, United
Kingdom. E-mail: {\tt david.conlon@maths.ox.ac.uk}. Research supported by a Royal Society University Research Fellowship.} \and Jacob Fox\thanks{Department of Mathematics, MIT, Cambridge, MA 02139-4307. Email: {\tt fox@math.mit.edu}. Research
supported by a Simons Fellowship and NSF grant DMS-1069197.} \and Yufei Zhao \thanks{Department of Mathematics,
MIT, Cambridge, MA 02139-4307. Email: {\tt yufeiz@math.mit.edu}. Research
supported by an Akamai Presidential Fellowship.} \and \vspace{2mm}
Accepted for publication in Advances in Mathematics}


\date{}

\maketitle

\begin{abstract}
Szemer\'edi's regularity lemma is a fundamental tool in extremal combinatorics. However, the original version is only helpful in studying dense graphs. In the 1990s, Kohayakawa and R\"odl proved
an analogue of Szemer\'edi's regularity lemma for sparse graphs as part of a
general program toward extending extremal results to sparse graphs. Many of
the key applications of Szemer\'edi's regularity lemma use an associated
counting lemma. In order to prove extensions of these results which also apply to sparse
graphs, it remained a well-known open problem to prove a counting lemma in
sparse graphs.

The main advance of this paper lies in a new counting lemma, proved
following the functional approach of Gowers, which complements the sparse
regularity lemma of Kohayakawa and R\"odl, allowing us to count small graphs
in regular subgraphs of a sufficiently pseudorandom graph. We use this to
prove sparse extensions of several well-known combinatorial theorems, including the removal lemmas for graphs and groups, the Erd\H{o}s-Stone-Simonovits theorem and Ramsey's theorem. These results extend and improve
upon a substantial body of previous work.

\end{abstract}

\tableofcontents

\section{Introduction}

Szemer\'edi's regularity lemma is one of the most powerful tools in extremal combinatorics. Roughly speaking, it says that the vertex set of every graph can be partitioned into a bounded number of parts so that the induced bipartite graph between almost all pairs of parts is pseudorandom. Many important results in graph theory, such as the graph removal lemma and the Erd\H{o}s-Stone-Simonovits theorem on Tur\'an numbers, have straightforward proofs using the regularity lemma.

Crucial to most applications of the regularity lemma is the use of a counting lemma. A counting lemma, roughly speaking, is a result that says that the number of embeddings of a fixed graph $H$ into a pseudorandom graph $G$ can be estimated by pretending that $G$ were a genuine random graph. The combined application of the regularity lemma and a counting lemma is known as the regularity method, and has important applications in graph theory, combinatorial geometry, additive combinatorics and theoretical computer science. For surveys on the regularity method and its applications, see \cite{KR03a, KS96, RS10}.

One of the limitations of Szemer\'edi's regularity lemma is that it is only meaningful for dense graphs.   While an analogue of the regularity lemma for sparse graphs has been proven by Kohayakawa \cite{K97} and by R\"odl (see also \cite{GS05, Sc11}), the problem of proving an associated counting lemma for sparse graphs has turned out to be much more difficult.  In random graphs, proving such an embedding lemma is a famous problem, known as the K\L R conjecture \cite{KLR97}, which has only been resolved very recently \cite{BMS12, CGSS12}.

Establishing an analogous result for pseudorandom graphs has been a central problem in this area. Certain partial results are known in this case \cite{KRSS10, KRS04}, but it has remained an open problem to prove a counting lemma for embedding a general fixed subgraph. We resolve this difficulty, proving a counting lemma for embedding any fixed small graph into subgraphs of sparse pseudorandom graphs.

As applications, we prove sparse extensions of several well-known combinatorial theorems, including the  removal lemmas for graphs and groups, the Erd\H{o}s-Stone-Simonovits theorem, and Ramsey's theorem. Proving such sparse analogues for classical combinatorial results has been an important trend in modern combinatorics research. For example, a sparse analogue of Szemer\'edi's theorem was an integral part of Green and Tao's proof \cite{GT08} that the primes contain arbitrarily long arithmetic progressions.

\subsection{Pseudorandom graphs}\label{sec:pseudos}

The binomial random graph $G_{n, p}$ is formed by taking an empty graph on $n$ vertices and choosing to add each edge independently with probability $p$. These graphs tend to be very well-behaved. For example, it is not hard to show that with high probability all large vertex subsets $X,Y$ have density approximately $p$ between them. Motivated by the question of determining when a graph behaves in a random-like manner, Thomason \cite{T85, T87} began the first systematic study of this property. Using a slight variant of Thomason's notion, we say that a graph on vertex set $V$ is {\it $(p, \beta)$-jumbled} if, for all vertex subsets $X,Y \subseteq V$,
\[|e(X,Y) - p|X||Y|| \leq \beta \sqrt{|X||Y|}.\]
The random graph $G_{n,p}$ is, with high probability, $(p, \beta)$-jumbled with $\beta = O(\sqrt{pn})$. It is not hard to show \cite{EGPS88,ES71} that this is optimal and that a graph on $n$ vertices with $p \leq 1/2$ cannot be $(p, \beta)$-jumbled with $\beta = o(\sqrt{pn})$. Nevertheless, there are many explicit examples which are optimally jumbled in that $\beta =O(\sqrt{pn})$. The Paley graph with vertex set $\mathbb{Z}_p$, where $p \equiv 1 (\mbox{mod } 4)$ is prime, and edge set given by connecting $x$ and $y$ if their difference is a quadratic residue is such a graph with $p = \frac{1}{2}$ and $\beta = O(\sqrt{n})$. Many more examples are given in the excellent survey \cite{KS06}.

A fundamental result of Chung, Graham and Wilson \cite{CGW89} states that for graphs of density $p$, where $p$ is a fixed positive constant, the property of being $(p, o(n))$-jumbled is equivalent to a number of other properties that one would typically expect in a random graph. The following theorem details some of these many equivalences.

\begin{thm} \label{CGW}
For any fixed $0 < p < 1$ and any sequence of graphs $(\Gamma_n)_{n \in \mathbb{N}}$ with $|V(\Gamma_n)| = n$ the following properties are equivalent.
\begin{itemize}
\item[$P_1$:]
$\Gamma_n$ is $(p, o(n))$-jumbled, that is, for all subsets $X,Y \subseteq V(\Gamma_n)$, $e(X,Y) = p |X||Y| + o(n^2)$;

\item[$P_2$:]
$e(\Gamma_n) \geq p \binom{n}{2} + o(n^2)$, $\lambda_1(\Gamma_n) = p n + o(n)$ and $|\lambda_2(\Gamma_n)| = o(n)$, where $\lambda_i(\Gamma_n)$ is the $i$th largest eigenvalue, in absolute value, of the adjacency matrix of $\Gamma_n$;

\item[$P_3$:]
for all graphs $H$, the number of labeled induced copies of $H$ in $\Gamma_n$ is $p^{\binom{\ell}{2}} n^{\ell} + o(n^{\ell})$, where $\ell = V(H)$;

\item[$P_4$:]
$e(\Gamma_n) \geq p \binom{n}{2} + o(n^2)$ and the number of labeled cycles of length $4$ in $\Gamma_n$ is at most $p^4 n^4 + o(n^4)$.
\end{itemize}
\end{thm}

Any graph sequence which satisfies any (and hence all) of these properties is said to be {\it $p$-quasirandom}. The most surprising aspect of this theorem, already hinted at in Thomason's work, is that if the number of cycles of length $4$ is as one would expect in a binomial random graph then this is enough to imply that the edges are very well-spread. This theorem has been quite influential. It has led to the study of quasirandomness in other structures such as hypergraphs \cite{CG91, G06}, groups \cite{G08}, tournaments, permutations and  sequences (see \cite{CG08} and it references), and progress on problems in different areas (see, e.g., \cite{C09,G07,G08}). It is also closely related to Szemer\'edi's regularity lemma and its recent hypergraph generalization \cite{G07, NRS06, RS04, T06} and all proofs of Szemer\'edi's theorem \cite{Sz75} on long arithmetic progressions in dense subsets of the integers use some notion of quasirandomness.

For sparser graphs, the equivalences between the natural
generalizations of these properties are not so clear cut (see
\cite{CG02, KR03, KRS04} for discussions). In this case, it is natural
to generalize the jumbledness condition for dense graphs by
considering graphs which are $(p, o(pn))$-jumbled. Otherwise, we would
not even have control over the density in the whole set. However, it
is no longer the case that being $(p, o(pn))$-jumbled implies that the
number of copies of any subgraph $H$ agrees approximately with the expected
count. For $H = K_{3,3}$ and $p = n^{-1/3}$, it is easy to see this by
taking the random graph $G_{n,p}$ and changing three vertices $u$, $v$
and $w$ so that they are each connected to everything else. This does
not affect the property of being $(p, o(pn))$-jumbled but it does
affect the $K_{3,3}$ count, since as well as the roughly $p^9 n^6 =
n^3$ copies of $K_{3,3}$ that one expects in a random graph, one gets
a further $\Omega(n^3)$ copies of $K_{3,3}$ containing all of $u$, $v$
and $w$.

However, for any given graph $H$ one can find a function $\beta_H : = \beta_H(p, n)$ such that if $\Gamma$ is a $(p, \beta_H)$-jumbled graph then $\Gamma$ contains a copy of $H$. Our chief concern in this paper will be to determine jumbledness conditions which are sufficient to imply other properties. In particular, we will be concerned with determining conditions under which certain combinatorial theorems continue to hold within jumbled graphs.

One particularly well-known class of $(p, \beta)$-jumbled graphs is the collection of {\it $(n, d, \lambda)$-graphs}. These are graphs on $n$ vertices which are $d$-regular and such that all eigenvalues of the adjacency matrix, save the largest, are smaller in absolute value than $\lambda$. The famous expander mixing lemma tells us that these graphs are $(p, \beta)$-jumbled with $p = d/n$ and $\beta = \lambda$. Bilu and Linial \cite{BL06} proved a converse of this fact, showing that every $(p,\beta)$-jumbled  $d$-regular  graph is an  $(n,d,\lambda)$-graph with $\lambda=O(\beta \log (d/\beta))$.  This shows that the jumbledness parameter $\beta$ and the second largest in absolute value eigenvalue $\lambda$ of a regular graph are within a logarithmic factor of each other.

Pseudorandom graphs have many surprising properties and applications and have recently attracted a lot of attention both in combinatorics and theoretical computer science (see, e.g., \cite{KS06}). Here we will focus on their behavior with respect to extremal properties. We discuss these properties in the next section.

\subsection{Extremal results in pseudorandom graphs} \label{sec:extremal}

In this paper, we study the extent to which several well-known combinatorial statements continue to hold relative to pseudorandom graphs or, rather, $(p, \beta)$-jumbled graphs and $(n, d, \lambda)$-graphs.

One of the most important applications of the regularity method is the graph removal lemma \cite{ADLRY94, RS78}. In the following statement and throughout the paper, $v(H)$ and $e(H)$ will denote the number of vertices and edges in the graph $H$, respectively. The graph removal lemma states that for every fixed graph $H$ and every $\e > 0$ there exists $\delta > 0$ such that if $G$ contains at most $\delta n^{v(H)}$ copies of $H$ then $G$ may be made $H$-free by removing at most $\e n^2$ edges. This innocent looking result, which follows easily from Szemer\'edi's regularity lemma and the graph counting lemma, has surprising applications in diverse areas, amongst others a simple proof of Roth's theorem on $3$-term arithmetic progressions in dense subsets of the integers. It is also much more difficult to prove than one might expect, the best known bound \cite{F11} on $\delta^{-1}$ being a tower function of height on the order of $\log \e^{-1}$.

An analogue of this result for random graphs (and hypergraphs) was proven in \cite{CG12}. For pseudorandom graphs, the following analogue of the triangle removal lemma was recently proven by Kohayakawa, R\"odl, Schacht and Skokan \cite{KRSS10}.

\begin{thm} \label{KRSS}
For every $\e > 0$, there exist $\delta > 0$ and $c > 0$ such that if $\beta \leq c p^3 n$ then any $(p, \beta)$-jumbled graph $\Gamma$ on $n$ vertices has the following property. Any subgraph of $\Gamma$ containing at most $\delta p^3 n^3$ triangles may be made triangle-free by removing at most $\e p n^2$ edges.
\end{thm}

Here we extend this result to all $H$. The {\it degeneracy $d(H)$} of
a graph $H$ is the smallest nonnegative integer $d$ for which there
exists an ordering of the vertices of $H$ such that each vertex has at
most $d$ neighbors which appear earlier in the ordering. Equivalently,
it may be defined as $d(H) = \max\{\delta(H'): H' \subseteq H\}$, where
$\delta(H)$ is the minimum degree of $H$. Throughout the paper, we will also use the
standard notation $\Delta(H)$ for the maximum degree of $H$.

The parameter we will use in our theorems, which we refer to as the $2$-degeneracy $d_2(H)$, is related to both of these natural parameters. Given an ordering $v_1, \dots, v_m$ of the vertices of $H$ and $i \leq j$, let $N_{i-1}(j)$ be the number of neighbors $v_h$ of $v_j$ with $h \leq i-1$. We then define $d_2(H)$ to be the minimum $d$ for which there is an ordering of the edges as $v_1, \dots, v_m$ such that for any edge $v_i v_j$ with $i < j$  the sum $N_{i-1}(i) + N_{i-1}(j) \leq 2d$. Note that $d_2(H)$ may be a half-integer. For comparison with degeneracy, note that $\frac{d(H)}{2} \leq d_2(H) \leq d(H) - \frac{1}{2}$ and both sides can be sharp.

\begin{theorem} \label{RemovalIntro}
For every graph $H$ and every $\e > 0$, there exist $\delta > 0$ and $c > 0$ such that if $\beta \leq c p^{d_2(H) + 3}n$ then any $(p, \beta)$-jumbled graph $\Gamma$ on $n$ vertices has the following property. Any subgraph of $\Gamma$ containing at most $\delta p^{e(H)} n^{v(H)}$ copies of  $H$ may be made $H$-free by removing at most $\e p n^2$ edges.
\end{theorem}

We remark that for many graphs $H$, the constant $3$ in the exponent
of this theorem may be improved, and this applies equally to all of the theorems stated below. While we will not dwell on this comment, we will call attention to it on occasion throughout the paper, pointing out where the above result may be improved.
Note that the above theorem generalizes the graph removal lemma by considering the case $\Gamma=K_n$, which is $(p,\beta)$-jumbled with $p=1$ and $\beta=1$. For the same reason, the other results we establish extend the original versions.

Green \cite{G05a} developed an arithmetic regularity lemma and used it to deduce an arithmetic removal lemma in abelian groups which extends Roth's theorem. Green's proof of the arithmetic regularity lemma relies on Fourier analytic techniques. Kr\'al, Serra and Vena \cite{KSV09} found a new proof of the arithmetic removal lemma using the removal lemma for directed cycles which extends to all groups. They proved that for each  $\epsilon>0$ and integer $m \geq 3$ there is $\delta>0$ such that if $G$ is a group of order $n$ and $A_1,\ldots,A_m$ are subsets of $G$ such that there are at most $\delta n^{m-1}$ solutions to the equation $x_1x_2 \cdots x_m=1$ with $x_i \in A_i$ for all $i$, then it is possible to remove at most $\epsilon n$ elements from each set $A_i$ so as to obtain sets $A_i'$ for which there are no solutions to $x_1x_2 \cdots x_m=1$ with $x_i \in A'_i$ for all $i$.

By improving the bound in Theorem \ref{RemovalIntro} for cycles, we obtain the following sparse extension of the removal lemma for groups. The {\it Cayley graph} $G(S)$ of a subset $S$ of a group
$G$ has vertex set $G$ and $(x,y)$ is an edge of $G$ if $x^{-1} y \in
S$. We say that a subset $S$ of a group $G$ is $(p,\beta)$-jumbled if
the Cayley graph $G(S)$ is $(p,\beta)$-jumbled. When $G$ is abelian, if
$\abs{\sum_{x \in S} \chi(x)} \leq \beta$ for all nontrivial
characters $\chi \colon G \to \CC$, then $S$ is $(\frac{\abs{S}}{\abs{G}},\beta)$-jumbled
(see \cite[Lemma 16]{KRSS10}). Let $k_3=3$, $k_4=2$, $k_{m}=1+\frac{1}{m-3}$ if $m \geq 5$ is odd, and $k_m=1+\frac{1}{m-4}$ if $m \geq 6$ is even. Note that $k_m$ tends to $1$ as $m \to \infty$.

\begin{theorem} \label{thm:removal-groups} For each $\epsilon>0$ and
  integer $m \geq 3$, there are $c,\delta>0$ such that the following
  holds. Suppose $B_1,\ldots,B_m$ are subsets of a group $G$ of order $n$ such that
  each $B_i$ is $(p,\beta)$-jumbled with $\beta \leq cp^{k_m}n$. If
  subsets $A_i \subseteq B_i$ for $i=1,\ldots,m$ are such that there
  are at most $\delta |B_1|\cdots|B_m|/n$ solutions to the equation
  $x_1x_2 \cdots x_m=1$ with $x_i \in A_i$ for all $i$, then it is
  possible to remove at most $\epsilon |B_i|$ elements from each set
  $A_i$ so as to obtain sets $A_i'$ for which there are no solutions
  to $x_1x_2 \cdots x_m=1$ with $x_i \in A'_i$ for all $i$.
\end{theorem}

This result easily implies a Roth-type theorem in quite sparse pseudorandom subsets of a group. We say that a subset $B$ of a group $G$ is {\it $(\epsilon,m)$-Roth} if for all integers $a_1,\ldots,a_m$ which satisfy $a_1+\cdots+a_m = 0$ and $\gcd(a_i,|G|)=1$ for $1 \leq i \leq m$, every subset $A \subseteq B$ which has no nontrivial solution to $x_1^{a_1}x_2^{a_2}\cdots x_m^{a_m}=1$ has $|A| \leq \epsilon |B|$.

\begin{corollary}\label{rothtype} For each $\epsilon>0$ and integer $m \geq 3$, there is $c>0$ such that the following holds. If $G$ is a group of order $n$ and $B$ is a $(p,\beta)$-jumbled subset of $G$ with $\beta \leq cp^{k_m}n$, then $B$ is $(\epsilon,m)$-Roth.
\end{corollary}

Note that Roth's theorem on $3$-term arithmetic progressions in dense sets of integers, follows from the special case of this result with $B=G=\mathbb{Z}_n$, $m=3$, and $a_1=a_2=1$, $a_3=-2$. The rather weak pseudorandomness condition in Corollary \ref{rothtype} shows that even quite sparse pseudorandom subsets of a group have the Roth property. For example, if $B$ is optimally jumbled, in that $\beta=O(\sqrt{pn})$ and $p  \geq Cn^{-\frac{1}{2k_m-1}}$, then $B$ is $(\epsilon,m)$-Roth. This somewhat resembles a key part of the proof of the Green-Tao theorem that the primes contain arbitrarily long arithmetic progressions, where they show that pseudorandom sets of integers have the Szemer\'edi property. As Corollary \ref{rothtype} applies to quite sparse pseudorandom subsets, it may lead to new applications in number theory.

Our methods are quite general and also imply similar results for other well-known combinatorial theorems. We say that a graph $\Gamma$ is {\it $(H, \e)$-Tur\'an} if any subgraph of $\Gamma$ with at least
$$\left(1 - \frac{1}{\chi(H) - 1} + \e\right) e(\Gamma)$$
edges contains a copy of $H$. Tur\'an's theorem itself \cite{T41}, or rather a generalization known as the  Erd\H{o}s-Stone-Simonovits theorem \cite{ES46}, says that $K_n$ is $(H, \e)$-Tur\'an provided that $n$ is large enough.

To find other graphs which are $(H,\epsilon)$-Tur\'an, it is natural to try the random graph $G_{n,p}$. A recent result of Conlon and Gowers \cite{CG12}, also proved independently by Schacht \cite{S12}, states that for every $t \geq 3$ and $\e > 0$ there exists a constant $C$ such that if $p \geq C n^{-2/(t+1)}$ the graph $G_{n,p}$ is with high probability $(K_t, \e)$-Tur\'an. This confirms a conjecture of Haxell, Kohayakawa, \L uczak and R\"odl \cite{HKL95, KLR97} and, up to the constant $C$, is best possible. Similar results also hold for more general graphs $H$ and hypergraphs.

For pseudorandom graphs and, in particular, $(n, d, \lambda)$-graphs, Sudakov, Szab\'o and Vu \cite{SSV05} showed the following. A similar result, but in a slightly more general context, was proved by Chung \cite{C05}.

\begin{thm} \label{SSV}
For every $\epsilon > 0$ and every positive integer $t \geq 3$, there exists $c > 0$ such that if $\lambda \leq c d^{t-1}/n^{t-2}$ then any $(n, d, \lambda)$-graph is $(K_t, \epsilon)$-Tur\'an.
\end{thm}

For $t = 3$, an example of Alon \cite{A94} shows that this is best possible. His example gives something even stronger, a triangle-free $(n,d,\lambda)$-graph for which $\lambda \leq c \sqrt{d}$ and $d \geq n^{2/3}$. Therefore, no combinatorial statement about the existence of triangles as subgraphs can surpass the threshold $\lambda \leq c d^2/n$. It has also been conjectured \cite{FLS12, KS06, SSV05} that $\lambda \leq c d^{t-1}/n^{t-2}$ is a natural boundary for finding $K_t$ as a subgraph in a pseudorandom graph, but no examples of such graphs exist for $t \geq 4$.  Finding such graphs remains an important open problem on pseudorandom graphs.

For triangle-free graphs $H$, Kohayakawa, R\"odl, Schacht, Sissokho and Skokan \cite{KRSSS07} proved the following result which gives a jumbledness condition that implies that a graph is $(H, \e)$-Tur\'an.

\begin{thm} \label{KRSSS}
For any fixed triangle-free graph $H$ and any $\e > 0$, there exists $c > 0$ such that if $\beta \leq c p^{\nu(H)} n$ then any $(p, \beta)$-jumbled graph on $n$ vertices is $(H, \e)$-Tur\'an. Here $\nu(H) = \frac{1}{2}(d(H) + D(H) + 1)$, where $D(H) = \min\{2d(H), \D(H)\}$.
\end{thm}

More recently, the case where $H$ is an odd cycle was studied by Aigner-Horev, H\`an and Schacht \cite{AHS12}, who proved the following result, optimal up to the logarithmic factors \cite{AK98}.

\begin{thm} \label{AHS}
For every odd integer $\ell \geq 3$ and any $\e > 0$, there exists $c > 0$ such that if $\beta \log^{\ell-3} n \leq c p^{1 + 1/(\ell-2)}n$ then any $(p, \beta)$-jumbled graph on $n$ vertices is $(C_\ell, \e)$-Tur\'an.
\end{thm}

In this paper, we prove that a similar result holds, but for general graphs $H$ and, in most cases, with a better bound on $\beta$.

\begin{theorem} \label{TuranIntro}
For every graph $H$ and every $\e > 0$, there exists $c > 0$ such that if $\beta \leq c p^{d_2(H) + 3}n$ then any $(p, \beta)$-jumbled graph on $n$ vertices is $(H, \e)$-Tur\'an.
\end{theorem}

We may also prove a structural version of this theorem, known as a stability result. In the dense case, this result, due to Erd\H{o}s and Simonovits \cite{Si68}, states that if an $H$-free
graph contains almost $\left(1 - \frac{1}{\chi(H) - 1}\right) \binom{n}{2}$ edges, then it must be very close to being $(\chi(H) - 1)$-partite.

\begin{theorem} \label{StabIntro}
For every graph $H$ and every $\e > 0$, there exist $\delta > 0$ and $c > 0$ such that if $\beta \leq c p^{d_2(H) + 3}n$ then any $(p, \beta)$-jumbled graph $\Gamma$ on $n$ vertices has the following property. Any $H$-free subgraph of $\Gamma$ with at least $\left(1 - \frac{1}{\chi(H) - 1} - \delta\right) p \binom{n}{2}$ edges may be made $(\chi(H) - 1)$-partite by removing at most $\epsilon p n^2$ edges.
\end{theorem}

The final main result that we will prove concerns Ramsey's theorem \cite{R30}. This states that for any graph $H$ and positive integer $r$, if $n$ is sufficiently large, any $r$-coloring of the edges of $K_n$ contains a monochromatic copy of $H$.

To consider the analogue of this result in sparse graphs, let us say that a graph $\Gamma$ is {\it $(H, r)$-Ramsey} if, in any $r$-coloring of the edges of $\Gamma$, there is guaranteed to be a monochromatic copy of $H$. For $G_{n,p}$, a result of R\"odl and Ruci\'nski \cite{RR95} determines the threshold down to which the random graph is $(H,r)$-Ramsey with high probability. For most graphs, including the complete graph $K_t$, this threshold is the same as for the Tur\'an property. These results were only extended to hypergraphs comparatively recently, by Conlon and Gowers \cite{CG12} and by Friedgut, R\"odl and Schacht \cite{FRS10}.

Very little seems to be known about the $(H,r)$-Ramsey property relative to pseudorandom graphs. In the triangle case, results of this kind are implicit in some recent papers \cite{DR08, Lu07} on Folkman numbers, but no general theorem seems to be known. We prove the following.

\begin{theorem}\label{RamseyIntro}
For every graph $H$ and every positive integer $r \geq 2$, there exists $c > 0$ such that if $\beta \leq c p^{d_2(H) + 3}n$ then any $(p, \beta)$-jumbled graph on $n$ vertices is $(H, r)$-Ramsey.
\end{theorem}

One common element to all these results is the requirement that $\beta \leq c p^{d_2(H) +3}n$. It is not hard to see that this condition is almost sharp. Consider the binomial random graph on $n$ vertices where each edge is chosen with probability $p = c n^{-2/d(H)}$, where $c < 1$.  By the definition of degeneracy, there exists some subgraph $H'$ of $H$ such that $d(H)$ is the minimum degree of $H'$. Therefore, $e(H') \geq v(H') d(H)/2$ and the expected number of copies of $H'$ is at most
\[p^{e(H')} n^{v(H')} \leq (p^{d(H)/2} n)^{v(H')} < 1.\]
We conclude that with positive probability $G_{n,p}$ does not contain a copy of $H'$ or, consequently, of $H$. On the other hand, with high probability, it is $(p, \beta)$-jumbled with
\[\beta = O(\sqrt{pn}) = O(p^{(d(H) + 2)/4} n).\]
Since $d_2(H)$ differs from $d(H)$ by at most a constant factor, we therefore see that, up to a multiplicative constant in the exponent of $p$, our results are best possible.

If $H = K_t$,  it is sufficient, for the various combinatorial theorems above to hold, that the graph $\Gamma$ be $(p, cp^t n)$-jumbled. For triangles, the example of Alon shows that there are $(p, cp^2 n)$-jumbled graphs which do not contain any triangles and, for $t \geq 4$, it is conjectured \cite{FLS12, KS06, SSV05} that there are $(p, cp^{t-1} n)$-jumbled graphs which do not contain a copy of $K_t$. If true, this would imply that in the case of cliques all of our results are sharp up to an additive constant of one in the exponent. A further discussion relating to the optimal exponent of $p$ for general graphs is in the concluding remarks.

\subsection{Regularity and counting lemmas} \label{sec:countintro}

One of the key tools in extremal graph theory is Szemer\'edi's regularity lemma \cite{Sz78}. Roughly speaking, this says that any graph may be partitioned into a collection of vertex subsets so that the bipartite graph between most pairs of vertex subsets is random-like. To be more precise, we introduce some notation. It will be to our advantage to be quite general from the outset.

A \emph{weighted graph} on a set of vertices $V$ is a symmetric function $G \colon V \x V \to [0,1]$. Here symmetric means that $G(x,y) = G(y,x)$. A weighted graph is bipartite (or multipartite) if it is supported on the edge set of a bipartite (or multipartite graph). A graph can be viewed as a weighted graph by taking $G$ to be the characteristic function of the edges.

Note that here and throughout the remainder of the paper, we will use integral notation for summing over vertices in a graph. For example, if $G$ is a bipartite graph with vertex sets $X$ and $Y$, and $f$ is any function $X \x Y \to \RR$, then we write
\[\int_{\substack{x \in X \\ y \in Y}} f(x,y) \ dxdy :=\frac{1}{\abs{X}\abs{Y}} \sum_{x \in X} \sum_{y \in Y} f(x,y).\]
The measure $dx$ will always denote the uniform probability distribution on $X$. The advantage of the integral notation is that we do not need to keep track of the number of vertices in $G$. All our formulas are, in some sense, scale-free with respect to the order of $G$. Consequently, our results also have natural extensions to graph limits~\cite{LS06}, although we do not explore this direction here.

\begin{definition}[DISC] \label{def:dense-uniform}
A weighted bipartite graph $G \colon X \x Y \to [0,1]$ is said to satisfy the \emph{discrepancy} condition $\DISC(q,\epsilon)$ if
  \begin{equation} \label{eq:disc-dense}
    \abs{\int_{\substack{x \in X
        \\ y \in Y}} (G(x,y) - q) u(x) v(y) \ dxdy } \leq \epsilon
  \end{equation}
  for all functions $u \colon X \to [0,1]$ and $v \colon Y \to [0,1]$. In any weighted graph $G$, if $X$ and $Y$ are subsets of vertices of $G$, we say that the pair $(X, Y)_G$ satisfies $\DISC(q,  \epsilon)$ if the induced weighted graph on $X \x Y$ satisfies $\DISC(q,\epsilon)$.
\end{definition}

The usual definition for discrepancy of an (unweighted) bipartite graph $G$ is that for all $X' \subseteq X$, $Y' \subseteq Y$, we have
$\abs{e(X',Y') - q \abs{X'}\abs{Y'}} \leq \epsilon\abs{X}\abs{Y}$. It is not hard to see that the two notions of discrepancy are equivalent (with the same $\epsilon$).

A partition $V(G) = V_1 \cup \dots \cup V_k$ is said to be {\it equitable} if all pieces in the partition are of comparable size, that is, if $||V_i| - |V_j|| \leq 1$ for all $i$ and $j$. Szemer\'edi's regularity lemma now says the following.

\begin{thm} [Szemer\'edi's regularity lemma] \label{RegLemma}
For every $\epsilon > 0$ and every positive integer $m_0$, there exists a positive integer $M$ such that any weighted graph $G$ has an equitable partition into $k$ pieces with $m_0 \leq k \leq M$ such that all but at most $\e k^2$ pairs of vertex subsets $(V_i, V_j)$ satisfy $\DISC(q_{ij}, \epsilon)$ for some $q_{ij}$.
\end{thm}

On its own, the regularity lemma would be an interesting result. But what really makes it so powerful is the fact that the discrepancy condition allows us to count small subgraphs. In particular, we have the following result, known as a counting lemma.

\begin{proposition} [Counting lemma in dense graphs]
  \label{prop:dense-counting}
  Let $G$ be a weighted $m$-partite graph with vertex subsets $X_1, X_2,
  \dots, X_m$. Let $H$ be a graph with vertex set $\set{1, \dots, m}$
  and with $e(H)$ edges. For each edge $(i,j)$ in $H$, assume that the induced bipartite graph  $G(X_i,X_j)$ satisfies $\DISC(q_{ij}, \epsilon)$. Define
  \[
  G(H) := \int_{x_1 \in X_1, \dots, x_m \in X_m} \prod_{(i,j) \in E(H)} G(x_i,x_j) \ dx_1 \cdots dx_m
  \]
  and
  \[
  q(H) := \prod_{(i,j) \in E(H) } q_{ij}.
  \]
  Then
  \[
  \abs{G(H) - q(H)} \leq e(H) \epsilon.
  \]
\end{proposition}

The above result, for an unweighted graph, is usually stated in the
following equivalent way: the number of embeddings of $H$ into $G$,
where the vertex $i \in V(H)$ lands in $X_i$, differs from
$\abs{X_1}\abs{X_2}\cdots \abs{X_m} \prod_{(i,j) \in E(H) } q_{ij}$ by
at most $e(H) \epsilon \abs{X_1}\abs{X_2}\cdots \abs{X_m}$. Our
notation $G(H)$ can be viewed as the probability that a random
embedding of vertices of $H$ into their corresponding parts in $G$
gives rise to a valid embedding as a subgraph.

Proposition~\ref{prop:dense-counting} may be proven by telescoping (see, e.g., Theorem 2.7 in \cite{BCLSV08}). Consider, for example, the case where $H$ is a triangle. Then
 $G(x_1,x_2)G(x_1,x_3)G(x_2,x_3) - q_{12}q_{13}q_{23}$ may be rewritten as
\begin{equation} \label{eq:telescope3}
(G(x_1,x_2)- q_{12})G(x_1,x_3)G(x_2,x_3) +  q_{12}(G(x_1,x_3)-q_{13})G(x_2,x_3) + q_{12}q_{13} (G(x_2,x_3) - q_{23}).
\end{equation}
Applying the discrepancy condition~(\ref{eq:disc-dense}), we see that, after integrating the above expression over all $x_1 \in X_1, x_2 \in X_2, x_3 \in X_3$, each term in (\ref{eq:telescope3}) is at most $\epsilon$ in absolute value. The result follows for triangles. The general case follows similarly.

In order to prove extremal results in sparse graphs, we would like to transfer some of this machinery to the sparse setting. Because the number of copies of a subgraph in a sparse graph $G$ is small, the error between the expected count and the actual count must also be small for a counting lemma to be meaningful. Another way to put this is that we aim to achieve a small multiplicative error in our count.

Since we require smaller errors when counting in sparse graphs, we need stronger discrepancy hypotheses. In the following definition, we should view $p$ as the order of magnitude density of the graph, so that the error terms should be bounded in the same order of magnitude. In a dense graph, $p = 1$. We assume that $q \leq p$. It may be helpful to think of $q/p$ as bounded below by some positive constant, although this is not strictly required.

\begin{definition}[DISC]
 \label{def:disc-sparse}
 A weighted bipartite graph $G \colon X \x Y \to [0,1]$ is said to satisfy $\DISC(q,p, \epsilon)$ if
  \[\abs{\int_{\substack{x \in X \\ y \in Y}} (G(x,y) - q)u(x)v(y) \ dxdy} \leq \epsilon p\]
  for all functions $u \colon X \to [0,1]$ and $v \colon Y \to [0,1]$.
\end{definition}

Unfortunately, discrepancy alone is not strong enough for counting in
sparse graphs. Consider the following example. Let $G$ be a tripartite
graph with vertex sets $X_1, X_2, X_3$, such that $(X_1, X_2)_G$ and
$(X_2,X_3)_G$ satisfy $\DISC(q,p,\frac{\epsilon}{2})$. Let $X'_2$ be a
subset of $X_2$ with size $\frac{\epsilon}{2}p\abs{X_2}$. Let $G'$ be
modified from $G$ by adding the complete bipartite graph between $X_1$
and $X'_2$, as well as the complete bipartite graph between $X'_2$ and
$X_3$. The resulting pairs $(X_1, X_2)_{G'}$ and $(X_2, X_3)_{G'}$
satisfy $\DISC(q,p,\epsilon)$. Consider the number of paths in $G$ and
$G'$ with one vertex from each of $X_1, X_2, X_3$ in turn. Given the
densities, we expect there to be approximately
$q^2\abs{X_1}\abs{X_2}\abs{X_3}$ such paths, and we would like the
error to be $\delta p^2 \abs{X_1}\abs{X_2}\abs{X_3}$ for some small
$\delta$ that goes to zero as $\epsilon$ goes to zero. However, the
number of paths in $G'$ from $X_1$ to $X'_2$ to $X_3$ is
$\frac{\epsilon}{2} p\abs{X_1}\abs{X_2}\abs{X_3}$, which is already
too large when $p$ is small.

For our counting lemma to work, $G$ needs to be a relatively dense subgraph of a much more pseudorandom host graph $\Gamma$. In the dense case, $\Gamma$ can be the complete graph. In the sparse world, we require $\Gamma$ to satisfy the jumbledness condition. In practice, we will use the following equivalent definition. The equivalence follows by considering random subsets of $X$ and $Y$, where $x$ and $y$ are chosen with probabilities $u(x)$ and $v(y)$, respectively.

\begin{definition}[Jumbledness]
  \label{def:jumbled}
 A bipartite graph $\Gamma = (X \cup Y, E_\Gamma)$ is $(p,\gamma \sqrt{\abs{X}\abs{Y}})$-jumbled if
\begin{equation} \label{eq:jumbled}
\abs{\int_{\substack{x \in X \\ y \in Y}} (\Gamma(x,y) - p)u(x)v(y) \ dxdy} \leq \gamma\sqrt{\int_{x \in X} u(x) \ dx} \sqrt{ \int_{y \in Y} v(y) \ dy}
\end{equation}
for all functions $u \colon X \to [0,1]$ and $v \colon Y \to [0,1]$.
 \end{definition}

With the discrepancy condition defined as in Definition \ref{def:disc-sparse}, we may now state a regularity lemma for sparse graphs. Such a lemma was originally proved independently by Kohayakawa \cite{K97} and by R\"odl (see also \cite{GS05, Sc11}). The following result, tailored specifically to jumbled graphs, follows easily from the main result in \cite{K97}.

\begin{theorem}[Regularity lemma in jumbled graphs] \label{thm:sparsereg}
For every $\epsilon > 0$ and every positive integer $m_0$, there exists $\eta > 0$ and a positive integer $M$ such that if $\Gamma$ is a $(p, \eta p n)$-jumbled graph on $n$ vertices any weighted subgraph $G$ of $\Gamma$ has an equitable partition into $k$ pieces with $m_0 \leq k  \leq M$ such that all but at most $\e k^2$ pairs of vertex subsets $(V_i, V_j)$ satisfy $\DISC(q_{ij}, p, \epsilon)$ for some $q_{ij}$.
\end{theorem}

The main result of this paper is a counting lemma which complements
this regularity lemma. Proving such an embedding lemma has remained an important open problem
ever since Kohayakawa and R\"odl first proved the sparse regularity
lemma. Most of the work has focused on applying the sparse regularity
lemma in the context of random graphs. The key conjecture in this
case, known as the K\L R conjecture, concerns the probability
threshold down to which a random graph is, with high probability, such
that any regular subgraph contains a copy of a particular subgraph
$H$. This conjecture has only been resolved very recently \cite{BMS12,
  CGSS12}. For pseudorandom graphs, it has been a wide open problem to
prove a counting lemma which complements the sparse regularity
lemma. The first progress on proving such a counting lemma was made
recently in \cite{KRSS10}, where they proved a counting lemma for
triangles. Here, we prove a counting lemma which works for any graph
$H$. Even for triangles, our counting lemma gives an improvement over
the results in \cite{KRSS10}, since our results have polynomial-type
dependence on the discrepancy parameters, whereas the results in
\cite{KRSS10} require exponential dependence since a weak regularity
lemma was used as an immediate step during their proof of the triangle
counting lemma.

Our results are also related to the work of Green and Tao \cite{GT08}
on arithmetic progressions in the primes. What they really prove is
the stronger result that Szemer\'edi's theorem on arithmetic
progressions holds in subsets of the primes. In order to do this, they
first show that the primes, or rather the almost primes, are a
pseudorandom subset of the integers and then that Szemer\'edi's
theorem continues to hold relative to such pseudorandom sets. In the
language of their paper, our counting lemma is a generalized von
Neumann theorem.

Here is the statement of our first counting lemma. Note that, given a graph
$H$, the line graph $L(H)$ is the graph whose vertices are the edges
of $H$ and where two vertices are adjacent if their corresponding
edges in $H$ share an endpoint. Recall that $d(\cdot)$ is the
degeneracy and $\Delta(\cdot)$ is the maximum degree.

\begin{theorem} \label{thm:sparse-counting} Let $H$ be a graph with
  vertex set $\set{1, \dots, m}$ and with $e(H)$ edges. For
  every $\theta > 0$, there exist $c, \epsilon > 0$ of size at least
  polynomial in $\theta$ so that the following holds.

  Let $p > 0$ and let $\Gamma$ be a graph
  with vertex subsets $X_1, \dots, X_m$ and suppose that the bipartite
  graph $(X_i, X_j)_\Gamma$ is $(p, cp^k
  \sqrt{\abs{X_i}\abs{X_j}})$-jumbled for every $i < j$, where $k \geq
  \min\set{\frac{\Delta(L(H)) + 4}{2}, \frac{d(L(H)) + 6}{2}}$. Let $G$ be a
  subgraph of $\Gamma$, with the vertex $i$ of $H$ assigned to the vertex subset $X_i$
  of $G$. For each edge $ij$ in $H$, assume that $(X_i, X_j)_G$
  satisfies $\DISC(q_{ij}, p, \epsilon)$.
  Define
  \[
  G(H) := \int_{x_1 \in X_1, \dots, x_m \in X_m} \prod_{(i,j) \in E(H)} G(x_i,x_j) \ dx_1 \cdots dx_m
  \]
  and
  \[
  q(H) := \prod_{(i,j) \in E(H) } q_{ij}.
  \]
  Then
  \[
  \abs{G(H) - q(H)} \leq \theta p^{e(H)}.
  \]
\end{theorem}

\begin{table}[bt!]
  \centering
  \caption{Sufficient conditions on $k$ in the jumbledness hypothesis $(p, cp^k \sqrt{\abs{X_i}\abs{X_j}})$
    for the counting lemmas of various graphs. Two-sided counting
    refers to results of the form $\abs{G(H) - q(H)} \leq \theta
    p^{e(H)}$ while one-sided counting refers to result of the form
    $G(H) \geq q(H) - \theta p^{e(H)}$.}
  \label{tab:k}
  \begin{tabular}{lccl}
    \toprule
    &  Two-sided counting & One-sided counting  &\\
    $H$ & $k \geq $ & $k \geq $ \\
    \midrule
    $K_t$ & $t$ & $t$ & $t \geq 3$ \\
    $C_\ell$ & $2$  & $2$  & $\ell = 4$ \\
            & $2$  & $ 1 + \frac{1}{2\floor{(\ell-3)/2}}$ & $\ell \geq
            5$ \\
    $K_{2,t}$ & $\frac{t+2}{2}$ & $\frac{5}{2}$ & $t \geq 3$\\
    $K_{s,t}$ & $\frac{s+t+1}{2}$ & $\frac{s+3}{2}$ & $3 \leq s \leq
    t$ \\
    Tree  &  $\frac{\Delta(H) + 1}{2}$ & {No jumbledness
      needed} & See Prop.~\ref{prop:tree} and \ref{prop:tree-oneside}\\
    $K_{1,2,2}$ & $4$ & $4$  \\
    \bottomrule
  \end{tabular}
\end{table}

For some graphs $H$, our methods allow us to achieve slightly better
values of $k$ in Theorem \ref{thm:sparse-counting}. However, the value
given in the theorem is the cleanest general statement. See
Table~\ref{tab:k} for some example of hypotheses on $k$ for various
graphs $H$. To see that the value of $k$ is never far from best
possible, we first note that $\Delta(H) - 1 \leq d(L(H)) \leq
\Delta(H) + d(H) - 2$.

Let $H$ have maximum degree $\Delta$. By considering the random graph
$G_{n,p}$ with $p = n^{-1/\Delta}$, we can find a $(p, c p^{\Delta/2}
n)$-jumbled graph $\Gamma$ containing approximately $p^{e(H)}
n^{v(H)}$ labeled copies of $H$. We modify $\Gamma$ to form $\Gamma'$
by fixing one vertex $v$ and connecting it to everything else. It is
easy to check that the resulting graph $\Gamma'$ is $(p, c'
p^{\Delta/2} n)$-jumbled. However, the number of copies of $H$
disagrees with the expected count, since there are approximately $p^{e(H)}n^{v(H)}$
labeled copies from the original graph $\Gamma$ and a further
approximately $p^{e(H)-\Delta} n^{v(H)-1} = p^{e(H)}n^{v(H)}$ labeled
copies containing $v$. We conclude that for $k < \Delta/2$ we
cannot hope to have such a counting lemma and, therefore, the value of $k$
in Theorem \ref{thm:sparse-counting} is close to optimal.

Since we are dealing with sparse graphs, the discrepancy condition \ref{def:disc-sparse} appears, at first sight, to be rather weak. Suppose, for instance, that we have a sparse graph satisfying $\DISC(q, p, \epsilon)$ between each pair of sets from $V_1, V_2$ and $V_3$ and we wish to embed a triangle between the three sets. Then, a typical vertex $v$ in $V_1$ will have neighborhoods of size roughly $q|V_2|$ and $q|V_3|$ in $V_2$ and $V_3$, respectively. But now the condition $\DISC(q, p, \epsilon)$ tells us nothing about the density of edges between the two neighborhoods. They are simply too small.

To get around this, Gerke, Kohayakawa, R\"odl and Steger \cite{GKRS07} showed that if $(X,Y)$ is a pair satisfying $\DISC(q, p, \epsilon)$ then, with overwhelmingly high probability, a small randomly chosen pair of subsets $X' \subseteq X$ and $Y' \subseteq Y$ will satisfy $\DISC(q, p, \epsilon')$, where $\epsilon'$ tends to zero with $\epsilon$. We say that the pair inherits regularity. This may be applied effectively to prove embedding lemmas in random graphs (see, for example, \cite{GMS07, KRSSz11}). For pseudorandom graphs, the beginnings of such an approach may be found in \cite{KRSS10}.

Our approach in this paper works in the opposite direction. Rather than using the inheritance property to aid us in proving counting lemmas, we first show how one may prove the counting lemma and then use it to prove a strong form of inheritance in jumbled graphs. For example, we have the following theorem.

\begin{proposition}
  \label{prop:inheritintro}
  For any $\alpha > 0$, $\xi > 0$ and $\epsilon' > 0$, there exists $c > 0$ and
  $\epsilon > 0$ of size at least polynomial in $\alpha, \xi, \epsilon'$ such that the following holds.

  Let $p \in (0,1]$ and $q_{XY}, q_{XZ}, q_{YZ} \in [\alpha p,p]$. Let
  $\Gamma$ be a tripartite graph with vertex sets $X$, $Y$ and
  $Z$ and $G$ be a subgraph of $\Gamma$. Suppose that
  \begin{itemize}
  \item $(X,Y)_\Gamma$ is $(p, cp^4\sqrt{\abs{X}\abs{Y}})$-jumbled and
    $(X,Y)_G$ satisfies $\DISC(q_{XY}, p, \epsilon)$; and
  \item $(X,Z)_\Gamma$ is $(p, cp^2\sqrt{\abs{X}\abs{Z}})$-jumbled and
    $(X,Z)_G$ satisfies $\DISC(q_{XZ}, p, \epsilon)$; and
  \item $(Y,Z)_\Gamma$ is $(p, cp^3\sqrt{\abs{Y}\abs{Z}})$-jumbled and
    $(Y,Z)_G$ satisfies $\DISC(q_{YZ}, p, \epsilon)$.
  \end{itemize}
  Then at least $(1 - \xi)\abs{Z}$ vertices $z \in Z$ have the
  property that $\abs{N_X(z)} \geq (1 - \xi)q_{XZ}\abs{X}$,
  $\abs{N_Y(z)} \geq (1 - \xi) q_{YZ}\abs{Y}$, and $(N_X(z),
  N_Y(z))_G$ satisfies $\DISC(q_{XY}, p, \epsilon')$.
\end{proposition}

The question now arises as to why one would prove that the inheritance
property holds if we already know its intended
consequence. Surprisingly, there is another counting lemma, giving only a lower bound on $G(H)$, which is
sufficient to establish the various extremal results but typically requires a much weaker jumbledness assumption. The proof of this statement relies on the inheritance property in a critical way. The notations $G(H)$ and $q(H)$ were defined in Theorem~\ref{thm:sparse-counting}.

\begin{theorem} \label{thm:onesidedintro} For every fixed graph $H$ on
  vertex set $\{1,2, \dots, m\}$ and every $\alpha, \theta > 0$, there
  exist constants $c > 0$ and $\epsilon > 0$ such that the following
  holds.

Let $\Gamma$ be a graph with vertex subsets $X_1, \dots, X_m$ and
suppose that the bipartite graph $(X_i, X_j)_\Gamma$ is $(p, cp^{d_2(H)
  + 3} \sqrt{\abs{X_i}\abs{X_j}})$-jumbled for every $i < j$
with $ij \in E(H)$. Let $G$ be a subgraph of $\Gamma$, with the vertex
$i$ of $H$ assigned to the vertex subset $X_i$ of $G$. For each edge
$ij$ of $H$, assume that $(X_i, X_j)_G$ satisfies $\DISC(q_{ij}, p,
\epsilon)$, where $\alpha p \leq q_{ij} \leq p$. Then
\[G(H) \geq (1 - \theta) q(H).\]
\end{theorem}

We refer to Theorem~\ref{thm:onesidedintro} as a one-sided counting
lemma, as we get a lower bound for $G(H)$ but no upper bound. However, in order to prove the theorems of Section \ref{sec:extremal}, we only need a lower bound. The proof of Theorem \ref{thm:onesidedintro} is a sparse version of a classical embedding strategy in regular graphs (see, for example, \cite{CRST83, FS09, GRR00}). Note that, as in the theorems of Section \ref{sec:extremal}, the exponent $d_2(H) + 3$ can be improved for certain graphs $H$. We will say more about this later. Moreover, one cannot hope to do better than $\beta = O(p^{(d(H) + 2)/4} n)$, so that the condition on $\beta$ is sharp up to a multiplicative constant in the exponent of $p$. We suspect that the exponent may even be sharp up to an additive constant.

{\bf \vspace{3mm}\noindent Organization.} We will begin in the next
section by giving a high level overview of the proof of our counting
lemmas. In Section~\ref{sec:counting-Gamma}, we prove some useful
statements about counting in the pseudorandom graph $\Gamma$. Then, in
Section \ref{sec:counting-G}, we prove the sparse counting lemma,
Theorem \ref{thm:sparse-counting}. The short proof of Proposition
\ref{prop:inheritintro} and some related propositions about
inheritance are given in Section \ref{sec:inherit}. The proof of the
one-sided counting lemma, which uses inheritance, is then given in
Section \ref{sec:oneside}. In Section \ref{sec:cycles}, we take a
closer look at one-sided counting in cycles. The sparse counting lemma
has a large number of applications extending many classical results to
the sparse setting. In Section \ref{sec:applications}, we discuss a
number of them in detail, including sparse extensions of
the Erd\H{o}s-Stone-Simonovits theorem, Ramsey's theorem, the graph removal
lemma and the removal lemma for groups. In Section~\ref{sec:conclude} we
briefly discuss a number of other applications, such as relative
quasirandomness, induced Ramsey numbers, algorithmic applications and
multiplicity results.

\section{Counting strategy} \label{sec:countingstrat}

In this section, we give a general overview of our approach to
counting. There are two types of counting results: two-sided counting
and one-sided counting. Two-sided counting refers to results of the
form $\abs{G(H) - q(H)} \leq \theta p^{e(H)}$ while one-sided counting
refers to results of the form $G(H) \geq q(H) - \theta
p^{e(H)}$. One-sided counting is always implied by two-sided counting,
although sometimes we are able to obtain one-sided counting results
under weaker hypotheses.

\subsection{Two-sided counting}

There are two main ingredients to the proof: \emph{doubling} and
\emph{densification}. These two procedures reduce the
problem of counting embeddings of $H$ to the same problem for some other graphs
$H'$.

If $a \in V(H)$, \emph{the graph $H$ with $a$ doubled}, denoted $H_{a
  \x 2}$, is the graph created from $V(H)$ by adding a new vertex $a'$
whose neighbors are precisely the neighbors of $a$. In the assignment
of vertices of $H_{a \x 2}$ to vertex subsets of $\Gamma$, the new vertex
$a'$ is assigned to the same vertex subset of $\Gamma$. For example, the
following figure shows a triangle with a vertex doubled.
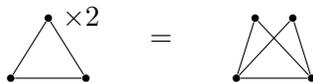
\begin{center}
 \begin{tikzpicture}
    \begin{scope}
      \node[p, label=right:$\x 2$] (a) at (0.5,.8) {};
      \node[p] (b) at (0,0) {};
      \node[p] (c) at (1,0) {};
      \draw (a)--(b)--(c)--(a);
    \end{scope}
    \node at (2,0.5) {$=$};
    \begin{scope}[xshift=3cm]
      \node[p] (a) at (0.25,.8) {};
      \node[p] (a2) at (0.75,.8) {};
      \node[p] (b) at (0,0) {};
      \node[p] (c) at (1,0) {};
      \draw (b)--(a)--(c)--(a2)--(b)--(c);
    \end{scope}
  \end{tikzpicture}
\end{center}

A typical reduction using doubling is summarized in
Figure~\ref{fig:doubling}. Each graph represents the claim that the number of embeddings of the graph
drawn, where the straight edges must land in $G$ and the wavy edges
must land in $\Gamma$, is approximately what one would expect from
multiplying together the appropriate edge densities between the vertex subsets of $G$
and $\Gamma$.

The top arrow in Figure~\ref{fig:doubling} is the doubling step. This
allows us to reduce the problem of counting $H$ to that of counting a
number of other graphs, each of which may have some edges which embed into $G$ and some which embed into $\Gamma$.
For example, if we let $H_{-a}$ be the graph that we get by omitting every edge
which is connected to a particular vertex $a$, we are interested in the number of copies of $H_{-a}$ in both $G$ and $\Gamma$. We are also interested in the original graph $H$, but now on the understanding that the edges incident with $a$ embed into $G$ while those that do not touch $a$ embed into $\Gamma$. Finally, we are interested in the graph $H_{a\x 2}$ formed by doubling the vertex $a$, but again the edges which do not touch $a$ or its copy $a'$ only have to embed into $\Gamma$. This reduction, which is justified by an application of the Cauchy-Schwarz inequality, will be detailed in Section~\ref{sec:doubling}.

The bottom two arrows in Figure~\ref{fig:doubling} are representative of another reduction, where
we can reduce the problem of counting a particular graph, with edges that map to both $G$ and $\Gamma$, into one where we only care about the edges that
embed into $G$. We can make such a reduction because counting
embeddings into $\Gamma$ is much easier due to its
jumbledness. We will discuss this reduction, amongst other properties of jumbled graphs $\Gamma$, in Section~\ref{sec:counting-Gamma}.

\begin{figure}[!htp]
  \centering
    \begin{tikzpicture}[scale=.5]
    \begin{scope}[shift={(0,0)}]
      \node[p, label=right:$a$] (a) at (0,2) {};
      \node[p] (1) at (-1.5,0) {};
      \node[p] (2) at (-.5,0) {};
      \node[p] (3) at (.5,0) {};
      \node[p] (4) at (1.5,0) {};
      \node[p] (5) at (-1,-1) {};
      \node[p] (6) at (0,-1) {};
      \node[p] (7) at (1,-1) {};
      \draw (a)--(1) (a)--(2) (a)--(3) (a)--(4);
      \draw (1)--(5)--(3)--(6)--(7)--(4) (2)--(3);
      \node at (0,-2) {$G(H)$};
      \node at (0,-3.5) {\LargeUparrow};
      \draw[thick] decorate [decoration={brace, amplitude=1cm}] {(-12,-7) -- (12,-7)};
    \end{scope}
    \begin{scope}[shift={(-9,-9)}]
      \node[p] (1) at (-1.5,0) {};
      \node[p] (2) at (-.5,0) {};
      \node[p] (3) at (.5,0) {};
      \node[p] (4) at (1.5,0) {};
      \node[p] (5) at (-1,-1) {};
      \node[p] (6) at (0,-1) {};
      \node[p] (7) at (1,-1) {};
      \draw (1)--(5)--(3)--(6)--(7)--(4) (2)--(3);
      \node at (0,-2) {$G(H_{-a})$};
    \end{scope}
    \begin{scope}[shift={(-3,-9)}]
      \node[p] (a) at (-.5,2) {};
      \node[p] (a2) at (.5,2) {};
      \node[p] (1) at (-1.5,0) {};
      \node[p] (2) at (-.5,0) {};
      \node[p] (3) at (.5,0) {};
      \node[p] (4) at (1.5,0) {};
      \node[p] (5) at (-1,-1) {};
      \node[p] (6) at (0,-1) {};
      \node[p] (7) at (1,-1) {};
      \draw (a)--(1)--(a2)--(2)--(a)--(3)--(a2)--(4)--(a);
      \draw decorate [Gamma] { (1)--(5)--(3)--(6)--(7)--(4) (2)--(3)};
      \node at (0,-2) {$g(H_{a \x 2})$};
      \node at (0,-3.5) {\LargeUparrow};
    \end{scope}
    \begin{scope}[shift={(3,-9)}]
      \node[p] (a) at (0,2) {};
      \node[p] (1) at (-1.5,0) {};
      \node[p] (2) at (-.5,0) {};
      \node[p] (3) at (.5,0) {};
      \node[p] (4) at (1.5,0) {};
      \node[p] (5) at (-1,-1) {};
      \node[p] (6) at (0,-1) {};
      \node[p] (7) at (1,-1) {};
      \draw (a)--(1) (a)--(2) (a)--(3) (a)--(4);
            \node at (0,-2) {$g(H)$};
      \draw decorate [Gamma] { (1)--(5)--(3)--(6)--(7)--(4) (2)--(3)};
            \node at (0,-3.5) {\LargeUparrow};
    \end{scope}
    \begin{scope}[shift={(9,-9)}]
      \node[p] (1) at (-1.5,0) {};
      \node[p] (2) at (-.5,0) {};
      \node[p] (3) at (.5,0) {};
      \node[p] (4) at (1.5,0) {};
      \node[p] (5) at (-1,-1) {};
      \node[p] (6) at (0,-1) {};
      \node[p] (7) at (1,-1) {};
      \draw decorate [Gamma] { (1)--(5)--(3)--(6)--(7)--(4) (2)--(3)};
            \node at (0,-2) {$\Gamma(H_{-a})$};
    \end{scope}
    \begin{scope}[shift={(-3,-16)}]
      \node[p] (a) at (-.5,2) {};
      \node[p] (a2) at (.5,2) {};
      \node[p] (1) at (-1.5,0) {};
      \node[p] (2) at (-.5,0) {};
      \node[p] (3) at (.5,0) {};
      \node[p] (4) at (1.5,0) {};
            \node at (0,-1) {$G(H_{a,a\x2})$};
      \draw (a)--(1)--(a2)--(2)--(a)--(3)--(a2)--(4)--(a);
    \end{scope}
    \begin{scope}[shift={(3,-16)}]
      \node[p] (a) at (0,2) {};
      \node[p] (1) at (-1.5,0) {};
      \node[p] (2) at (-.5,0) {};
      \node[p] (3) at (.5,0) {};
      \node[p] (4) at (1.5,0) {};
            \node at (0,-1) {$G(H_{a})$};
      \draw (a)--(1) (a)--(2) (a)--(3) (a)--(4);
    \end{scope}
  \end{tikzpicture}
  \caption{The doubling reduction. Each graph represents some counting
  lemma. The straight edges must embed into $G$ while wavy edges must
  embed into the jumbled graph $\Gamma$.}
  \label{fig:doubling}
\end{figure}
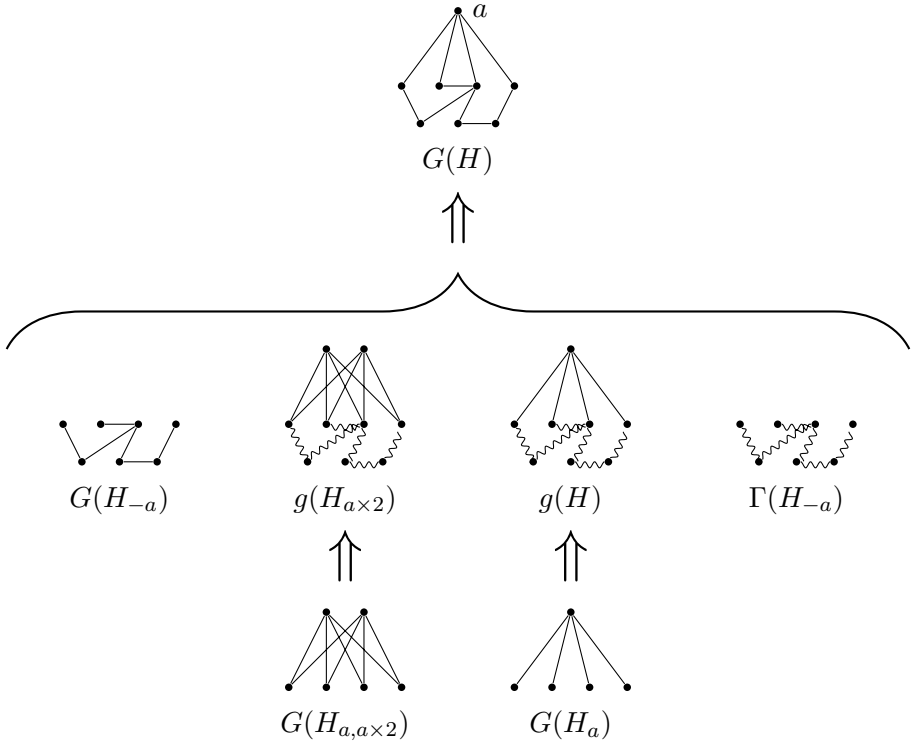

For triangles, a similar reduction is shown in Figure~\ref{fig:triangle-doubling}.
In the end, we have changed the task of counting triangles to the task of
counting the number of cycles of length $4$. It would be natural now to apply doublng
to the $4$-cycle but, unfortunately, this process is circular. Instead, we introduce an
alternative reduction process which we refer to as densification.

\begin{figure}[!htp]
  \centering
  \begin{tikzpicture}[scale=.5]
    \begin{scope}[shift={(0,.5)}]
      \node[p,label=right:$a$] (a) at (0,1.6) {};
      \node[p] (b) at (-1,0) {};
      \node[p] (c) at (1,0) {};
      \draw (b)--(a)--(c);
      \draw (b)--(c);
      \node at (0,-1.5) {\LargeUparrow};
      \draw[thick] decorate [decoration={brace, amplitude=1cm}] {(-10,-4.5) -- (10,-4.5)};
    \end{scope}
    \begin{scope}[shift={(-9,-6)}]
      \node[p] (b) at (-1,0) {};
      \node[p] (c) at (1,0) {};
      \draw (b)--(c);
    \end{scope}
    \begin{scope}[shift={(-3,-6)}]
      \node[p] (a) at (-.5,1.6) {};
      \node[p] (a2) at (.5, 1.6) {};
      \node[p] (b) at (-1,0) {};
      \node[p] (c) at (1,0) {};
      \draw (b)--(a)--(c)--(a2)--(b);
      \draw decorate [Gamma] {(b)--(c)};
      \node at (0,-1.5) {\LargeUparrow};
    \end{scope}
    \begin{scope}[shift={(-3,-11)}]
      \node[p] (a) at (-.5,1.6) {};
      \node[p] (a2) at (.5, 1.6) {};
      \node[p] (b) at (-1,0) {};
      \node[p] (c) at (1,0) {};
      \draw (b)--(a)--(c)--(a2)--(b);
    \end{scope}
    \begin{scope}[shift={(3,-6)}]
      \node[p] (a) at (0,1.6) {};
      \node[p] (b) at (-1,0) {};
      \node[p] (c) at (1,0) {};
      \draw (b)--(a)--(c);
      \draw decorate [Gamma] {(b)--(c)};
      \node at (0,-1.5) {\LargeUparrow};
    \end{scope}
    \begin{scope}[shift={(3,-11)}]
      \node[p] (a) at (0,1.6) {};
      \node[p] (b) at (-1,0) {};
      \node[p] (c) at (1,0) {};
      \draw (b)--(a)--(c);
    \end{scope}
    \begin{scope}[shift={(9,-6)}]
      \node[p] (b) at (-1,0) {};
      \node[p] (c) at (1,0) {};
      \draw decorate [Gamma] {(b)--(c)};
    \end{scope}
  \end{tikzpicture}
  \caption{The doubling reduction for counting triangles.}
  \label{fig:triangle-doubling}
\end{figure}
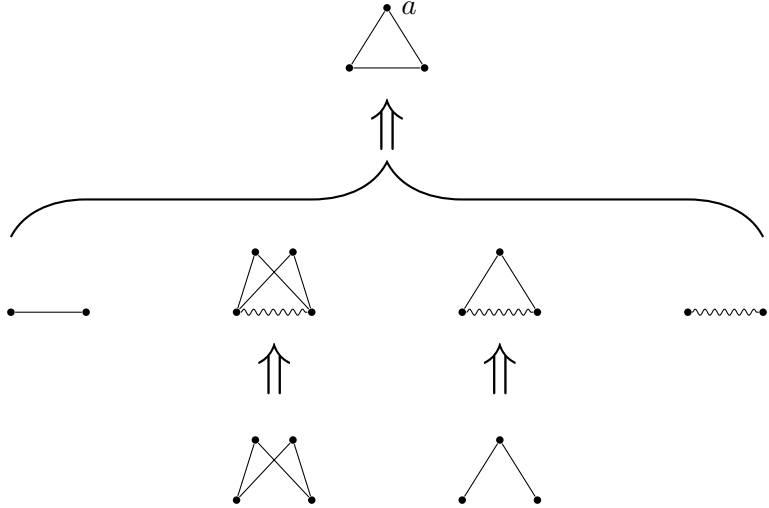

In the above reduction from triangles to $4$-cycles, two of the vertices of the 4-cycle are
embedded into the same part $X_i$ of $G$. We actually consider
the more general setting where the vertices of the 4-cycle lie in
different parts, $X_1, X_2, X_3, X_4$, of $G$.

Assume without loss of generality that there is no edge between $X_1$
and $X_3$ in $G$. Let us add a weighted graph between $X_1$ and
$X_3$, where the weight on the edge $x_1x_3$ is proportional to
the number of paths $x_1x_4x_3$ for $x_4 \in X_4$. Since $(X_1, X_4)_G$
and $(X_3,X_4)_G$ satisfy discrepancy, the number of paths will be on
the order of $q_{14}q_{34} \abs{X_4}$ for most pairs $(x_1,
x_3)$. After discarding a negligible set of pairs $(x_1, x_3)$ that
give too many paths, and then appropriately rescaling the weights of the
other edges $x_1x_3$, we create a weighted bipartite graph between $X_1$
and $X_3$ that behaves like a \emph{dense} weighted graph satisfying
discrepancy. Furthermore, counting 4-cycles in $X_1, X_2, X_3, X_4$ is
equivalent to counting triangles in $X_1, X_2, X_3$ due to the choice
of weights. We call this process densification. It is
illustrated below. In the figure, a thick edge
signifies that the bipartite graph that it embeds into is dense.

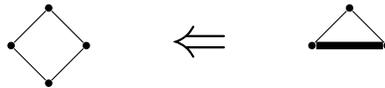
\begin{center}
  \begin{tikzpicture}[scale=.5]
    \begin{scope}
      \node[p] (1) at (1,0) {};
      \node[p] (2) at (0,1) {};
      \node[p] (3) at (-1,0) {};
      \node[p] (4) at (0,-1) {};
      \draw (1)--(2)--(3)--(4)--(1);
    \end{scope}

    \node at (4,0) {\LargeLeftarrow};

    \begin{scope}[shift={(8,0)}]
      \node[p] (1) at (1,0) {};
      \node[p] (2) at (0,1) {};
      \node[p] (3) at (-1,0) {};
      \draw (1)--(2)--(3);
      \draw [dense] (1)--(3);
    \end{scope}
  \end{tikzpicture}
\end{center}

More generally, if $b$ is a vertex of $H$ of degree 2, with neighbors
$\set{a,c}$, such that $a$ and $c$ are not adjacent, then
densification allows us to transform $H$ by removing the edges $ab$
and $bc$ and adding a dense edge $ac$, as illustrated below. For
more on this process, we refer the reader to Section \ref{sec:densification}.

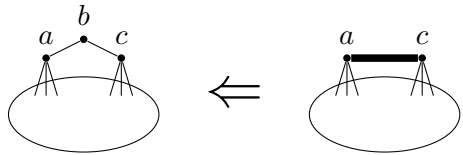
\begin{center}
  \begin{tikzpicture}[scale=.5]
    \begin{scope}
      \node[p,label=above:$a$] (a) at (-1,0) {};
      \node[p,label=above:$b$] (b) at (0,.5) {};
      \node[p,label=above:$c$] (c) at (1,0) {};
      \draw (a)--(b)--(c);
      \draw (a) -- +(-.3,-1) (a) -- +(0,-1)  (a) -- +(.3,-1) ;
      \draw (c) -- +(-.3,-1) (c) -- +(0,-1) (c) -- +(.3,-1) ;
      \draw (0,-1.5) ellipse (2cm and 1cm);
    \end{scope}
    \node at (4,-1) {\LargeLeftarrow};
    \begin{scope}[shift={(8,0)}]
      \node[p,label=above:$a$] (a) at (-1,0) {};
      \node[p,label=above:$c$] (c) at (1,0) {};
      \draw[dense] (a)--(c);
      \draw (a) -- +(-.3,-1) (a) -- +(0,-1)  (a) -- +(.3,-1) ;
      \draw (c) -- +(-.3,-1) (c) -- +(0,-1) (c) -- +(.3,-1) ;
      \draw (0,-1.5) ellipse (2cm and 1cm);
    \end{scope}
  \end{tikzpicture}
\end{center}

We needed to count 4-cycles in order to count triangles, so it seems
at first as if our reduction from $4$-cycles to triangles is circular. However, instead of counting triangles in a sparse graph, we now have a
dense bipartite graph between one of the pairs of vertex subsets. Since it is easier to
count in dense graphs than in sparse graphs, we have made
progress. The next step is to do doubling again. This is shown in
Figure~\ref{fig:triangle-one-dense-edge}. The bottommost arrow is
another application of densification.

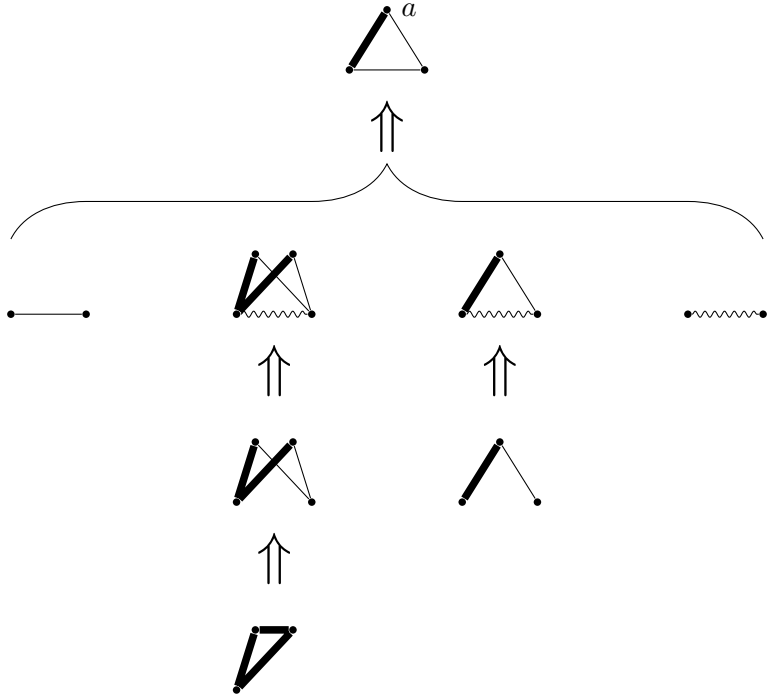
\begin{figure}[!htp]
  \centering
  \begin{tikzpicture}[scale=.5]
    \begin{scope}[shift={(0,.5)}]
      \node[p, label=right:$a$] (a) at (0,1.6) {};
      \node[p] (b) at (-1,0) {};
      \node[p] (c) at (1,0) {};
      \draw[dense] (a)--(b);
      \draw (a)--(c);
      \draw (b)--(c);
      \node at (0,-1.5) {\LargeUparrow};
      \draw decorate [decoration={brace, amplitude=1cm}] {(-10,-4.5) -- (10,-4.5)};
    \end{scope}
    \begin{scope}[shift={(-9,-6)}]
      \node[p] (b) at (-1,0) {};
      \node[p] (c) at (1,0) {};
      \draw (b)--(c);
    \end{scope}
    \begin{scope}[shift={(-3,-6)}]
      \node[p] (a) at (-.5,1.6) {};
      \node[p] (a2) at (.5, 1.6) {};
      \node[p] (b) at (-1,0) {};
      \node[p] (c) at (1,0) {};
      \draw [dense] (a)--(b)--(a2);
      \draw (a)--(c)--(a2);
      \draw decorate [Gamma] {(b)--(c)};
      \node at (0,-1.5) {\LargeUparrow};
    \end{scope}
    \begin{scope}[shift={(-3,-11)}]
      \node[p] (a) at (-.5,1.6) {};
      \node[p] (a2) at (.5, 1.6) {};
      \node[p] (b) at (-1,0) {};
      \node[p] (c) at (1,0) {};
      \draw [dense] (a)--(b)--(a2);
      \draw (a)--(c)--(a2);
      \node at (0,-1.5) {\LargeUparrow};
    \end{scope}
    \begin{scope}[shift={(-3,-16)}]
      \node[p] (a) at (-.5,1.6) {};
      \node[p] (a2) at (.5, 1.6) {};
      \node[p] (b) at (-1,0) {};
      \draw [dense] (a)--(b)--(a2)--(a);
    \end{scope}

    \begin{scope}[shift={(3,-6)}]
      \node[p] (a) at (0,1.6) {};
      \node[p] (b) at (-1,0) {};
      \node[p] (c) at (1,0) {};
      \draw [dense] (a)--(b);
      \draw (a)--(c);
      \draw decorate [Gamma] {(b)--(c)};

      \node at (0,-1.5) {\LargeUparrow};
    \end{scope}
    \begin{scope}[shift={(3,-11)}]
      \node[p] (a) at (0,1.6) {};
      \node[p] (b) at (-1,0) {};
      \node[p] (c) at (1,0) {};
      \draw [dense] (a)--(b);
      \draw (a)--(c);
    \end{scope}
    \begin{scope}[shift={(9,-6)}]
      \node[p] (b) at (-1,0) {};
      \node[p] (c) at (1,0) {};
      \draw decorate [Gamma] {(b)--(c)};
    \end{scope}
  \end{tikzpicture}
  \caption{The doubling reduction for triangles with one dense edge.}
  \label{fig:triangle-one-dense-edge}
\end{figure}

We have therefore reduced the problem of counting triangles in a sparse graph
to that of counting triangles in a dense weighted graph, which we already know
how to do. This completes the counting lemma for 4-cycles.

In Figure~\ref{fig:doubling}, doubling reduces counting in a general $H$ to
counting $H$ with one vertex deleted (which we handle by induction) as
well as graphs of the form $K_{1,t}$ and $K_{2,t}$. Trees like $K_{1,t}$ are not too
hard to count. It therefore remains to count $K_{2,t}$. As with counting $C_4$
(the case $t=2$), we first perform a densification.
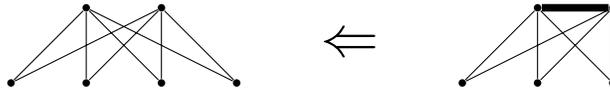
\begin{center}
  \begin{tikzpicture}
    \begin{scope}
      \node[p] (a) at (0,1) {};
      \node[p] (b) at (1,1) {};
      \node[p] (0) at (-1,0) {};
      \node[p] (1) at (0,0) {};
      \node[p] (2) at (1,0) {};
      \node[p] (3) at (2,0) {};
      \draw (b)--(0)--(a)--(1)--(b)--(2)--(a)--(3)--(b);
      \node at (3.5,0.5) {\LargeLeftarrow};
    \end{scope}
    \begin{scope}[shift={(6,0)}]
      \node[p] (a) at (0,1) {};
      \node[p] (b) at (1,1) {};
      \node[p] (0) at (-1,0) {};
      \node[p] (1) at (0,0) {};
      \node[p] (2) at (1,0) {};
      \draw[dense] (a)--(b);
      \draw (b)--(0)--(a)--(1)--(b)--(2)--(a);
    \end{scope}
  \end{tikzpicture}
\end{center}
The graph on the right can be counted using doubling and induction, as
shown in Figure~\ref{fig:K2t-doubling}. Note that the $C_4$ count is
required as an input to this step. This then completes the
proof of the counting lemma.

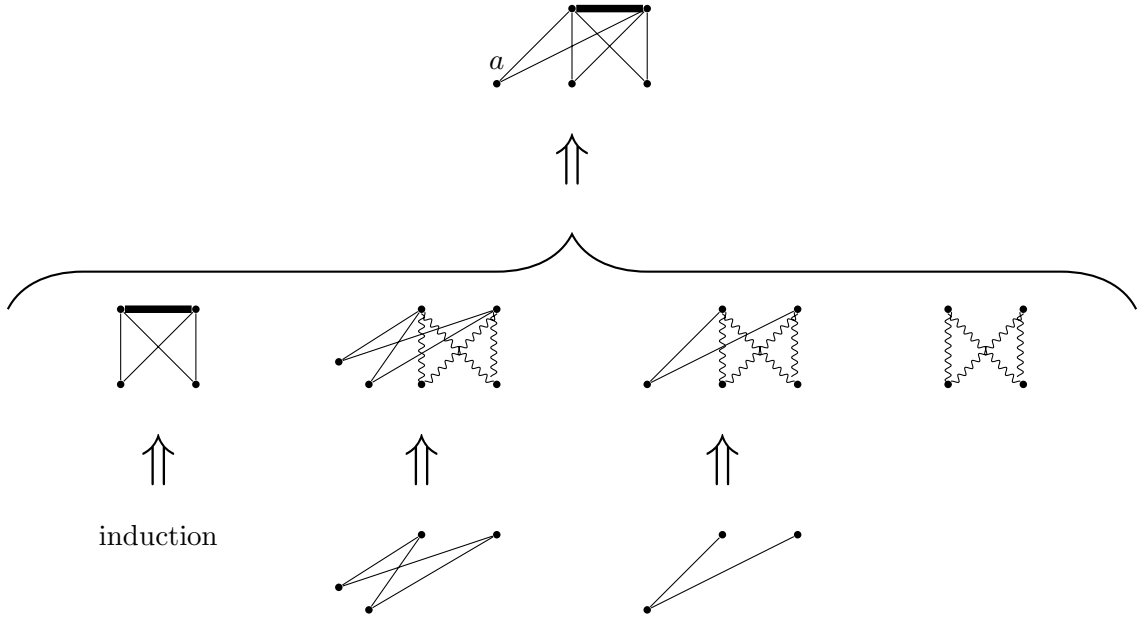
\begin{figure}[!htp]
  \centering
  \begin{tikzpicture}
    \begin{scope}[shift={(0,1)}]
        \node[p] (a) at (0,1) {};
        \node[p] (b) at (1,1) {};
        \node[p,label=above:$a$] (0) at (-1,0) {};
        \node[p] (1) at (0,0) {};
        \node[p] (2) at (1,0) {};
        \draw[dense] (a)--(b);
        \draw (b)--(0)--(a)--(1)--(b)--(2)--(a);
      \node at (0,-1) {\LargeUparrow};
      \draw[thick] decorate [decoration={brace, amplitude=1cm}]
      {(-7.5,-3) -- (7.5,-3)};
    \end{scope}
    \begin{scope}[shift={(-6,-3)}]
        \node[p] (a) at (0,1) {};
        \node[p] (b) at (1,1) {};
        \node[p] (1) at (0,0) {};
        \node[p] (2) at (1,0) {};
        \draw[dense] (a)--(b);
        \draw (a)--(1)--(b)--(2)--(a);
        \node at (.5,-1) {\LargeUparrow};
        \node at (.5, -2) {induction};
    \end{scope}
    \begin{scope}[shift={(-2,-3)}]
        \node[p] (a) at (0,1) {};
        \node[p] (b) at (1,1) {};
        \node[p] (0) at (-.7,0) {};
        \node[p] (0') at (-1.1, 0.3) {};
        \node[p] (1) at (0,0) {};
        \node[p] (2) at (1,0) {};
        \draw (a)--(0')--(b)--(0)--(a);
        \draw decorate[Gamma] {(a)--(1)--(b)--(2)--(a)};
      \node at (0,-1) {\LargeUparrow};
    \end{scope}
    \begin{scope}[shift={(-2,-6)}]
        \node[p] (a) at (0,1) {};
        \node[p] (b) at (1,1) {};
        \node[p] (0) at (-.7,0) {};
        \node[p] (0') at (-1.1, 0.3) {};
        \draw (a)--(0')--(b)--(0)--(a);
    \end{scope}
    \begin{scope}[shift={(2,-3)}]
        \node[p] (a) at (0,1) {};
        \node[p] (b) at (1,1) {};
        \node[p] (0) at (-1,0) {};
        \node[p] (1) at (0,0) {};
        \node[p] (2) at (1,0) {};
        \draw (b)--(0)--(a);
        \draw decorate[Gamma] {(a)--(1)--(b)--(2)--(a)};
      \node at (0,-1) {\LargeUparrow};
    \end{scope}
    \begin{scope}[shift={(2,-6)}]
        \node[p] (a) at (0,1) {};
        \node[p] (b) at (1,1) {};
        \node[p] (0) at (-1,0) {};
        \draw (b)--(0)--(a);
    \end{scope}
    \begin{scope}[shift={(5,-3)}]
        \node[p] (a) at (0,1) {};
        \node[p] (b) at (1,1) {};
        \node[p] (1) at (0,0) {};
        \node[p] (2) at (1,0) {};
        \draw decorate[Gamma] {(a)--(1)--(b)--(2)--(a)};
    \end{scope}
  \end{tikzpicture}
  \caption{The doubling reduction for counting $K_{2,t}$.}
  \label{fig:K2t-doubling}
\end{figure}

\subsection{One-sided counting}

For one-sided counting, we embed the vertices of $H$ into those of $G$
one at a time. By making a choice for where a vertex $a$ of $H$ lands
in $G$, we shrink the set of possible targets for each neighbor of
$a$.  These target sets shrink by a factor roughly corresponding to
the edge densities of $G$, as most vertices of $G$ have close to the
expected number of neighbors due to discrepancy. This allows us to
obtain a lower bound on the number of embeddings of $H$ into $G$.

The above argument is missing one important ingredient. When we shrink
the set of possible targets of vertices in $H$, we do not know if $G$
restricted to these smaller vertex subsets still satisfies
the discrepancy condition, which is needed for embedding later vertices. When $G$ is dense, this is not an issue, since the
restricted vertex subsets have size at least a constant factor of the
original vertex subsets, and thus discrepancy is inherited. When $G$
is sparse, the restricted vertex subsets can become much smaller than the
original vertex subsets, so discrepancy is not automatically
inherited.

To address this issue, we observe that discrepancy between two vertex
sets follows from some variant of the $K_{2,2}$ count (and the
counting lemma shows that they are in fact equivalent). By our
counting lemma, we also know that the graph below has roughly the expected
count. This in turn implies that discrepancy is inherited in the
neighborhoods of $G$ since, roughly speaking, it implies that almost
every vertex has roughly the expected number of $4$-cycles in its
neighborhood. The one-sided counting approach sketched above then
carries through. For further details on inheritance of discrepancy,
see Section~\ref{sec:inherit}. The proof of the one-sided counting
lemma may be found in Section~\ref{sec:oneside}.
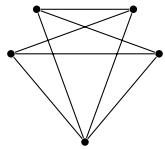
\[
  \begin{tikzpicture}
    \node[p] (x) at (-90:1) {};
    \node[p] (y) at (130:1) {};
    \node[p] (y1) at (170:1) {};
    \node[p] (z) at (50:1) {};
    \node[p] (z1) at (10:1) {};
    \draw (x)--(y)--(z)--(x)--(y1)--(z1)--(x)
    (y1)--(z) (y)--(z1);
  \end{tikzpicture}
\]

We also prove a one-sided counting lemma for large cycles using much weaker
jumbledness hypotheses. The idea is to extend densification to more
than two edges at a time. We will show how to transform a multiply subdivided
edge into a single dense edge, as illustrated below.
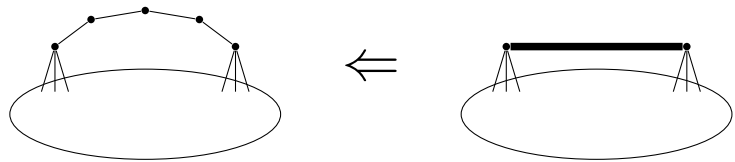
\[
  \begin{tikzpicture}[scale=.6]
    \begin{scope}
      \node[p] (a) at (-2,0) {};
      \node[p] (b1) at (-1.2,.6) {};
      \node[p] (b2) at (0,.8) {};
      \node[p] (b3) at (1.2,.6) {};
      \node[p] (c) at (2,0) {};
      \draw (a)--(b1)--(b2)--(b3)--(c);
      \draw (a) -- +(-.3,-1) (a) -- +(0,-1)  (a) -- +(.3,-1) ;
      \draw (c) -- +(-.3,-1) (c) -- +(0,-1) (c) -- +(.3,-1) ;
      \draw (0,-1.5) ellipse (3cm and 1cm);
    \end{scope}

    \node at (5,-.5) {\LargeLeftarrow};

    \begin{scope}[shift={(10,0)}]
      \node[p] (a) at (-2,0) {};
      \node[p] (c) at (2,0) {};
      \draw[dense] (a)--(c);
      \draw (a) -- +(-.3,-1) (a) -- +(0,-1)  (a) -- +(.3,-1) ;
      \draw (c) -- +(-.3,-1) (c) -- +(0,-1) (c) -- +(.3,-1) ;
      \draw (0,-1.5) ellipse (3cm and 1cm);
    \end{scope}
  \end{tikzpicture}
\]
Starting with a long cycle, we can perform two such densifications, as
shown below. The resulting triangle is easy to count, since a typical
embedding of the top vertex gives a linear-sized neighborhood. The full
details may be found in Section \ref{sec:cycles}.
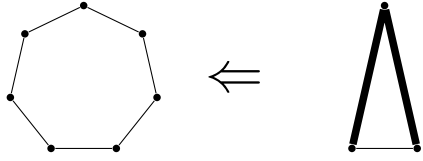
\[
  \begin{tikzpicture}
    \path[use as bounding box] (-2,-1) rectangle (6,1);
    \begin{scope}
      \foreach \i in {0,1,2,3,4,5,6}{
        \node[p] (\i) at (360/7 *\i + 90:1) {};
      }
      \draw (0)--(1)--(2)--(3)--(4)--(5)--(6)--(0);
    \end{scope}

    \node at (2,0) {\LargeLeftarrow};

    \begin{scope}[shift={(4,0)}]
      \foreach \i in {0,3,4}{
        \node[p] (\i) at (360/7 *\i + 90:1) {};
      }
      \draw[dense] (3)--(0)--(4);
      \draw (3)--(4);
    \end{scope}
  \end{tikzpicture}
\]

\section{Counting in \texorpdfstring{$\Gamma$}{Gamma}} \label{sec:counting-Gamma}

In this section, we develop some tools for counting in $\Gamma$.
Here is the setup for this section.

\begin{setup} \label{set:Gamma}
Let $\Gamma$ be a graph with
vertex subsets $X_1, \dots, X_m$. Let $p, c \in (0,1]$ and $k
\geq 1$. Let $H$ be a graph with vertex set
$\set{1, \dots, m}$, with vertex $a$ assigned to $X_a$. For every edge
$ab$ in $H$, one of the
following two holds:
\begin{itemize}
\item $(X_a, X_b)_\Gamma$ is $(p, cp^k
  \sqrt{\abs{X_a}\abs{X_b}})$-jumbled, in which case we set $p_{ab} =
  p$ and  say that $ab$
  is a \emph{sparse edge}, or
\item $(X_a, X_b)_\Gamma$ is a complete bipartite graph, in which case
  we set $p_{ab} = 1$ and say that $ab$ is a \emph{dense edge}.
\end{itemize}
Let $H^\spa$ denote the subgraph of $H$ consisting of sparse edges.
\end{setup}

\subsection{Example: counting triangles in \texorpdfstring{$\Gamma$}{Gamma}}

We start by showing, as an example, how to prove the counting
lemma in $\Gamma$ for triangles. Most of the ideas found in the rest of this
section can already be found in this special case.

\begin{proposition}
  \label{prop:Gamma-triangle}
  Assume Setup~\ref{set:Gamma}. Let $H$ be a triangle with vertices
  $\set{1,2,3}$. Assume that $k \geq 2$. Then $\abs{\Gamma(H) - p^3} \leq 5cp^3$.
\end{proposition}

\begin{proof}
  In the following integrals, we assume that $x,y$ and $z$ vary uniformly over $X_1, X_2$ and $X_3$, respectively.
  We have the telescoping sum
  \begin{multline} \label{eq:Gamma-triangle-telescope}
    \Gamma(H) - p^3
    = \int_{x,y,z} (\Gamma(x,y) - p)\Gamma(x,z)\Gamma(y,z) \ dxdydz \\
    + \int_{x,y,z} p(\Gamma(x,z)-p)\Gamma(y,z) \ dxdydz
    + \int_{x,y,z} p^2(\Gamma(y,z)-p) \ dxdydz.
  \end{multline}
  The third integral on the right-hand side of
  \eqref{eq:Gamma-triangle-telescope} is bounded in absolute value by
  $cp^3$ by the jumbledness of $\Gamma$. In particular, this implies
  that $\int_{y,z} \Gamma(y,z) \ dydz \leq (1+c)p$. Similarly we have
  $\int_{x,z} \Gamma(x,z) \ dxdz \leq (1+c)p$. Using
  \eqref{eq:jumbled}, the second integral is bounded in absolute
  value by
  \[
  \int_y c p^3 \sqrt{\int_z \Gamma(y,z) \ dz} \ dy \leq cp^3 \sqrt{\int_{y,z} \Gamma(y,z) \ dydz} \leq c\sqrt{(1+c)p}p^3.
  \]
  Finally, the first integral on the right-hand
  side of \eqref{eq:Gamma-triangle-telescope} is bounded in absolute
  value by, using \eqref{eq:jumbled} and the Cauchy-Schwarz inequality,
  \begin{equation} \label{eq:Gamma-triangle-cauchy}
  \int_z cp^2 \sqrt{\int_x \Gamma(x,z) \ dx}\sqrt{\int_y \Gamma(y,z) \
    dy} \ dz
  \leq cp^2 \sqrt{\int_{x,z} \Gamma(x,z) \ dxdz}\sqrt{\int_{y,z} \Gamma(y,z) \
    dydz}
  \leq c(1+c)p^3.
  \end{equation}
  Therefore, \eqref{eq:Gamma-triangle-telescope} is bounded in
  absolute value by $5cp^3$.
\end{proof}

\begin{remark}
  (1) In the more general proof, the step corresponding to
  \eqref{eq:Gamma-triangle-cauchy} will be slightly different but is
  similar in its application of the Cauchy-Schwarz
  inequality.

  (2)
  The proof shows that we do not need the full strength of the
  jumbledness everywhere---we only need
  $(p,cp^{3/2}\sqrt{\abs{X}\abs{Z}})$-jumbledness for $(X,Z)_\Gamma$
  and $(p,cp\sqrt{\abs{Y}\abs{Z}})$-jumbledness for $(Y,Z)_\Gamma$. In
  Section~\ref{sec:oneside}, it will be useful to have a counting lemma
  with such non-balanced jumbledness assumptions in order to optimize our
  result. To keep things simple and clear, we will assume balanced
  jumbledness conditions here and remark later on the changes needed when we
  wish to have optimal non-balanced ones.
\end{remark}

\subsection{Notation} \label{sec:notation}

In the proof of the counting lemmas we often meet
expressions such as $G(x_1,x_2)G(x_1,x_3)G(x_2,x_3)$ and their
integrals. We introduce some compact notation for such products and
integrals. Note that if we are counting copies of $H$, we will usually
assign each vertex $a$ of $H$ to some vertex subset $X_a$ and we will
only be interested in counting those embeddings where each vertex of
$H$ is mapped into the vertex subset assigned to it. If $U \subseteq
V(H)$, a map $U \to V(G)$ or $U \to V(\Gamma)$ is called
\emph{compatible} if each vertex of $U$ gets mapped into the vertex
set assigned to it. We can usually assume without loss of generality
that the vertex subsets $X_a$ are disjoint for different vertices of
$H$, as we can always create a new multipartite graph with disjoint
vertex subsets $X_a$ with the same $H$-embedding counts as the
original graph.

If $f$ is a symmetric function on pairs of vertices of $G$ (actually
we only care about its values on $X_a \x X_b$ for $ab \in E(H)$) and $\bx
\colon V(H) \to V(G)$ is any compatible map (we write $\bx(a) = x_a$), then we define
\[f\cond{H}{\bx} := \prod_{ab \in E(H)} f(x_a, x_b).\]
By taking the expectation as $\bx$ varies uniformly over all compatible maps $V(H) \to V(G)$, we can define the value of a function on a graph.
\[f(H) := \EE_{\bx}\sqb{f\cond{H}{\bx}} = \int_{\bx} f\cond{H}{\bx} \
d\bx.\] We shall always assume that the measure $d \bx$ is the uniform
probability measure on compatible maps.

For unweighted graphs, we use $G$ and $\Gamma$ to denote also the characteristic function of
the edge set of the graph, so that $G(H)$ is the probability that a uniformly
random compatible map $V(H) \to V(G)$ is a graph homomorphism from $H$
to $G$. For weighted graphs, the value on the edges are the edge
weights. For counting lemmas, we are interested in comparing $G(H)$
with $q(H)$, which comes from setting $q(x_a,x_b)$ to be some constant
$q_{ab}$ for each $ab \in E(H)$.

It will be useful to have some notation for the conditional sum of a function $f$ given that some vertices have been fixed. If $U \subseteq V(H)$ and $\by \colon U \to V(G)$ is any compatible map, then
\[
f\cond{H}{\by} := \EE_{\bx}\sqcond{f\cond{H}{\bx}}{\bx\vert_{U} = \by} = \int_{\bz} f\cond{H}{\by, \bz} \ d\bz,
\]
where, in the integral, $\bz$ varies uniformly over all compatible maps $V(H)\setminus U \to V(G)$, and the notation $\by,\bz$ denotes the compatible map $V(H) \to V(G)$ built from combining $\by$ and $\bz$. Note that when $U = \emptyset$, $f\cond{H}{\by} = f(H)$. When $U = V(H)$, the two definitions of $f\cond{H}{\by}$ agree. When $U = \set{a_1, \dots, a_t}$, we sometimes write $\by$ as $a_1 \to y_1$, \dots, $a_t \to y_t$, so we can write $f\cond{H}{a_1 \to y_1, \dots, a_t \to y_t}$.

Since we work with approximations frequently, it will be convenient if we introduce some shorthand. If $A, B, P$ are three quantities, we write
\[A \approxmod^P_{c,\epsilon} B\]
to mean that for every $\theta > 0$, we can find $c, \epsilon > 0$ of size
at least polynomial in $\theta$ (i.e., $c,\epsilon \geq
\Omega(\theta^r)$ as $\theta \to 0$ for some $r > 0$) so that $\abs{A - B} \leq \theta P$. Sometimes one of $c$
or $\epsilon$ is omitted from the $\approxmod$ notation if
$\theta$ does not depend on the parameter. Note that the dependencies
do not depend on the parameters $p$ and $q$, but may depend on the
graphs to be embedded, e.g., $H$. For instance, a counting lemma can
be phrased in the form
\[
G(H) \approxmod_{c,\epsilon}^{p(H)} q(H).
\]

\subsection{Counting graphs in \texorpdfstring{$\Gamma$}{Gamma}} \label{sec:counting-graphs-Gamma}

We begin by giving a counting lemma in $\Gamma$, which is
significantly easier than counting in $G$. We remark that a similar
counting lemma for $\Gamma$ an $(n,d,\lambda)$ regular graph was
proven by Alon (see \cite[Thm.~4.10]{KS06}).

\begin{proposition} \label{prop:Gamma-counting} Assume
  Setup~\ref{set:Gamma}. If $k \geq \frac{d(L(H^\spa)) + 2}{2}$, then
  \[
  \abs{\Gamma (H) - p(H)} \leq \paren{(1+c)^{e(H^\spa)}  - 1} p(H).
  \]
\end{proposition}

The exact coefficient of $p(H)$ in the bound is not important. Any
bound of the form $O(c)p(H)$ suffices.

Dense edges play no role, so it suffices to consider the case when all
edges of $H$ are sparse. We prove Proposition~\ref{prop:Gamma-counting} by iteratively applying the
following inequality.

\begin{lemma} \label{lem:Gamma-cauchy}
    Let $H$ be a graph with vertex set $\set{1,\dots,m}$.
  Let $\Gamma$ be a graph with vertex subsets $X_1, \dots, X_m$.
  Let $ab \in E(H)$. Let $H_{-ab}$ denote $H$
  with the edge $ab$ removed. Let $H_{-a,-b}$ denote $H$ with all
  edges incident to $a$ or $b$ removed. Assume that $\Gamma(X_a, X_b)$ is $(p,
  \gamma\sqrt{\abs{X_a}\abs{X_b}})$-jumbled. Let $f \colon V(\Gamma) \x
  V(\Gamma) \to [0,1]$ be any symmetric function. Then
  \[
  \abs{\int_{\substack{x \in X_a \\ y \in X_b}} (\Gamma(x,y) - p)
    f\cond{H_{-ab}}{a \to x, b \to y} \ dxdy} \leq \gamma \sqrt{
    f(H_{-ab})f(H_{-a,-b})}.
  \]
\end{lemma}

\begin{proof}
  Let $H_{a,-ab}$ denote the edges of $H_{-ab}$ incident to $a$, and
  let $H_{b,-ab}$ be the edges of $H_{-ab}$ incident to $b$. Then
  $H_{-ab} = H_{-a,-b} \uplus H_{a,-ab} \uplus H_{b,-ab}$, as a disjoint union
  of edges. In the following calculation, $\bz$ varies uniformly over
  compatible maps $V(H)\setminus\set{a,b} \to V(\Gamma)$, $x$ varies
  uniformly over $X_a$, and $y$ varies uniformly over $X_b$. The three
  inequalities that appear in the calculation follow from, in order,
  the triangle inequality, the jumbledness condition, and the
  Cauchy-Schwarz inequality.
  {\allowdisplaybreaks
  \begin{align*}
    & \abs{\int_{x,y} (\Gamma(x,y) - p) f\cond{H_{-ab}}{a \to x, b \to y} \
      dxdy}
    \\
    &= \abs{\int_{\bz} f\cond{H_{-a,-b}}{\bz} \int_{x,y} (\Gamma(x,y) - p)
      f\cond{H_{a,-ab}}{a \to x,\bz} f\cond{H_{b,-ab}}{b \to y,\bz} \ dxdy d\bz}
    \\
    &\leq \int_{\bz} f\cond{H_{-a,-b}}{\bz} \abs{\int_{x,y}
      (\Gamma(x,y) - p)f\cond{H_{a,-ab}}{a \to x,\bz}f\cond{H_{b,-ab}}{b \to
        y,\bz} \ dxdy} d\bz
    \\
    &\leq \int_{\bz} f\cond{H_{-a,-b}}{\bz} \gamma
    \sqrt{\int_{x} f\cond{H_{a,-ab}}{a \to x,\bz} \ dx}
    \sqrt{\int_{y} f\cond{H_{b,-ab}}{b \to y,\bz} \ dy} d\bz
    \\
    &= \gamma \int_{\bz} f\cond{H_{-a,-b}}{\bz}
    \sqrt{f\cond{H_{a,-ab}}{\bz}}\sqrt{f\cond{H_{b,-ab}}{\bz}} d\bz
    \\
    &\leq \gamma \sqrt{\int_{\bz} f\cond{H_{-a,-b}}{\bz} \ d\bz}
    \sqrt{\int_{\bz}
      f\cond{H_{-a,-b}}{\bz}f\cond{H_{a,-ab}}{\bz} f\cond{H_{b,-ab}}{\bz} d\bz}
    \\
    &= \gamma \sqrt{ f(H_{-a,-b}) f(H_{-ab})}.
  \end{align*}
  }
\end{proof}

\begin{proof}
  [Proof of Proposition~\ref{prop:Gamma-counting}]
  As remarked after the statement of the proposition, it suffices to
  prove the result in the case when all edges of $H$ are sparse.
  We induct on the number of edges of $H$. If $H$ has no edges, then
  $\Gamma(H) = p(H) = 1$. So assume that $H$ has at least one edge. Since $k \geq \frac{1}{2}(d(L(H))
  + 2)$, we can find an edge $ab$ of $H$ such that $\deg_H(a) +
  \deg_H(b) \leq d(L(H)) + 2 \leq 2k$. Let
  $H_{-ab}$ and $H_{-a,-b}$ be as in Lemma~\ref{lem:Gamma-cauchy}. Since $L(H)$
  is $(2k-2)$-degenerate, the line graph of any subgraph of $H$ is
  also $(2k-2)$-degenerate. By the
  induction hypothesis, we have $\abs{\Gamma(H_{-ab}) - p(H_{-ab})}
  \leq ((1+c)^{e(H)-1} - 1)p(H_{-ab})$ and
  $\abs{\Gamma(H_{-a,-b}) - p(H_{-a,-b})} \leq ((1+c)^{e(H)-1} - 1)p(H_{-a,-b})$. We have
  \begin{equation*}
    \label{eq:Gamma-decomp}
    \Gamma(H) - p(H) = p\cdot \paren{\Gamma(H_{-ab}) - p(H_{-ab})}
    + \int_{\substack{x \in X_a \\ y \in X_b}} (\Gamma(x,y) -
    p)\Gamma\cond{H_{-ab}}{a\to x, b \to y} \ dxdy.
  \end{equation*}
  The first term on the right is bounded in absolute value by
  $((1+c)^{e(H)-1} - 1)p(H)$. For the second term, by Lemma~\ref{lem:Gamma-cauchy} and the induction hypothesis, we have
  \begin{equation} \label{eq:Gamma-p}
    \begin{split}
    \abs{\int_{x,y} (\Gamma(x,y) - p)\Gamma\cond{H_{-ab}}{a\to x, b \to y}
      \ dxdy}
    &\leq cp^k \sqrt{\Gamma(H_{-ab}) \Gamma(H_{-a,-b})}
    \\
    &\leq cp^k (1+c)^{e(H)-1} \sqrt{p(H_{-ab})p(H_{-a,-b})}
    \\
    &\leq c (1+c)^{e(H)-1} p(H).
    \end{split}
  \end{equation}
  The last inequality is where we used $2k \geq \deg_H(a) + \deg_H(b)$.
  Combining the two estimates gives the desired result.
\end{proof}

\subsection{Counting partial embeddings into \texorpdfstring{$\Gamma$}{Gamma}} \label{sec:Gamma-partial}

As outlined in Section~\ref{sec:countingstrat}, we need to count
embeddings of $H$ where some edges are embedded into $G$ (the
straight edges in the figures) and some edges are embedded into
$\Gamma$ (the wavy edges). We prove counting estimates for these embeddings here. The
main result of this section is summarized in the figure below. The
proofs are almost identical to that of
Proposition~\ref{prop:Gamma-counting}. We just need to be a little more
careful with the exponents on the jumbledness parameter.

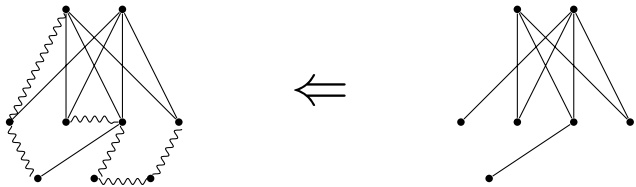
\begin{center}
  \begin{tikzpicture}[scale=.75]
      \begin{scope}
      \node[p] (a) at (-.5,2) {};
      \node[p] (a2) at (.5,2) {};
      \node[p] (1) at (-1.5,0) {};
      \node[p] (2) at (-.5,0) {};
      \node[p] (3) at (.5,0) {};
      \node[p] (4) at (1.5,0) {};
      \node[p] (5) at (-1,-1) {};
      \node[p] (6) at (0,-1) {};
      \node[p] (7) at (1,-1) {};
      \draw (1)--(a2)--(2)--(a)--(3)--(a2)--(4)--(a) (3)--(5);
      \draw decorate [Gamma] { (a)--(1)--(5)(3)--(6)--(7)--(4) (2)--(3)};
      \node at (4,0.5) {\LargeLeftarrow};
    \end{scope}
      \begin{scope}[shift={(8,0)}]
      \node[p] (a) at (-.5,2) {};
      \node[p] (a2) at (.5,2) {};
      \node[p] (1) at (-1.5,0) {};
      \node[p] (2) at (-.5,0) {};
      \node[p] (3) at (.5,0) {};
      \node[p] (4) at (1.5,0) {};
      \node[p] (5) at (-1,-1) {};
      \draw (1)--(a2)--(2)--(a)--(3)--(a2)--(4)--(a) (3)--(5);
    \end{scope}
  \end{tikzpicture}
\end{center}

First we consider the case where exactly one edge needs to be embedded
into $\Gamma$ and the other edges are embedded into some subgraph of
$\Gamma$. To state the result requires a little notation. Suppose that $H = H' \cup H''$ is an edge disjoint partition of the
graph $H$ into two subgraphs $H'$ and $H''$. We define $d(L(H', H''))$ to be
the smallest $d$ such there is an ordering of the edges of $H$ with the edges
of $H'$ occurring before the edges of $H''$ such that every edge $e$ has at most
$d$ neighbors, that is, edges containing either of the endpoints of $e$, which appear earlier in the ordering.

\begin{lemma}\label{lem:Gamma-partial-one-edge}
  Assume Setup~\ref{set:Gamma}.
  Let $ab \in E(H)$ and $H_{-ab}$ be the graph $H$ with edge $ab$
  removed. Assume $k \geq \frac{d(L(H_{-ab}^\spa, ab^\spa)) + 2}{2}$.
  Let $G$ be any weighted subgraph of $\Gamma$ (i.e., $0 \leq G \leq
  \Gamma$ as functions). Let $g$ denote the function that agrees with
  $\Gamma$ on $X_a \x X_b$ and with $G$
  everywhere else. Then
  \[
  \abs{g(H) - p_{ab}G(H_{-ab})} \leq c(1+c)^{e(H^{\spa})-1} p(H).
  \]
\end{lemma}

The lemma follows from essentially the same
calculation as \eqref{eq:Gamma-p}, except that we take $ab$ as our
first edge to remove (this is why there is a stronger requirement on
$k$) and then use $G \leq \Gamma$.

Iterating the lemma, we obtain the following result where multiple
edges need to be embedded into $\Gamma$. It can be proved by iterating
Lemma~\ref{lem:Gamma-partial-one-edge} or mimicking the proof of Proposition~\ref{prop:Gamma-counting}.

\begin{lemma} \label{lem:Gamma-partial}
  Assume Setup~\ref{set:Gamma}. Let $H'$ be a subgraph of $H$. Assume
  $k \geq \frac{d(L(H'^{\spa}, (H \setminus H')^{\spa})) +
  2}{2}$. Let $G$ be a weighted subgraph of $\Gamma$. Let $g$ be a
  function that agrees with $\Gamma$ on $X_a \x X_b$ when $ab \in
  E(H \setminus H')$ and with $G$ otherwise. Then
  \[
  \abs{g(H) - p(H \setminus H')G(H')} \leq \paren{(1+c)^{e(H^{\spa})} -
    1}p(H).
  \]
\end{lemma}

\subsection{Exceptional sets} \label{sec:Gamma-exceptional}

This section contains a couple of lemmas about $\Gamma$ that we will need later
on. The reader may choose to skip this section until the results are needed.

We begin with a standard estimate for the number of vertices in a jumbled graph whose degrees deviate
from the expected value. The proof follows immediately from the
definition of jumbledness.

\begin{lemma} \label{lem:deg-est} Let $\Gamma$ be a $(p,
  \gamma\sqrt{\abs{X}\abs{Y}})$-jumbled graph between vertex subsets $X$
  and $Y$. Let $v \colon Y \to [0,1]$ and let $\xi >
  0$. If
  \[
  U \subseteq \setcond{ x \in X}{\int_{y \in Y} \Gamma(x,y)v(y) \ dy \geq (1
    + \xi)p \EE v}
  \]
  or
  \[
  U \subseteq \setcond{ x \in X}{\int_{y \in Y} \Gamma(x,y)v(y) \ dy \leq (1
    - \xi)p \EE v},
  \]
  then
  \[
  \frac{\abs{U}}{\abs{X}} \leq \frac{\gamma^2}{\xi^2 p^2 \EE v}.
  \]
\end{lemma}

The next lemma says that restrictions of the count $\Gamma(H)$ to
small sets of vertices or pairs of vertices yield small counts. This
will be used in Section~\ref{sec:densification} to bound the
contributions from exceptional sets.

\begin{lemma}
  \label{lem:small-restrict}
  Assume Setup~\ref{set:Gamma} with $k \geq \frac{d(L(H^\spa)) +
    2}{2}$. Let $\bx \colon V(H) \to V(\Gamma)$ vary uniformly over
  compatible maps. Let $u \colon V(\Gamma) \to [0,1]$ be any function
  and write $u(\bx) = \prod_{a \in V(H)} u(x_a)$. Let $E'$ be a
  weighted graph with the same vertices as $\Gamma$ whose edge
  set is supported on $X_a \x X_b$ for $ab \notin H^\spa$. Let $H'$ be
  any graph with the same vertices as $H$. Then
  \[
  \int_\bx \Gamma\cond{H}{\bx} u(\bx)
    E'\cond{H'}{\bx} \ d\bx \leq \paren{(1+c)^{e(H^\spa)} - 1 +
    \int_\bx  u(\bx)
    E'\cond{H'}{\bx} \ d\bx}p(H).
  \]
\end{lemma}

Lemma~\ref{lem:small-restrict} follows by showing that
\[
\abs{\int_\bx \paren{\Gamma\cond{H}{\bx} - p(H)} u(\bx)
  E'\cond{H'}{\bx} \ d\bx} \leq \paren{(1+c)^{e(H^\spa)} - 1}p(H).
\]
The proof is similar to that of
Proposition~\ref{prop:Gamma-counting}. In the step analogous to
\eqref{eq:Gamma-p}, after applying the jumbledness condition as our
first inequality, we bound $u$ and $E'$ by $1$ and then continue
exactly the same way.

\section{Counting in $G$} \label{sec:counting-G}

In this section we develop the counting lemma for subgraphs $G$ of $\Gamma$, as outlined in
Section~\ref{sec:countingstrat}. The two key ingredients are doubling
and densification, which are discussed in Sections~\ref{sec:doubling}
and \ref{sec:densification}, respectively.
Here is the common setup for this section.

\begin{setup} \label{set:G}
  Assume Setup~\ref{set:Gamma}. Let $\epsilon > 0$. Let $G$ be a weighted subgraph of $\Gamma$. For
  every edge $ab \in E(H)$, assume that $(X_a, X_b)_G$ satisfies
  $\DISC(q_{ab}, p_{ab}, \epsilon)$, where $0 \leq q_{ab} \leq p_{ab}$.
\end{setup}

Unlike in Section~\ref{sec:counting-Gamma}, we do not make an effort
to keep track of the unimportant coefficients of $p(H)$ in the error
bounds, as it would be cumbersome to do so. Instead, we use the
$\approxmod$ notation introduced in Section~\ref{sec:notation}.

The goal of this section is to prove the following counting
lemma. This is slightly more general than
Theorem~\ref{thm:sparse-counting} in that it allows $H$ to have both
sparse and dense edges.

\begin{proposition}
  \label{prop:G-counting}
  Assume Setup~\ref{set:G} with $k \geq
  \min\set{\frac{\Delta(L(H^\spa)) + 4}{2}, \frac{d(L(H^\spa)) +
      6}{2}}$. Then
  \[
  G(H) \approxmod_{c,\epsilon}^{p(H)} q(H).
  \]
\end{proposition}

The requirement on $k$ stated in Proposition~\ref{prop:G-counting} is
not necessarily best possible. The proof of the counting lemma
will be by induction on the vertices of $H$, removing one vertex at a
time. A better bound on $k$ can sometimes be obtained by tracking the
requirements on $k$ at each step of the procedure, as explained in a
tutorial in Section~\ref{sec:tut-jumble}.

\subsection{Doubling} \label{sec:doubling}

Doubling is a technique used to reduce the problem of counting
embeddings of $H$ in $G$ to the problem of counting embeddings of $H$
with one vertex deleted.

If $a \in V(H)$, $H_{a \x 2}$ is the graph $H$ with vertex $a$
doubled. In the assignment of vertices of $H_{a \x 2}$ to vertex
subsets of $\Gamma$, the new vertex $a'$ is assigned to the same
vertex subset as $a$. Let $H_a$ be the subgraph of $H$ consisting of
edges with $a$ as an endpoint, and let $H_{a,a\x 2}$ be $H_a$ with $a$
doubled. Let $H_{-a}$ be the subgraph of $H$ consisting of edges not
having $a$ as an endpoint. We refer to Figure~\ref{fig:doubling} for
an illustration.

\begin{lemma}[Doubling]
  \label{lem:doubling}
  Let $H$ be a graph with vertex set $\set{1, \dots, m}$. Let
  $\Gamma$ be a weighted graph with vertex subsets $X_1, \dots, X_m$, and let $G$
  be a weighted subgraph of $\Gamma$. For each edge $bc$ of $H$, we have
  numbers $0 \leq q_{bc} \leq p_{bc} \leq 1$. Let $g$ be a function
  that agrees with $G$ on $X_i
  \x X_j$ whenever $a \in \set{i,j}$ and with $\Gamma$ on $X_i
  \x X_j$ whenever $a \notin \set{i,j}$. Then
  \begin{multline} \label{eq:doubling}
    \abs{G(H) - q(H)} \\
    \leq q(H_{a})\abs{G(H_{-a}) - q(H_{-a})}
    + G(H_{-a})^{1/2}\paren{ g(H_{a\x 2}) - 2q(H_a) g(H) + q(H_{a})^2
      \Gamma(H_{-a})}^{1/2}.
  \end{multline}
\end{lemma}

\begin{proof}
    Let $\by$ vary uniformly over compatible maps $V(H) \setminus \set{a} \to
    V(G)$ where $\by(b) \in X_b$ for each $b \in
    V(H)\setminus\set{a}$.
  We have
  \[
  G(H) - q(H) = q(H_a)(G(H_{-a}) - q(H_{-a})) + \int_{\by} \paren{G\cond{H_a}{\by} - q(H_a)}
  G\cond{H_{-a}}{\by} \ d\by.
  \]
  It remains to bound the integral, which we can do using the Cauchy-Schwarz inequality.
  \begin{align*}
    & \paren{\int_{\by} \paren{G\cond{H_a}{\by} - q(H_a)}
      G\cond{H_{-a}}{\by} \ d\by}^2
    \\
    & \leq \paren{\int_{\by} G\cond{H_{-a}}{\by} \
      d\by} \paren{\int_{\by} \paren{G\cond{H_a}{\by} - q(H_a)}^2
      G\cond{H_{-a}}{\by} \ d\by}
    \\
    & = G(H_{-a}) \int_{\by} \paren{G\cond{H_a}{\by} - q(H_a)}^2
    G\cond{H_{-a}}{\by} \ d\by
    \\
    & \leq G(H_{-a}) \int_{\by} \paren{G\cond{H_a}{\by} - q(H_a)}^2 \Gamma
    \cond{H_{-a}}{\by} \ d\by
    \\
    & = G(H_{-a})\paren{g(H_{a\x 2}) - 2q(H_a) g(H) + q(H_{a})^2
      \Gamma(H_{-a})}.
   \end{align*}
\end{proof}

Using Lemma~\ref{lem:Gamma-partial}, we know that under appropriate
hypotheses, we have
\begin{align*}
  g(H_{a\x 2}) &\approxmod_{c}^{p(H_{a\x 2})} p(H_{-a}) G(H_{a,a\x 2}),
  \\
  g(H) &\approxmod_{c}^{p(H)} p(H_{-a})G(H_{a})
  \\
  \text{and} \quad \Gamma(H_{-a}) &\approxmod_{c}^{p(H_{-a})} p(H_{-a}).
\end{align*}
If we can show that
\begin{align*}
  G(H_{a,a\x 2}) &\approxmod_{c,\epsilon}^{p(H_{a,a\x 2})} q(H_{a,a\x 2})
  \\
  \text{and} \quad
  G(H_{a}) &\approxmod_{c,\epsilon}^{p(H_a)} q(H_{a}),
\end{align*}
then the rightmost term in \eqref{eq:doubling} is
$\approxmod_{c,\epsilon}^{p(H)} 0$, which would reduce the problem to
showing that $G(H_{-a}) \approxmod_{c,\epsilon}^{p(H_{-a})}
q(H_{-a})$. This reduction step is spelled out below. See
Figure~\ref{fig:doubling-reduction} for an illustration.

\begin{figure}[!htp]
  \centering
  \begin{tikzpicture}[scale=.5]
    \begin{scope}[shift={(0,0)}]
      \node[p, label=right:$a$] (a) at (0,2) {};
      \node[p] (1) at (-1.5,0) {};
      \node[p] (2) at (-.5,0) {};
      \node[p] (3) at (.5,0) {};
      \node[p] (4) at (1.5,0) {};
      \node[p] (5) at (-1,-1) {};
      \node[p] (6) at (0,-1) {};
      \node[p] (7) at (1,-1) {};
      \draw (a)--(1) (a)--(2) (a)--(3) (a)--(4);
      \draw[dense] (a)--(1) (3)--(5);
      \draw (1)--(5)(3)--(6)--(7)--(4) (2)--(3);
      \node at (0,-2) {$G(H)$};
      \node at (0,-3.5) {\LargeUparrow};
      \draw[thick] decorate [decoration={brace, amplitude=1cm}] {(-9,-7) -- (9,-7)};
    \end{scope}
    \begin{scope}[shift={(-6,-9)}]
      \node[p] (1) at (-1.5,0) {};
      \node[p] (2) at (-.5,0) {};
      \node[p] (3) at (.5,0) {};
      \node[p] (4) at (1.5,0) {};
      \node[p] (5) at (-1,-1) {};
      \node[p] (6) at (0,-1) {};
      \node[p] (7) at (1,-1) {};
      \draw[dense] (5)--(3);
      \draw (1)--(5)(3)--(6)--(7)--(4) (2)--(3);
      \node at (0,-2) {$G(H_{-a})$};
    \end{scope}
    \begin{scope}[shift={(0,-9)}]
      \node[p] (a) at (-.5,2) {};
      \node[p] (a2) at (.5,2) {};
      \node[p] (1) at (-1.5,0) {};
      \node[p] (2) at (-.5,0) {};
      \node[p] (3) at (.5,0) {};
      \node[p] (4) at (1.5,0) {};
      \node at (0,-2) {$G(H_{a,a\x2})$};
      \draw[dense] (a)--(1)--(a2);
      \draw (a2)--(2)--(a)--(3)--(a2)--(4)--(a);

    \end{scope}
    \begin{scope}[shift={(6,-9)}]
      \node[p] (a) at (0,2) {};
      \node[p] (1) at (-1.5,0) {};
      \node[p] (2) at (-.5,0) {};
      \node[p] (3) at (.5,0) {};
      \node[p] (4) at (1.5,0) {};
      \node at (0,-2) {$G(H_{a})$};
      \draw[dense] (a)--(1);
      \draw (a)--(2) (a)--(3) (a)--(4);
    \end{scope}
  \end{tikzpicture}
  \caption{The doubling reduction.}
  \label{fig:doubling-reduction}
\end{figure}
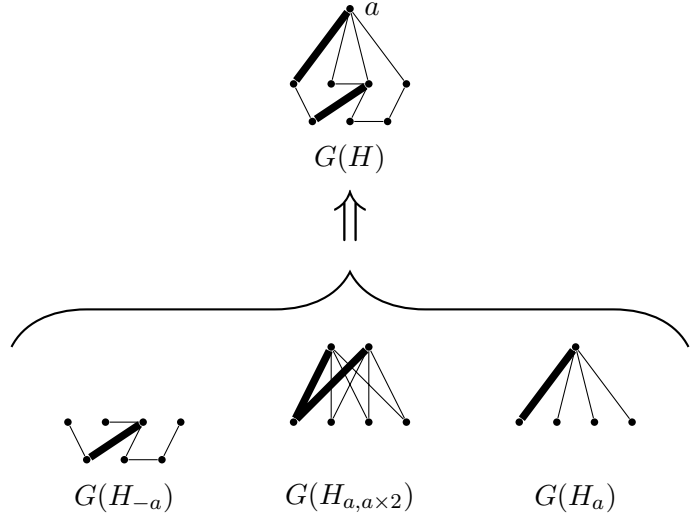

\begin{lemma}
  \label{lem:doubling-reduction}
  Assume Setup~\ref{set:G}.
  Let $a \in V(H)$.
  Suppose that $k \geq \frac{d(L(H_{a,a\x 2}^\spa, H_{-a}^\spa)) +
    2}{2}$. Suppose that
  \[
  G(H_{-a}) \approxmod_{c,\epsilon}^{p(H_{-a})} q(H_{-a}), \quad
  G(H_{a}) \approxmod_{c,\epsilon}^{p(H_{a})} q(H_{a}) \text{ and}
  \quad
  G(H_{a,a\x2}) \approxmod_{c,\epsilon}^{p(H_{a,a\x2})} q(H_{a,a\x2}).
  \]
  Then
  \[
  G(H) \approxmod_{c,\epsilon}^{p(H)} q(H).
  \]
\end{lemma}

\begin{remark}
  We do not always need the full strength of Setup~\ref{set:G}
  (although it is convenient to state it as such). For
  example, when $H$ is a triangle with vertices $\set{1,2,3}$,
  $H_{-1}$ is a single edge, so we do not need discrepancy on
  $(X_2,X_3)_G$ to obtain $G(H_{-a}) \approxmod_{c,\epsilon}^p
  q_{23}$. In particular, our approach gives the triangle counting lemma in the
  form stated in Kohayakawa et al.~\cite{KRSS10}, where discrepancy is assumed for only two of the
  three pairs of vertex subsets of $G$.
\end{remark}

\subsection{Densification} \label{sec:densification}

Densification is the technique that allows us to transform a
subdivided edge of $H$ into a single dense edge, as summarized in the
figure below.  This section also contains a counting lemma for trees
(Proposition~\ref{prop:tree}).

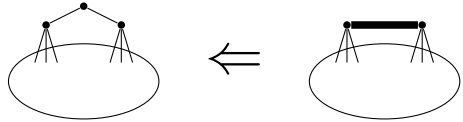
\begin{center}
  \begin{tikzpicture}[scale=.5]
    \begin{scope}
      \node[p] (a) at (-1,0) {};
      \node[p] (b) at (0,.5) {};
      \node[p] (c) at (1,0) {};
      \draw (a)--(b)--(c);
      \draw (a) -- +(-.3,-1) (a) -- +(0,-1)  (a) -- +(.3,-1) ;
      \draw (c) -- +(-.3,-1) (c) -- +(0,-1) (c) -- +(.3,-1) ;
      \draw (0,-1.5) ellipse (2cm and 1cm);
    \end{scope}
    \node at (4,-1) {\LargeLeftarrow};
    \begin{scope}[shift={(8,0)}]
      \node[p] (a) at (-1,0) {};
      \node[p] (c) at (1,0) {};
      \draw[dense] (a)--(c);
      \draw (a) -- +(-.3,-1) (a) -- +(0,-1)  (a) -- +(.3,-1) ;
      \draw (c) -- +(-.3,-1) (c) -- +(0,-1) (c) -- +(.3,-1) ;
      \draw (0,-1.5) ellipse (2cm and 1cm);
    \end{scope}
  \end{tikzpicture}
\end{center}

We introduce the following notation for  the density analogues of degree and codegree. If $\Gamma$ (and similarly $G$)
is a weighted graph with vertex subsets $X, Y, Z$, then for $x \in X$
and $z \in Z$, we write
\begin{align*}
  G(x,Y) &= \int_{y \in Y} G(x,y) \ dy, \\
  \text{and} \quad
  G(x,Y,z) &= \int_{y \in Y} G(x,y)G(y,z) \ dz.
\end{align*}

Now we state the goal of this section.

\begin{lemma}[Densification]
  \label{lem:densification}
  Assume Setup~\ref{set:G} with $k \geq \frac{d(L(H^\spa)) + 2}{2}$.
  Let $1,2,3$ be vertices in $H$ such that $1$ and $3$ are the only
  neighbors of $2$ in $H$, and $13 \notin E(H)$. Replace the induced
  bipartite graph $(X_1, X_3)_G$ by the weighted bipartite graph
  defined by
  \[
  G(x_1,x_3) = \frac{1}{2p_{12}p_{23}} \min\set{G(x_1,X_2,x_3), 2p_{12}p_{23}}.
  \]
  Let $H'$ denote the graph obtained from $H$ by deleting edges
  $12$ and $23$ and adding edge $13$. Let $q_{13} =
  \frac{q_{12}q_{23}}{2p_{12}p_{23}}$ and $p_{13} = 1$. Then $(X_1, X_3)_G$ satisfies
  $\DISC(q_{13}, 1, 2\epsilon + 18c)$ and
  \[
  \abs{G(H) - 2p_{12}p_{23} G(H')} \leq ((1+c)^{e(H^{\spa})} - 1 + 26c^2)p(H).
  \]
\end{lemma}

Note that $q(H) = 2p_{12}p_{23} q(H')$. So we obtain the following
reduction step as a corollary.

\begin{corollary}
  \label{cor:densification}
  Continuing with Lemma~\ref{lem:densification}. If $G(H')
  \approxmod_{c,\epsilon}^{p(H')} q(H')$ then $G(H)
  \approxmod_{c,\epsilon}^{p(H)} q(H)$ in the original graph.
\end{corollary}

The proof of Lemma~\ref{lem:densification} consists of the following
steps:
\begin{enumerate}
\item Show that the weighted graph on $X_1 \x X_3$ with weights
  $G(x_1, X_2, x_3)$ satisfies discrepancy.
\item Show that the capping of weights has negligible effect on discrepancy.
\item Show that the capping of weights has negligible effect on the $H$-count.
\end{enumerate}

Steps 2 and 3 are done by bounding the contribution from pairs of
vertices in $X_1 \x X_3$ which have too high co-degree with $X_2$ in
$\Gamma$.

We shall focus on the more difficult case when both edges $12$ and
$23$ are sparse. The case when at least one of the two edges is dense
is analogous and much easier. Let us start with a warm-up by showing
how to do step 1 for the latter dense case. We shall omit the rest of the
details in this case.

\begin{lemma}
  \label{lem:G-squared-d}
  Let $0 \leq q_1 \leq p_1\leq 1$, $0 \leq q_2 \leq 1$, $\epsilon > 0$.
  Let $G$ be a weighted graph with vertex subsets $X, Y, Z$, such that
  $(X,Y)_G$ satisfies $\DISC(q_1, p_1, \epsilon)$ and $(Y,Z)_G$
  satisfies $\DISC(q_2, 1, \epsilon)$. Then the graph $G'$ on
  $(X,Z)$ defined by
  $G'(x,z) = G(x,Y,z)$
  satisfies $\DISC(q_1q_2, p_1, 2\epsilon)$.
\end{lemma}

\begin{proof}
  Let $u \colon X \to [0,1]$ and $w \colon Z \to [0,1]$ be arbitrary
  functions. In the following integrals, let $x,y$ and $z$ vary uniformly
  over $X, Y$ and $Z$, respectively. We have
  \begin{align}
    &\int_{x,z} u(x) (G(x,Y,z) - q_1q_2) w(z) \ dxdz \nonumber
    \\
    &= \int_{x,y,z} u(x) (G(x,y)G(y,z) - q_1q_2) w(z) \ dxdydz \nonumber
    \\
    &= \int_{x,y,z} u(x) (G(x,y) - q_1)G(y,z)w(z) \ dxdydz + q_1
    \int_{x,y,z} u(x)(G(y,z)-q_2)w(z) \ dxdydz. \label{eq:G-squared-telescope}
  \end{align}
  Each of the two integrals in the last sum is bounded by
  $\epsilon p_1$ in absolute value by the discrepancy
  hypotheses. Therefore $(X,Z)_{G'}$ satisfies $\DISC(q_1q_2, p_1, 2\epsilon)$.
\end{proof}

The next lemma is step 1 for the sparse case.

\begin{lemma}
  \label{lem:G-squared}
  Let $c,p,\epsilon \in (0,1]$ and $q_1, q_2 \in [0,p]$.
  Let $\Gamma$ be a graph with vertex subsets $X, Y, Z$, and $G$ a
  weighted subgraph of $G$. Suppose that
  \begin{itemize}
  \item  $(X,Y)_\Gamma$ is $(p,
    cp^{3/2}\sqrt{\abs{X}\abs{Y}})$-jumbled and $(X,Y)_G$ satisfies
    $\DISC(q_1, p, \epsilon)$; and
  \item $(Y,Z)_\Gamma$ is $(p, cp^{3/2}\sqrt{\abs{Y}\abs{Z}})$-jumbled
    and $(Y,Z)_G$ satisfies $\DISC(q_2,p,\epsilon)$.
  \end{itemize}
  Then the graph $G'$ on $(X,Z)$ defined by $G'(x,z) = G(x,Y,z)$
  satisfies $\DISC(q_1q_2, p^2, 3\epsilon + 6c)$.
\end{lemma}

\begin{remark}
  By unraveling the proof of Lemma~\ref{lem:small-restrict}, we see
  that the exponent of $p$ in the jumbledness of $(X,Y)_\Gamma$ can be
  relaxed from $\frac{3}{2}$ to $1$.
\end{remark}

\begin{proof}
  We begin the proof the same way as Lemma~\ref{lem:G-squared-d}. In
  \eqref{eq:G-squared-telescope}, the second integral is bounded in
  absolute value by $\epsilon p^2$. We need to do more work to bound
  the first integral.

  Define $v \colon Y \to [0,1]$ by
  \[
  v(y) = \int_z G(y,z)w(z) \ dz.
  \]
  So the first integral in \eqref{eq:G-squared-telescope}, the
  quantity we need to bound, equals
  \begin{equation}
    \label{eq:G-squared-v}
    \int_{x,y}u(x)(G(x,y) - q_1)v(y) \ dxdy.
  \end{equation}
  If we apply discrepancy immediately, we get a bound of $\epsilon p$,
  which is not small enough, as we need a bound on the order of
  $o(p^2)$. The key observation is that $v(y)$ is bounded above by $2p$ on
  most of $Y$. Indeed, let
  \[
  Y' = \setcond{y \in Y}{\Gamma(y, Z) > 2p}.
  \]
  By Lemma~\ref{lem:deg-est} we have $\abs{Y'} \leq c^2p \abs{Y}$. Since $v1_{Y\setminus Y'}$ is bounded
  above by $2p$, we can apply discrepancy on $(X,Y)_G$ with the functions
  $u$ and $\frac{1}{2p} v1_{Y\setminus Y'}$ to obtain
  \[
  \abs{\int_{x,y}u(x)(G(x,y) - q_1)v(y)1_{Y\setminus Y'} \ dxdy} \leq
  2\epsilon p^2.
  \]
  In the following calculation, the first inequality follows from the
  triangle inequality; the second inequality follows from expanding
  $v(y)$ and using $G \leq \Gamma$ and $u, w \leq 1$; the third
  inequality follows from Lemma~\ref{lem:small-restrict} (applied with
  $u$ in the Lemma being the function $1_{Y'}$ on $Y$ and $1$
  everywhere else, and $H'$ the empty graph so that $E'\cond{H'}{\bx}
  = 1$ always).
  \begin{align*}
    &\abs{\int_{x,y}u(x)(G(x,y) - q_1)v(y)1_{Y'}(y) \ dxdy}
    \\
    &\leq \int_{x,y}u(x)G(x,y)v(y)1_{Y'} \ dxdy + \int_{x,y}
    u(x)q_1v(y)1_{Y'}(y) \ dxdy
    \\
    & \leq \int_{x,y,z} \Gamma(x,y) 1_{Y'}(y) \Gamma(y,z) \ dxdydz + q_1
    \int_{y,z}  1_{Y'}(y) \Gamma(y,z)\ dydz
    \\
    &
    \leq \paren{(1+c)^2 - 1 + \frac{\abs{Y'}}{\abs{Y}}}p^2 +
    q_1\paren{(1+c) - 1 + \frac{\abs{Y'}}{\abs{Y}}}p
    \\
    & \leq 6cp^2.
  \end{align*}
  Therefore, \eqref{eq:G-squared-v} is at most $(2\epsilon + 6c)p^2$
  in absolute value. Recall that the second integral in
  \eqref{eq:G-squared-telescope} was bounded by $\epsilon p^2$. The
  result follows from combining these two estimates.
\end{proof}

The technique used in Lemma~\ref{lem:G-squared} also allows us to count trees
in $G$.

\begin{proposition} \label{prop:tree}
  Assume Setup~\ref{set:G} with $H$ a tree and $k
  \geq \frac{\Delta(H^\spa) + 1}{2}$. Then
  \[
  G(H) \approxmod_{c,\epsilon}^{p(H)} q(H).
  \]
\end{proposition}

In fact, it can be shown that the error has the form
\[
\abs{G(H) - q(H)} \leq M_H (c + \epsilon)p(H)
\]
for some real number $M_H > 0$ depending on $H$.

To prove Proposition~\ref{prop:tree}, we formulate a weighted version
and induct on the number of
edges. The weighted version is stated below.

\begin{lemma} \label{lem:tree-weighted} Assume the same setup as in
  Proposition~\ref{prop:tree}. Let $u \colon V(G) \to [0,1]$ be any
  function. Let $\bx$ vary uniformly over all compatible maps $V(H)
  \to V(G)$. Write $u(\bx) = \prod_{a \in V(H)} u(x_a)$. Then,
  \[
  \int_{\bx} G\cond{H}{\bx} u(\bx) \ d\bx
  \approxmod_{c,\epsilon}^{p(H)} q(H) \int_{\bx} u(\bx) \ d\bx.
  \]
\end{lemma}

To prove Lemma~\ref{lem:tree-weighted}, we remove one leaf of $H$ at a
time and use the technique in the proof of Lemma~\ref{lem:G-squared}
to transfer the weight of the leaf to its unique neighboring
vertex and use Lemma~\ref{lem:small-restrict} to bound the
contributions of the vertices with large degrees in $\Gamma$. We omit the details.

Continuing with the proof of densification, the following estimate is
needed for steps 2 and 3.

\begin{lemma}
  \label{lem:two-densify-step-ss}
  Let $\Gamma$ be a graph with vertex subsets $X, Y, Z$, such that
  $(X,Y)_\Gamma$ is $(p, cp\sqrt{\abs{X}\abs{Y}})$-jumbled and
  $(Y,Z)_\Gamma$ is $(p, cp^{3/2}\sqrt{\abs{Y}\abs{Z}})$-jumbled. Let
  \[
  E' = \setcond{(x,z) \in X \x Z}{\Gamma(x,Y,z) > 2p^2}.
  \]
  Then $\abs{E'} \leq 26c^2\abs{X}\abs{Z}$.
\end{lemma}

\begin{proof}
  Let
  \[
  X' = \setcond{x \in X}{\abs{\Gamma(x, Y) - p} > \frac{p}{2}}.
  \]
  Then, by Lemma~\ref{lem:deg-est}, $\abs{X'} \leq 8c^2 \abs{X}$. For every $x \in X$, let
  \[
  Z'_x = \setcond{z \in Z}{\Gamma(x,Y,z) > 2p^2}.
  \]
  For $x \in X \setminus X'$, we have, again by Lemma~\ref{lem:deg-est}, that $\abs{Z'_x} \leq 18c^2
  \abs{Z}$. The result follows by noting that $E' \subseteq (X' \x Z)
  \cup \setcond{(x,z)}{x \in X \setminus X', z \in Z'_x}$.
\end{proof}

The following lemma is step 2 in the program.

\begin{lemma}
  \label{lem:G-square-bound} Let $c,\epsilon,p \in (0,1]$ and $q_1,
  q_2 \in [0,p]$.  Let $\Gamma$ be a graph with vertex subsets $X, Y,
  Z$, and $G$ a weighted subgraph of $\Gamma$. Suppose that
  \begin{itemize}
  \item $(X,Y)_\Gamma$ is $(p, cp\sqrt{\abs{X}\abs{Y}})$-jumbled and
    $(X, Y)_G$ satisfies $\DISC(q_1, p, \epsilon)$; and
  \item $(Y,Z)_\Gamma$ is $(p, cp^{3/2}\sqrt{\abs{Y}\abs{Z}})$-jumbled
    and $(Y,Z)_G$ satisfies $\DISC(q_2,p,\epsilon)$.
  \end{itemize} Then the graph $G'$ on $(X,Z)$ defined by $G'(x,z) =
  \min\set{G(x,Y,z), 2p^2}$ satisfies $\DISC(q_1q_2, p^2, 3\epsilon +
  35c)$.
\end{lemma}

\begin{proof}
  Let $u \colon X \to [0,1]$ and $w \colon Z \to [0,1]$ be any
  functions. In the following integrals, $x, y$ and $z$ vary uniformly over
  $X, Y$ and $Z$, respectively. We have
  \begin{multline*}
    \int_{x,z} (G'(x,z) - q_1q_2)u(x)w(z) \ dxdz \\
    = \int_{x,z} (G(x,Y,z) - q_1q_2)u(x)w(z) \ dxdz
    - \int_{x,z} (G(x,Y,z) - G'(x,z))u(x)w(z) \ dxdz.
  \end{multline*}
  The first integral on the right-hand side can be bounded in absolute
  value by $(3\epsilon+6c)p^2$ by Lemma~\ref{lem:G-squared}. For the
  second integral, let $E' = \setcond{(x,z) \in X \x Z}{\Gamma(x,Y,z)
    > 2p^2}$. We have
  \begin{align*}
    0 &\leq \int_{x,z} (G(x,Y,z) - G'(x,z))u(x)w(z) \ dxdz \\
    &\leq \int_{x,z} G(x,Y,z)1_{E'}(x,z) \ dxdz \\
    &\leq \int_{x,y,z} \Gamma(x,y)\Gamma(y,z) 1_{E'}(y,z) \ dxdz \\
    &\leq \paren{(1+c)^2 - 1 + \frac{\abs{E'}}{\abs{X}\abs{Z}}} p^2 \\
    &\leq 29cp^2
  \end{align*}
  by Lemmas~\ref{lem:small-restrict} and \ref{lem:two-densify-step-ss}. The result follows by combining
  the estimates.
\end{proof}

Finally we prove step 3 in the program, thereby completing the proof
of densification.

\begin{proof}
  [Proof of Lemma~\ref{lem:densification}]
  We prove the result when both edges $12$ and $23$ are sparse. When
  at least one of $12$ and $23$ is dense, the proof is analogous and
  easier.

  Lemma~\ref{lem:G-square-bound} implies that $(X_1, X_3)_G$ satisfies
  $\DISC(q_{13}, \frac{1}{2}, 3\epsilon + 35c)$, and hence it must also
  satisfy $\DISC(q_{13}, 1, 2\epsilon + 18c)$.

  Let $E' = \setcond{(x_1, x_3)}{\Gamma(x_1, X_2, x_3) >
    2p_{12}p_{23}}$. We have $\abs{E'} \leq 26c^2\abs{X_1}\abs{X_3}$ by Lemma~\ref{lem:two-densify-step-ss}. In the following integrals, let $\bx \colon V(H)
  \to V(\Gamma)$ vary uniformly over all compatible maps. Then
  \[
    0 \leq G(H) - 2p_{12}p_{23} G(H')
    \leq \int_{\bx} \Gamma\cond{H}{\bx}1_{E'}(x_1, x_2) \ d\bx
    \leq ((1+c)^{e(H^{\spa})} - 1 + 26c^2)p(H)
  \]
  by Lemma~\ref{lem:small-restrict}.
\end{proof}

\subsection{Counting $C_4$} \label{sec:C4}

With the tools of doubling and densification, we are now ready to
count embeddings in $G$. We start by showing how to count $C_4$, as
it is an important yet tricky step.

\begin{proposition}
  \label{prop:C4}
  Assume Setup~\ref{set:G} with $H = C_4$ and $k \geq 2$. Then
  \[
  \abs{G(C_4) - q(C_4)} \leq 100(c+\epsilon)^{1/2} p(C_4).
  \]
\end{proposition}

The constant 100 is unimportant. It can be obtained by unraveling
the calculations. We omit the details.

\begin{figure}[!htp]
  \centering
    \begin{tikzpicture}[scale=.5]
      \begin{scope}[shift={(0,5.5)}]
      \node[p] (a) at (-1,2) {};
      \node[p] (d) at (1,2) {};
      \node[p] (b) at (-1,0) {};
      \node[p] (c) at (1,0) {};
      \draw (a)--(b)--(c)--(d)--(a);
      \node at (0,-1.5) {\LargeUparrow};
    \end{scope}
    \begin{scope}[shift={(0,.5)}]
      \node[p] (a) at (1,2) {};
      \node[p] (b) at (-1,0) {};
      \node[p] (c) at (1,0) {};
      \draw[dense] (a)--(b);
      \draw (a)--(c);
      \draw (b)--(c);
      \node at (0,-1.5) {\LargeUparrow};
      \draw decorate [decoration={brace, amplitude=1cm}] {(-8,-4.5) -- (8,-4.5)};
    \end{scope}
    \begin{scope}[shift={(-6,-6)}]
      \node[p] (b) at (-1,0) {};
      \node[p] (c) at (1,0) {};
      \draw (b)--(c);
    \end{scope}
    \begin{scope}[shift={(0,-6)}]
      \node[p] (a) at (0,2) {};
      \node[p] (a2) at (1.4, 2) {};
      \node[p] (b) at (-1,0) {};
      \node[p] (c) at (1,0) {};
      \draw [dense] (a)--(b)--(a2);
      \draw (a)--(c)--(a2);
      \node at (0,-1.5) {\LargeUparrow};
    \end{scope}
    \begin{scope}[shift={(0,-11)}]
      \node[p] (a) at (0,1.6) {};
      \node[p] (a2) at (1.4, 1.6) {};
      \node[p] (b) at (-1,0) {};
      \draw [dense] (a)--(b)--(a2)--(a);
    \end{scope}
    \begin{scope}[shift={(6,-6)}]
      \node[p] (a) at (0,1.6) {};
      \node[p] (b) at (-1,0) {};
      \node[p] (c) at (1,0) {};
      \draw [dense] (a)--(b);
      \draw (a)--(c);
    \end{scope}
  \end{tikzpicture}
  \caption{The proof that $G(C_4) \approxmod_{c,\epsilon}^{p^4}
    q(C_4)$. The vertical arrows correspond to densification,
    doubling the top vertex and densification, respectively.}
  \label{fig:C4}
\end{figure}

Proposition~\ref{prop:C4} follows from repeated applications of
doubling (Lemma~\ref{lem:doubling-reduction}) and densification
(Corollary~\ref{cor:densification}). The chain of implications is
summarized in Figure~\ref{fig:C4} in the case when all four edges of
$C_4$ are sparse (the other cases are easier). In the figure, each graph represents a claim of the
form $G(H) \approxmod_{c,\epsilon}^{p(H)} q(H)$. The sparse and dense
edges are distinguished by thickness. The claim for the dense triangle
follows from the counting lemma for dense graphs
(Proposition~\ref{prop:dense-counting}) and the claim for the
rightmost graph follows from Lemma~\ref{lem:G-squared-d}.

\subsection{Finishing the proof of the counting lemma} \label{sec:finish}

Given a graph $H$, we can use the doubling lemma, Lemma~\ref{lem:doubling-reduction}, to
reduce the problem of counting $H$ in $G$ to the problem of counting
$H_{-a}$ in $G$, where $H_{-a}$ is $H$ with some vertex $a$ deleted,
provided we can also count $H_{a}$ and $H_{a,a\x 2}$. Suppose $a$ has
degree $t$ in $H$ and degree $t'$ in $H^\spa$. The graph $H_a$ is
isomorphic to some $K_{1,t}$. Since $K_{1,t}$ is a tree, we can count copies using
Proposition~\ref{prop:tree}, provided that the exponent of
$p$ in the jumbledness of $\Gamma$ satisfies $k \geq \frac{t' +
  1}{2}$. The following lemma shows that we can count embeddings of
$H_{a,a \x 2}$ as well.

\begin{lemma} \label{lem:K2t}
  Assume Setup~\ref{set:G} where $H = K_{2,t}$ with vertices
  $\set{a_1, a_2; b_1, \dots, b_t}$. Assume that the edges $a_ib_j$
  are sparse for $1 \leq j \leq t'$ and dense for $j > t'$.
  If $k \geq \frac{t'+2}{2}$, then
  \[
  G(H) \approxmod_{c,\epsilon}^{p(H)} q(H).
  \]
\end{lemma}

\begin{proof}
  When $t' = 0$, all edges of $H$ are dense, so the
  result follows from the dense counting lemma. So assume $t' \geq
  1$. First we apply densification as follows:
  \begin{center}
    \begin{tikzpicture}
      \begin{scope}
        \node[p] (a) at (0,1) {};
        \node[p] (b) at (1,1) {};
        \node[p] (0) at (2,0) {};
        \node[p] (1) at (1,0) {};
        \node[p] (2) at (0,0) {};
        \node[p] (3) at (-1,0) {};
        \draw [dense] (a)--(0)--(b);
        \draw (a)--(1)--(b)--(2)--(a)--(3)--(b);
        \node at (3.5,0.5) {\LargeLeftarrow};
      \end{scope}
      \begin{scope}[shift={(6,0)}]
        \node[p] (a) at (0,1) {};
        \node[p] (b) at (1,1) {};        \node[p] (0) at (2,0) {};
        \node[p] (1) at (1,0) {};
        \node[p] (2) at (0,0) {};
        \draw [dense] (a)--(0)--(b)--(a);
        \draw (a)--(1)--(b)--(2)--(a);
      \end{scope}
    \end{tikzpicture}
  \end{center}
  When $t' = 1$, we get a dense graph so we are done. Otherwise, the
  result follows by induction using doubling as shown below, where we
  use Propositions~\ref{prop:C4} and \ref{prop:tree} to count
  $C_4$ and $K_{1,2}$, respectively.
  
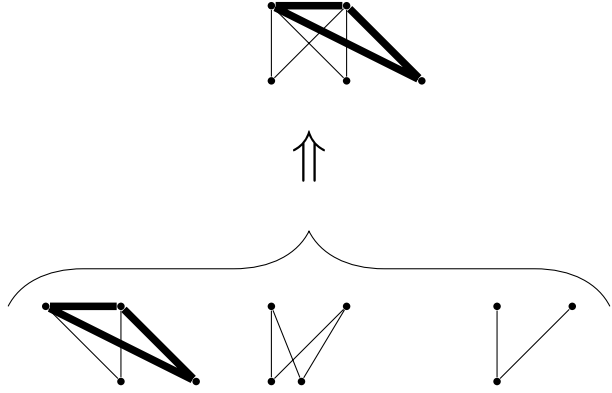
\begin{center}
    \begin{tikzpicture}
    \begin{scope}[shift={(-.5,0)}]
      \node[p] (a) at (0,1) {};
        \node[p] (b) at (1,1) {};
        \node[p] (0) at (2,0) {};
        \node[p] (1) at (1,0) {};
        \node[p] (2) at (0,0) {};
        \draw [dense] (a)--(0)--(b)--(a);
        \draw (a)--(1)--(b)--(2)--(a);
    \end{scope}
    \node at (0,-1) {\LargeUparrow};
    \draw decorate [decoration={brace, amplitude=1cm}] {(-4,-3) -- (4,-3)};
    \begin{scope}[shift={(-3.5,-4)}]
      \node[p] (a) at (0,1) {};
        \node[p] (b) at (1,1) {};
        \node[p] (0) at (2,0) {};
        \node[p] (1) at (1,0) {};
        \draw [dense] (a)--(0)--(b)--(a);
        \draw (a)--(1)--(b);
    \end{scope}
    \begin{scope}[shift={(-.5,-4)}]
      \node[p] (a) at (0,1) {};
        \node[p] (b) at (1,1) {};
        \node[p] (2) at (0,0) {};
        \node[p] (2') at (0.4,0) {};
        \draw (a)--(2)--(b)--(2')--(a);
    \end{scope}
    \begin{scope}[shift={(2.5,-4)}]
      \node[p] (a) at (0,1) {};
        \node[p] (b) at (1,1) {};
        \node[p] (2) at (0,0) {};
        \draw (b)--(2)--(a);
    \end{scope}
  \end{tikzpicture}
  \end{center}
\end{proof}

Once we can count $H_{a}$ and $H_{a,a \x 2}$, we obtain the following
reduction result via doubling.

\begin{lemma}
  \label{lem:counting-induct}
  Assume Setup~\ref{set:G}. Let $a$ be a vertex of $H$. If $k \geq \frac{d(L(H_{a,a\x 2}^\spa, H_{-a}^\spa)) +
      2}{2}$, then
  \[
  G(H) \approxmod_{c,\epsilon}^{p(H)} q(H_{a})G(H_{-a}).
  \]
\end{lemma}

The proof of the counting lemma follows once we keep track of the
requirements on $k$.

\begin{proof}
  [Proof of Proposition~\ref{prop:G-counting}]
  When $H$ has no sparse edges, the result follows from the dense
  counting lemma (Proposition~\ref{prop:dense-counting}). Otherwise,
  using Lemma~\ref{lem:counting-induct}, it remains to show that if $k
  \geq \min\left\{\frac{\Delta(L(H^\spa)) + 4}{2}, \frac{d(L(H^\spa)) +
      6}{2}\right\}$, then there exists some vertex $a$ of $H$ satisfying $k \geq \frac{d(L(H_{a,a\x 2}^\spa, H_{-a}^\spa)) +
    2}{2}$.  Actually, the hypothesis on $k$ is strong enough
  that any $a$ will do. Indeed, we have $\Delta(L(H^\spa)) + 2 \geq
  \Delta(L(H_{a \x 2}^\spa)) \geq d(L(H_{a,a\x2}^\spa,
  H_{-a}^\spa))$ since doubling $a$ increases the degree of every
  vertex by at most $1$. We also have $d(L(H^\spa)) \geq d(L(H_{a,a\x2}^\spa,
  H_{-a}^\spa)) - 4$ since every edge in $H_{-a}^\spa$ shares an
  endpoint with at most $4$ edges in $H_{a, a \x 2}^\spa$.
\end{proof}

\subsection{Tutorial: determining jumbledness requirements} \label{sec:tut-jumble}

The jumbledness requirements stated in our counting lemmas are often
not the best that come out of our proofs. We had to make a tradeoff
between strength and simplicity while formulating the results. In this
section, we give a short tutorial on finding the jumbledness
requirements needed for our counting lemma to work for any particular graph
$H$. These fine-tuned bounds can be extracted from a careful
examination of our proofs, with no new ideas introduced in this
section.

We work in a more general setting where we allow non-balanced
jumbledness conditions between vertex subsets of $\Gamma$. This
will arise naturallly in Section~\ref{sec:oneside} when we prove a one-sided counting lemma.

\begin{setup} \label{set:nonuniform} Let $\Gamma$ be a graph with
  vertex subsets $X_1, \dots, X_m$. Let $p, c \in (0,1]$. Let $H$ be a
  graph with vertex set $\set{1, \dots, m}$, with vertex $a$ assigned
  to $X_a$. For every edge $ab$ in $H$, one of the following two
  holds:
\begin{itemize}
\item $(X_a, X_b)_\Gamma$ is $(p, cp^{k_{ab}}
  \sqrt{\abs{X_a}\abs{X_b}})$-jumbled for some $k_{ab} \geq 1$, in which case we set $p_{ab} =
  p$ and say that $ab$
  is a \emph{sparse edge}, or
\item $(X_a, X_b)_\Gamma$ is a complete bipartite graph, in which case
  we set $p_{ab} = 1$ and say that $ab$ is a \emph{dense edge}.
\end{itemize}
Let $H^\spa$ denote the subgraph of $H$ consisting of sparse edges.

Let $\epsilon > 0$. Let $G$ be a weighted subgraph of $\Gamma$. For
every edge $ab \in E(H)$, assume that $(X_a, X_b)_G$ satisfies
$\DISC(q_{ab}, p_{ab}, \epsilon)$, where $0 \leq q_{ab} \leq p_{ab}$.
\end{setup}

In the figures in this section, we label the edges by the lower bounds
on $k_{ab}$ that are sufficient for the two-sided counting lemma to
hold. For instance, the figure below shows the jumbledness conditions
that are
sufficient for the triangle counting lemma\footnote{As mentioned in
the remark after Lemma~\ref{lem:doubling-reduction}, we do not actually
need DISC on $(X_a, X_b)_G$, since edge density is enough. We do not dwell
on this point in this section and instead focus on jumbledness requirements.}, namely $k_{ab}
\geq 3$, $k_{bc} \geq 2$, $k_{ac} \geq \frac{3}{2}$.
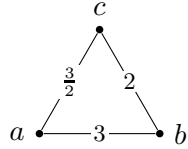
\begin{center}
  \begin{tikzpicture}[scale=.8]
    \node[p,label=left:$a$] (a) at (-1,0) {};
    \node[p,label=right:$b$] (b) at (1,0) {};
    \node[p,label=above:$c$] (c) at (0,1.73) {};
    \draw (a)--(b) node[midway,l]{3}
          (b)--(c) node[midway,l]{2}
          (c)--(a) node[midway,l]{$\frac32$};
  \end{tikzpicture}
\end{center}
Although we are primarily interested in embeddings of $H$ into $G$, we
need to consider partial embeddings where some of the edges of $H$
are allowed to embed into $\Gamma$. So we encounter three types of edges
of $H$, summarized in Table~\ref{tab:H-edges}. (Note that for dense edges $ab$,
$(X_a, X_b)_\Gamma$ is a complete bipartite graph, so such embeddings
are trivial and $ab$ can be ignored.)

\begin{table}[htbp]
  \caption{Types of edges in $H$.}
  \centering
  \label{tab:H-edges}
\begin{tabular}{cll}
  \toprule
  Figure & Name & Description \\
  \midrule
  \tikz \draw decorate[Gamma] {(0,0) -- (1,0) node[midway,l] {$\kappa$}}; &
  Jumbled edge & An edge to be embedded in $(X_a,X_b)_\Gamma$ with
  $p_{ab} = p$ and $k_{ab} \geq \kappa$.
  \\
  \tikz \draw[dense] (0,0) -- (1,0); &
  Dense edge & An edge to be embedded in $(X_a, X_b)_G$ with $p_{ab} =
  1$.
  \\
  \tikz \draw (0,0) -- (1,0) node[midway,l] {$\kappa$}; &
  Sparse edge & An edge to be embedded in $(X_a, X_b)_G$ with $p_{ab}
  = p$ and $k_{ab} \geq \kappa$.
  \\
  \bottomrule
\end{tabular}
\end{table}

Our counting lemma is proved through a number of reduction
procedures. At each step, we transform $H$ into one or more other
graphs $H'$. At the end of the reduction procedure, we should arrive
at a graph which only has dense edges. To determine the jumbledness
conditions required to count some $H$, we perform these reduction
steps and keep
track of the requirements at each step. We explain how to do this for
each reduction procedure.

\paragraph{Removing a jumbled edge.}
To remove a jumbled edge $ab$ from $H$, we need $k_{ab}$ to be at
least the average of the sparse degrees (i.e., counting both sparse
and jumbled edges) at the endpoints of $ab$, i.e.,
$k_{ab} \geq \frac{1}{2}(\deg_{H^\spa}(a) + \deg_{H^\spa} (b))$. See
Lemma~\ref{lem:Gamma-partial-one-edge}. For
example, $k_{ab} \geq \frac{5}{2}$ is sufficient to remove the edge
$ab$ in the graph below.
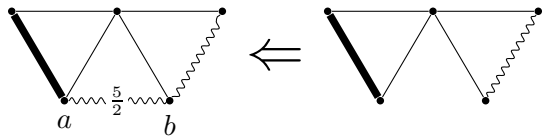
\[
  \begin{tikzpicture}[scale=.7]
    \begin{scope}
      \node[p,label=below:$a$] (a) at (-1,0) {};
      \node[p,label=below:$b$] (b) at (1,0) {};
      \node[p] (c) at (-2,1.7) {};
      \node[p] (d) at (0,1.7) {};
      \node[p] (e) at (2,1.7) {};
      \draw decorate[Gamma] {(a) -- (b) node[midway,l]{$\frac52$}
        --(e)};
      \draw (c)--(d)--(a) (e)--(d)--(b);
      \draw [dense] (a)--(c);
    \end{scope}

    \node at (3,.8) {\LargeLeftarrow};

    \begin{scope}[shift={(6,0)}]
      \node[p] (a) at (-1,0) {};
      \node[p] (b) at (1,0) {};
      \node[p] (c) at (-2,1.7) {};
      \node[p] (d) at (0,1.7) {};
      \node[p] (e) at (2,1.7) {};
      \draw decorate[Gamma] {(b)
        --(e)};
      \draw (c)--(d)--(a) (e)--(d)--(b);
      \draw [dense] (a)--(c);
    \end{scope}
  \end{tikzpicture}
\]
By removing jumbled edges one at a time, we can find conditions
that are sufficient for counting embeddings into $\Gamma$
(Proposition~\ref{prop:Gamma-counting}). The following figure shows
how this is done for a $4$-cycle.
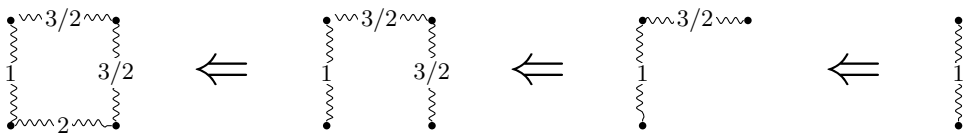
\[
  \begin{tikzpicture}[scale=.7]
    \begin{scope}
      \node[p] (a) at (1,-1) {};
      \node[p] (b) at (1,1) {};
      \node[p] (c) at (-1,1) {};
      \node[p] (d) at (-1,-1) {};
      \draw decorate[Gamma]{
        (a)--(b) node[midway,l] {$3/2$}
        --(c) node[midway,l] {$3/2$}
        --(d) node[midway,l] {$1$}
        --(a) node[midway,l] {$2$}
      };
      \node at (3,0) {\LargeLeftarrow};
    \end{scope}
    \begin{scope}[shift={(6,0)}]
      \node[p] (a) at (1,-1) {};
      \node[p] (b) at (1,1) {};
      \node[p] (c) at (-1,1) {};
      \node[p] (d) at (-1,-1) {};
      \draw decorate[Gamma]{
        (a)--(b) node[midway,l] {$3/2$}
        --(c) node[midway,l] {$3/2$}
        --(d) node[midway,l] {$1$}
      };
      \node at (3,0) {\LargeLeftarrow};
    \end{scope}
    \begin{scope}[shift={(12,0)}]
      \node[p] (b) at (1,1) {};
      \node[p] (c) at (-1,1) {};
      \node[p] (d) at (-1,-1) {};
      \draw decorate[Gamma]{
        (b)--(c) node[midway,l] {$3/2$}
        --(d) node[midway,l] {$1$}
      };
      \node at (3,0) {\LargeLeftarrow};
    \end{scope}
    \begin{scope}[shift={(18,0)}]
      \node[p] (c) at (-1,1) {};
      \node[p] (d) at (-1,-1) {};
      \draw decorate[Gamma]{
        (c)--(d) node[midway,l] {$1$}
      };
    \end{scope}
  \end{tikzpicture}
\]

\paragraph{Doubling}
The figure below illustrates doubling. If the jumbledness hypotheses
are sufficient to count the two graphs on the right, then they are
sufficient to count the original graph. The first graph is produced by
deleting all edges with $a$ as an endpoint, and the second graph is
produced by doubling $a$ and then, for all edges not adjacent to $a$,
deleting the dense edges and converting sparse edges to jumbled ones.
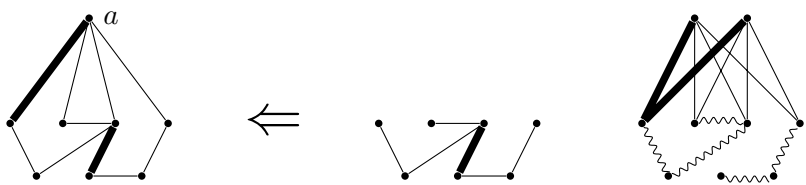
\[
    \begin{tikzpicture}[scale=.7]
      \begin{scope}[shift={(0,0)}]
      \node[p, label=right:$a$] (a) at (0,2) {};
      \node[p] (1) at (-1.5,0) {};
      \node[p] (2) at (-.5,0) {};
      \node[p] (3) at (.5,0) {};
      \node[p] (4) at (1.5,0) {};
      \node[p] (5) at (-1,-1) {};
      \node[p] (6) at (0,-1) {};
      \node[p] (7) at (1,-1) {};
      \draw[dense] (a)--(1);
      \draw (a)--(2) (a)--(3) (a)--(4);
      \draw[dense] (3)--(6);
      \draw (1)--(5)--(3)(6)--(7)--(4) (2)--(3);
    \end{scope}
    \node at (3.5,0) {\LargeLeftarrow};
    \begin{scope}[shift={(7,0)}]
      \node[p] (1) at (-1.5,0) {};
      \node[p] (2) at (-.5,0) {};
      \node[p] (3) at (.5,0) {};
      \node[p] (4) at (1.5,0) {};
      \node[p] (5) at (-1,-1) {};
      \node[p] (6) at (0,-1) {};
      \node[p] (7) at (1,-1) {};
      \draw[dense] (3)--(6);
      \draw (1)--(5)--(3)(6)--(7)--(4) (2)--(3);
    \end{scope}
    \begin{scope}[shift={(12,0)}]
      \node[p] (a) at (-.5,2) {};
      \node[p] (a2) at (.5,2) {};
      \node[p] (1) at (-1.5,0) {};
      \node[p] (2) at (-.5,0) {};
      \node[p] (3) at (.5,0) {};
      \node[p] (4) at (1.5,0) {};
      \node[p] (5) at (-1,-1) {};
      \node[p] (6) at (0,-1) {};
      \node[p] (7) at (1,-1) {};
            \draw[dense] (a)--(1)--(a2);
      \draw (a2)--(2)--(a)--(3)--(a2)--(4)--(a);
      \draw decorate [Gamma] { (1)--(5)--(3)(6)--(7)--(4) (2)--(3)};
    \end{scope}
  \end{tikzpicture}
\]

\paragraph{Densification}
To determine the jumbledness needed to perform densification, delete
all dense edges, transform all sparse edges into jumbled edges, and
use the earlier method to determine the jumbledness required to count
embeddings into $\Gamma$.
For example, the jumbledness on the left
figure below shows the requirements on $C_4$ needed to perform the
densification step. It may be the case that even stronger hypotheses
are needed to count the new graph (although for this example this is
not the case).
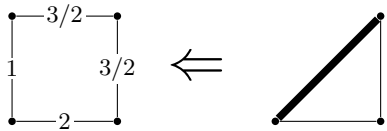
\[
  \begin{tikzpicture}[scale=.7]
    \begin{scope}
      \node[p] (a) at (1,-1) {};
      \node[p] (b) at (1,1) {};
      \node[p] (c) at (-1,1) {};
      \node[p] (d) at (-1,-1) {};
      \draw
      (a)--(b) node[midway,l] {$3/2$}
      --(c) node[midway,l] {$3/2$}
      --(d) node[midway,l] {$1$}
      --(a) node[midway,l] {$2$}
      ;
    \end{scope}
    \node at (2.5,0) {\LargeLeftarrow};
    \begin{scope}[shift={(5,0)}]
      \node[p] (a) at (1,-1) {};
      \node[p] (b) at (1,1) {};
      \node[p] (d) at (-1,-1) {};
      \draw[dense] (b)--(d);
      \draw (b)--(a)--(d);
    \end{scope}
  \end{tikzpicture}
\]

\paragraph{Trees}
To determine the jumbledness needed to count some tree $H$, delete all
dense edges in $H$ and transform all sparse edges into jumbled edges
and use the earlier method, removing one leaf at a time to determine
the jumbledness required to count embeddings into $\Gamma$
(Proposition \ref{prop:tree}).

\begin{example} [$C_4$]
  \label{ex:tut-C4}
  Let us check that the labeling of $C_4$
  in the densification paragraph gives sufficient jumbledness to count
  $C_4$. It remains to check that the jumbledness hypotheses are
  sufficient to count the triangle with a single edge. We can double
  the top vertex so that it remains to check the first graph below (the other graph produced from doubling is a single edge, which
  is trivial to count). We can remove the jumbled edge, and then
  perform densification to get a dense triangle, which we know how to count.
  
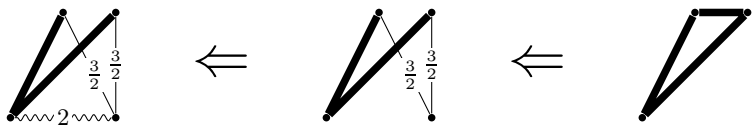
\[
    \begin{tikzpicture}[scale=.7]
      \begin{scope}
      \node[p] (a) at (1,-1) {};
      \node[p] (b) at (0,1) {};
      \node[p] (b') at (1,1) {};
      \node[p] (d) at (-1,-1) {};
      \draw (b)--(a) node[pos=.6,l]{$\frac32$}
              --(b') node[midway,l]{$\frac32$};

      \draw[dense] (b)--(d)--(b');
      \draw decorate[Gamma] {(a)--(d) node[midway,l]{$2$}};
      \node at (3,0) {\LargeLeftarrow};
      \end{scope}
            \begin{scope}[shift={(6,0)}]
      \node[p] (a) at (1,-1) {};
      \node[p] (b) at (0,1) {};
      \node[p] (b') at (1,1) {};
      \node[p] (d) at (-1,-1) {};
      \draw (b)--(a) node[pos=.6,l]{$\frac32$}
              --(b') node[midway,l]{$\frac32$};

      \draw[dense] (b)--(d)--(b');
      \node at (3,0) {\LargeLeftarrow};
      \end{scope}
      \begin{scope}[shift={(12,0)}]
      \node[p] (b) at (0,1) {};
      \node[p] (b') at (1,1) {};
      \node[p] (d) at (-1,-1) {};
      \draw[dense] (b)--(d)--(b')--(b);
      \end{scope}
    \end{tikzpicture}
  \]
\end{example}

\begin{example}[$K_3$]
  \label{ex:tut-K3}
  The following diagram illustrates the process of checking sufficient
  jumbledness hypotheses to count triangles (again, the first graph
  resulting from doubling is a single edge and is thus omitted from the
  figure). The sufficiency for $C_4$ follows from the previous example.
  
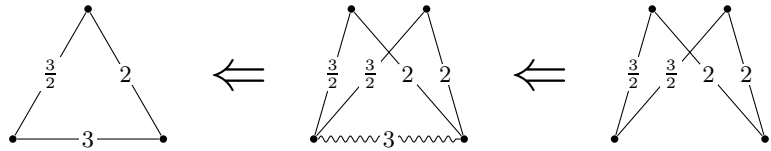
\[
  \begin{tikzpicture}[scale=1]
    \begin{scope}
      \node[p] (a) at (-1,0) {};
      \node[p] (b) at (1,0) {};
      \node[p] (c) at (0,1.73) {};
      \draw (a)--(b) node[midway,l]{3}
      (b)--(c) node[midway,l]{2}
      (c)--(a) node[midway,l]{$\frac32$};
      \node at (2,.8) {\LargeLeftarrow};
    \end{scope}
    \begin{scope}[shift={(4,0)}]
      \node[p] (a) at (-1,0) {};
      \node[p] (b) at (1,0) {};
      \node[p] (c) at (-.5,1.73) {};
            \node[p] (c') at (.5,1.73) {};
      \draw decorate[Gamma]{(a)--(b) node[midway,l]{3}};
      \draw   (c)--(b) node[midway,l]{2}
                 --(c') node[midway,l]{2}
                 --(a) node[midway,l]{$\frac32$}
                 --(c) node[midway,l]{$\frac32$};
      \node at (2,.8) {\LargeLeftarrow};
    \end{scope}
    \begin{scope}[shift={(8,0)}]
      \node[p] (a) at (-1,0) {};
      \node[p] (b) at (1,0) {};
      \node[p] (c) at (-.5,1.73) {};
      \node[p] (c') at (.5,1.73) {};
      \draw   (c)--(b) node[midway,l]{2}
                 --(c') node[midway,l]{2}
                 --(a) node[midway,l]{$\frac32$}
                 --(c) node[midway,l]{$\frac32$};
    \end{scope}
  \end{tikzpicture}
  \]
\end{example}

\begin{example}[$K_{2,t}$]
  \label{ex:tut-K2t}
  The following diagram shows sufficient jumbledness to count
  $K_{2,4}$.  The same pattern holds for $K_{2,t}$. The reduction
  procedure was given in the proof of Lemma~\ref{lem:K2t}. First we
  perform densification to the two leftmost edges, and then apply
  doubling to the remaining middle vertices in order from left to right.
  
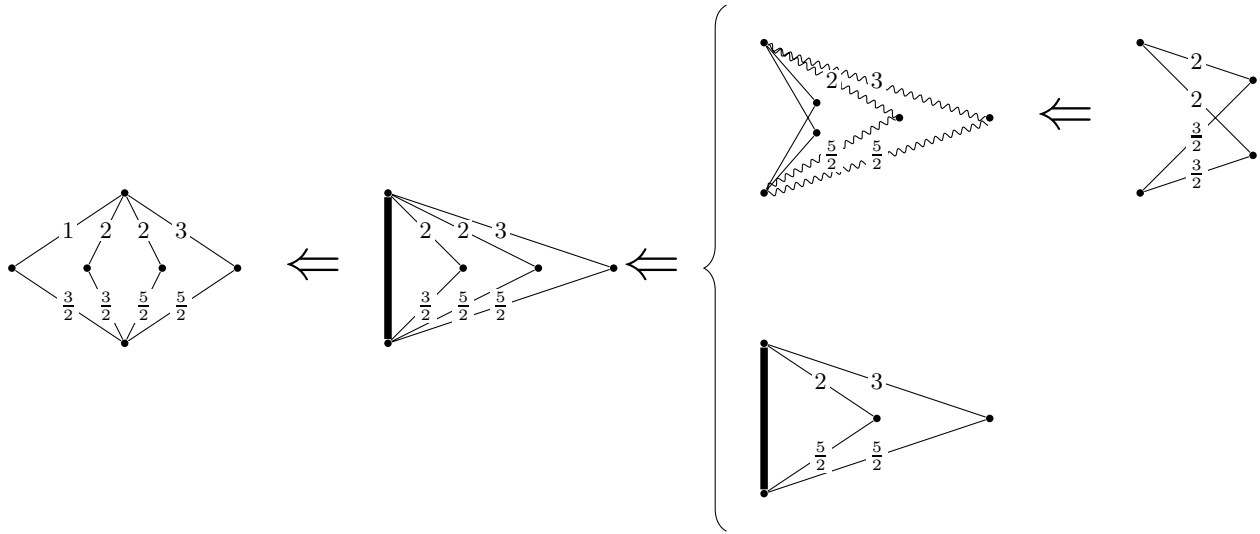
\[
    \begin{tikzpicture}
      \begin{scope}
      \node[p] (a) at (1.5,1) {};
      \node[p] (b) at (1.5,-1) {};
      \node[p] (1) at (0,0) {};
      \node[p] (2) at (1,0) {};
      \node[p] (3) at (2,0) {};
      \node[p] (4) at (3,0) {};
      \draw (a)--(1) node[midway,l] {$1$}
               --(b) node[midway,l] {$\frac32$}
               --(2) node[midway,l] {$\frac32$}
               --(a) node[midway,l] {$2$}
               --(3) node[midway,l] {$2$}
               --(b) node[midway,l] {$\frac52$}
               --(4) node[midway,l] {$\frac52$}
               --(a) node[midway,l] {$3$};
      \node at (4,0) {\LargeLeftarrow};
      \end{scope}
      \begin{scope}[shift={(5,0)}]
      \node[p] (a) at (0,1) {};
      \node[p] (b) at (0,-1) {};
      \node[p] (2) at (1,0) {};
      \node[p] (3) at (2,0) {};
      \node[p] (4) at (3,0) {};
      \draw[dense](a)--(b);
      \draw (b)--(2) node[midway,l] {$\frac32$}
               --(a) node[midway,l] {$2$}
               --(3) node[midway,l] {$2$}
               --(b) node[midway,l] {$\frac52$}
               --(4) node[midway,l] {$\frac52$}
               --(a) node[midway,l] {$3$};
      \node at (3.5,0) {\LargeLeftarrow};
      \draw decorate [decoration={brace, amplitude=.3cm}] {(4.5,-3.5) -- (4.5,3.5)};
      \end{scope}

      \begin{scope}[shift={(10,2)}]
      \node[p] (a) at (0,1) {};
      \node[p] (b) at (0,-1) {};
      \node[p] (2) at (.7,.2) {};
      \node[p] (2') at (.7,-.2) {};
      \node[p] (3) at (1.8,0) {};
      \node[p] (4) at (3,0) {};
      \draw (b)--(2) --(a) (b)--(2') --(a);
      \draw decorate[Gamma] {(a)                --(3) node[midway,l] {$2$}
        --(b) node[midway,l] {$\frac52$}
        --(4) node[midway,l] {$\frac52$}
        --(a) node[midway,l] {$3$}};
      \node at (4,0) {\LargeLeftarrow};
    \end{scope}

      \begin{scope}[shift={(15,2)}]
      \node[p] (a) at (0,1) {};
      \node[p] (b) at (0,-1) {};
      \node[p] (2) at (1.5,.5) {};
      \node[p] (2') at (1.5,-.5) {};
      \draw (b)--(2) node[midway,l] {$\frac32$}
               --(a) node[midway,l] {$2$}
            (b) --(2') node[midway,l] {$\frac32$}
               --(a) node[midway,l] {$2$};
      \end{scope}

      \begin{scope}[shift={(10,-2)}]
      \node[p] (a) at (0,1) {};
      \node[p] (b) at (0,-1) {};
      \node[p] (3) at (1.5,0) {};
      \node[p] (4) at (3,0) {};
      \draw[dense](a)--(b);
      \draw (a)--(3) node[midway,l] {$2$}
               --(b) node[midway,l] {$\frac52$}
               --(4) node[midway,l] {$\frac52$}
               --(a) node[midway,l] {$3$};
      \end{scope}

    \end{tikzpicture}
  \]
\end{example}

\begin{example}[$K_{1,2,2}$]
  \label{ex:tut-K122}
  The following diagram shows sufficient jumbledness to count
  $K_{1,2,2}$. This example will be used in the next section on
  inheriting regularity.
  
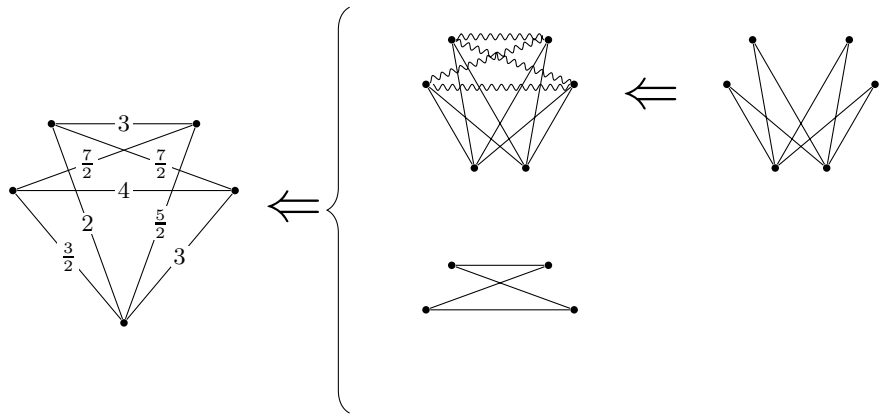
\[
    \begin{tikzpicture}
      \begin{scope}[scale=1.5]
      \node[p] (z) at (-90:1) {};
      \node[p] (x) at (130:1) {};
      \node[p] (x') at (170:1) {};
      \node[p] (y') at (50:1) {};
      \node[p] (y) at (10:1) {};
      \draw (x)--(y) node[pos=.6,l] {$\frac72$};
      \draw (y')--(x') node[pos=.6,l] {$\frac72$};
      \draw (x)--(y') node[midway,l] {$3$};
      \draw (x')--(y) node[midway,l] {$4$};
      \draw (z)--(x') node[midway,l] {$\frac32$};
      \draw (z)--(x) node[midway,l] {$2$};
      \draw (z)--(y') node[midway,l] {$\frac52$};
      \draw (z)--(y) node[midway,l] {$3$};
      \node at (1.5,0) {\LargeLeftarrow};
      \draw decorate [decoration={brace, amplitude=.3cm}] {(2,-1.8) -- (2,1.8)};
    \end{scope}

    \begin{scope}[shift={(5,1.5)}]
      \node[p] (z) at (-70:1) {};
      \node[p] (z') at (-110:1) {};
      \node[p] (x) at (130:1) {};
      \node[p] (x') at (170:1) {};
      \node[p] (y') at (50:1) {};
      \node[p] (y) at (10:1) {};
      \draw decorate[Gamma]{(x)--(y)--(x')--(y')--(x)};
      \draw (x)--(z)--(x')--(z')--(y)--(z)--(y')--(z')--(x);
      \node at (2,0) {\LargeLeftarrow};
    \end{scope}

    \begin{scope}[shift={(9,1.5)}]
      \node[p] (z) at (-70:1) {};
      \node[p] (z') at (-110:1) {};
      \node[p] (x) at (130:1) {};
      \node[p] (x') at (170:1) {};
      \node[p] (y') at (50:1) {};
      \node[p] (y) at (10:1) {};
      \draw (x)--(z)--(x')--(z')--(y)--(z)--(y')--(z')--(x);
    \end{scope}

    \begin{scope}[shift={(5,-1.5)}]
      \node[p] (x) at (130:1) {};
      \node[p] (x') at (170:1) {};
      \node[p] (y') at (50:1) {};
      \node[p] (y) at (10:1) {};
      \draw (x)--(y)--(x')--(y')--(x);
    \end{scope}
  \end{tikzpicture}
  \]
\end{example}

\section{Inheriting regularity} \label{sec:inherit}

Regularity is inherited on large subsets, in the sense
that if $(X, Y)_G$ satisfies $\DISC(q, 1, \epsilon)$, then for any $U
\subseteq X$ and $V \subseteq Y$, the induced pair $(U, V)_G$ satisfies
$\DISC(q, 1, \epsilon')$ with $\epsilon' =
\frac{\abs{X}\abs{Y}}{\abs{U}\abs{V}} \epsilon$. This is a trivial
consequence of the definition of discrepancy, and the change in
$\epsilon$ comes from rescaling the measures $dx$ and $dy$ after
restricting the uniform distribution to a subset. The loss in
$\epsilon$ is a constant factor as long as $\frac{\abs{U}}{\abs{X}}$
and $\frac{\abs{V}}{\abs{Y}}$ are bounded from below. So if $G$ is a
dense tripartite graph with vertex subsets $X, Y, Z$, with each pair
being dense and regular, then we expect that for most vertices
$z \in Z$, its neighborhoods $N_X(z)$ and $N_Y(z)$ are large, and hence
they induce regular pairs with only a constant factor loss in the
discrepancy parameter $\epsilon$.

The above argument does not hold in sparse pseudorandom graphs.
It is still true that if $(X, Y)_G$ satisfies $\DISC(q, p, \epsilon)$
then for any $U \subseteq X$ and $V \subseteq Y$ the induced pair $(U,
V)_G$ satisfies $\DISC(q, 1, \epsilon')$ with $\epsilon' =
\frac{\abs{X}\abs{Y}}{\abs{U}\abs{V}} \epsilon$. However, in the tripartite
setup from the previous paragraph, we expect most $N_X(z)$ to have
size on the order of $p\abs{X}$. So the naive approach shows that most
$z \in Z$ induce a bipartite graph satisfying $\DISC(q,p,\epsilon')$
where $\epsilon'$ is on the order of $\frac{\epsilon}{p^2}$. This is
undesirable, as we do not want $\epsilon$ to depend on $p$.

It turns out that for most $z \in Z$, the bipartite graph induced by
the neighborhoods satisfies $\DISC(q,p,\epsilon')$ for some $\epsilon'$
depending on $\epsilon$ but not $p$. In this section we prove this
fact using the counting lemma developed earlier in the paper. We
recall the statement from the introduction.

\vspace{2mm}\noindent
{\bf Proposition
  \ref{prop:inheritintro}} {\it
  For any $\alpha > 0$, $\xi > 0$ and $\epsilon' > 0$, there exists $c > 0$ and
  $\epsilon > 0$ of size at least polynomial in $\alpha, \xi, \epsilon'$ such that the following holds.

  Let $p \in (0,1]$ and $q_{XY}, q_{XZ}, q_{YZ} \in [\alpha p,p]$. Let
  $\Gamma$ be a tripartite graph with vertex subsets $X$, $Y$ and
  $Z$ and $G$ be a subgraph of $\Gamma$. Suppose that
  \begin{itemize}
  \item $(X,Y)_\Gamma$ is $(p, cp^4\sqrt{\abs{X}\abs{Y}})$-jumbled and
    $(X,Y)_G$ satisfies $\DISC(q_{XY}, p, \epsilon)$; and
  \item $(X,Z)_\Gamma$ is $(p, cp^2\sqrt{\abs{X}\abs{Z}})$-jumbled and
    $(X,Z)_G$ satisfies $\DISC(q_{XZ}, p, \epsilon)$; and
  \item $(Y,Z)_\Gamma$ is $(p, cp^3\sqrt{\abs{Y}\abs{Z}})$-jumbled and
    $(Y,Z)_G$ satisfies $\DISC(q_{YZ}, p, \epsilon)$.
  \end{itemize}
  Then at least $(1 - \xi)\abs{Z}$ vertices $z \in Z$ have the
  property that $\abs{N_X(z)} \geq (1 - \xi)q_{XZ}\abs{X}$,
  $\abs{N_Y(z)} \geq (1 - \xi) q_{YZ}\abs{Y}$, and $(N_X(z),
  N_Y(z))_G$ satisfies $\DISC(q_{XY}, p, \epsilon')$.
}
\vspace{2mm}

The idea of the proof is to first show that a bound on the $K_{2,2}$ count
implies DISC and then to use the $K_{1,2,2}$ count to bound the $K_{2,2}$
count between neighborhoods.

We also state a version where only one side gets smaller. While the previous proposition is sufficient for embedding cliques, this second version will be needed for embedding general graphs $H$.

\begin{proposition}
  \label{prop:inherit2}
  For any $\alpha > 0$, $\xi > 0$ and $\epsilon' > 0$, there exists $c
  > 0$ and $\epsilon > 0$ of size at least polynomial in $\alpha, \xi,
  \epsilon'$ such that the following holds.

  Let $p \in (0,1]$ and $q_{XY}, q_{XZ} \in [\alpha p,p]$. Let
  $\Gamma$ be a tripartite graph with vertex subsets $X$, $Y$ and
  $Z$ and $G$ be a subgraph of $\Gamma$. Suppose that
  \begin{itemize}
  \item $(X,Y)_\Gamma$ is $(p, cp^{5/2}\sqrt{\abs{X}\abs{Y}})$-jumbled and
    $(X,Y)_G$ satisfies $\DISC(q_{XY}, p, \epsilon)$; and
  \item $(X,Z)_\Gamma$ is $(p, cp^{3/2}\sqrt{\abs{X}\abs{Z}})$-jumbled and
    $(X,Z)_G$ satisfies $\DISC(q_{XZ}, p, \epsilon)$.
  \end{itemize}
  Then at least $(1 - \xi)\abs{Z}$ vertices $z \in Z$ have the
  property that $\abs{N_X(z)} \geq (1 - \xi)q_{XZ}\abs{X}$ and
  $(N_X(z), Y)_G$ satisfies $\DISC(q_{XY}, p, \epsilon')$.
\end{proposition}

\subsection{$C_4$  implies DISC} \label{sec:C4-DISC}

From our counting lemma we already know that if $G$ is a subgraph of a
sufficiently jumbled graph with vertex subsets $X$ and $Y$ such that
$(X,Y)_G$ satisfies $\DISC(q,p,\epsilon)$, then the number
of $K_{2,2}$ in $G$ across $X$ and $Y$ is roughly
$q^4\abs{X}^2\abs{Y}^2$. In this section, we show that the converse is
true, that the $K_{2,2}$ count implies discrepancy, even without any
jumbledness hypotheses.

In what follows, for any function $f \colon X \x Y \to \RR$, we write
\[
f(K_{s,t}) = \int_{\substack{x_1,\dots, x_s \in X \\ y_1, \dots, y_t
    \in Y}} \prod_{i=1}^s \prod_{j=1}^t f(x_i,y_j) \ dx_1 \cdots dx_s
  dy_1 \cdots dy_t.
\]

The following lemma shows that a bound on the ``de-meaned'' $C_4$-count implies discrepancy.

\begin{lemma}
  \label{lem:C4-DISC}
  Let $G$ be a bipartite graph between vertex sets $X$ and $Y$. Let $0
  \leq q \leq p \leq 1$ and $\epsilon > 0$. Define $f \colon X \x Y
  \to \RR$ by $f(x,y) = G(x,y) - q$. If $ f(K_{2,2}) \leq \epsilon^4
  p^4$ then $(X,Y)_G$ satisfies $\DISC(q,p,\epsilon)$.
\end{lemma}

\begin{proof}
  Let $u \colon X \to [0,1]$ and $v \colon Y \to [0,1]$ be any
  functions. Applying the Cauchy-Schwarz inequality twice, we have
  {\allowdisplaybreaks
  \begin{align*}
    \paren{\int_{x \in X} \int_{y \in Y} f(x,y)u(x)v(y) \ dydx}^4
    &
    \leq
    \paren{\int_{x \in X} \paren{\int_{y \in Y} f(x,y)u(x)v(y) \ dy}^2
      \ dx}^2
    \\
    &=
    \paren{\int_{x \in X} u(x)^2 \paren{\int_{y \in Y} f(x,y)v(y) \ dy}^2
      \ dx}^2
    \\
    &\leq
    \paren{\int_{x \in X} \paren{\int_{y \in Y} f(x,y)v(y) \ dy}^2
      \ dx}^2
    \\
    &=
    \paren{\int_{x \in X} \int_{y, y' \in Y} f(x,y)f(x,y')v(y)v(y') \ dydy'
      dx}^2
    \\
    &\leq
    \int_{y, y' \in Y} \paren{\int_{x \in X} f(x,y)f(x,y')v(y)v(y') \
      dx}^2 \ dydy'
    \\
    &=
    \int_{y, y' \in Y} v(y)^2v(y')^2 \paren{\int_{x \in X}
      f(x,y)f(x,y') \
      dx}^2 \ dydy'
    \\
    &\leq
    \int_{y, y' \in Y} \paren{\int_{x \in X} f(x,y)f(x,y') \
      dx}^2 \ dydy'
    \\
    &=
    \int_{y, y' \in Y} \int_{x,x' \in X} f(x,y)f(x,y')f(x',y)f(x',y')
    \ dxdx'dydy'
    \\
    &= f(K_{2,2})
    \\
    &
    \leq \epsilon^4p^4.
  \end{align*}
  }\noindent
  Thus
  \[
  \abs{\int_{x \in X} \int_{y \in Y} (G(x,y) - q)u(x)v(y) \ dydx} \leq
  \epsilon p.
  \]
  Hence $(X,Y)_G$ satisfies $\DISC(q,p,\epsilon)$.
\end{proof}

\begin{lemma}
  \label{lem:C4-DISC-subgraph}
  Let $G$ be a bipartite graph between vertex sets $X$ and $Y$. Let $0
  \leq q \leq p \leq 1$ and $\epsilon > 0$. Let $U
  \subseteq X$ and $V \subseteq Y$. Let $\mu = \frac{\abs{U}}{\abs{X}}$
  and $\nu = \frac{\abs{V}}{\abs{Y}}$. Define $f \colon X \x Y \to \RR$ by
  $f(x,y) = (G(x,y) - q)1_U(x)1_V(y)$. If $
  f(K_{2,2}) \leq \epsilon^4 p^4 \mu^2\nu^2$, then $(U,V)_G$ satisfies
  $\DISC(q,p,\epsilon)$.
\end{lemma}

\begin{proof}
This lemma is equivalent to Lemma~\ref{lem:C4-DISC} after
appropriate rescaling of the measures $dx$ and $dy$.
\end{proof}

The above lemmas are sufficient for proving inheritance of regularity,
so that the reader may now skip to the next subsection. The rest of
this subsection contains a proof that an upper bound on the actual
$C_4$ count implies discrepancy, a result of independent interest
which is discussed further in Section~\ref{sec:rel-quasi} on relative quasirandomness.

\begin{proposition}
  \label{prop:C4-DISC}
  Let $G$ be a bipartite graph between vertex sets $X$ and $Y$. Let
  $0 \leq q \leq 1$ and $\epsilon > 0$. Suppose $G(K_{1,1})
  \geq (1 - \epsilon)q$ and $G(K_{2,2}) \leq (1+\epsilon)^4 q^4$, then
  $(X,Y)_G$ satisfies $\DISC(q,q, 4 \epsilon^{1/36})$.
\end{proposition}

The hypotheses in Proposition~\ref{prop:C4-DISC} actually imply
two-sided bounds on $G(K_{1,1})$, $G(K_{1,2})$, $G(K_{2,1})$, and
$G(K_{2,2})$, by the following lemma.

\begin{lemma}
  \label{lem:K11-K12-K22}
  Let $G$ be a bipartite graph between vertex sets $X$ and $Y$ and $f \colon X \x Y \to \RR$ be any function. Then
  $f(K_{1,1})^4 \leq f(K_{1,2})^2 \leq f(K_{2,2})$.
\end{lemma}

\begin{proof}
  The result follows from two applications of the Cauchy-Schwarz
  inequality.
  {\allowdisplaybreaks
  \begin{align*}
    f(K_{2,2})
    &
    = \int_{y, y' \in Y} \int_{x,x' \in X} f(x,y)f(x,y')f(x',y)f(x',y')
    \ dxdx'dydy'
    \\
    &=
    \int_{y, y' \in Y} \paren{\int_{x \in X} f(x,y)f(x,y') \
      dx}^2 \ dydy'
    \\
    &\geq
    \paren{\int_{y, y' \in Y} \int_{x \in X} f(x,y)f(x,y') \
      dx \ dydy'}^2
    \\
    &= f(K_{1,2})^2
    \\
    &=\paren{\int_{x \in X} \paren{\int_{y \in Y} f(x,y) \ dy}^2
      \ dx}^2
    \\
    &\geq \paren{\int_{x \in X} \int_{y \in Y} f(x,y) \ dy
      \ dx}^4
    \\
    &= f(K_{1,1})^4.
  \end{align*}
  }
\end{proof}

\newcommand{\pathbipartite}[1]
{
  \begin{tikzpicture}[scale =.3, baseline=-4pt,font=\tiny, label distance = -3pt]
    \path[use as bounding box] (-1.2,-.8) rectangle (1.2,.8);
    \node[p] (x) at (-1,.5) {};
    \node[p] (x') at (-1,-.5) {};
    \node[p] (y) at (1,.5) {};
    \node[p] (y') at (1,-.5) {};
   \draw #1;
    \end{tikzpicture}
}
\newcommand{\Pxyxy}{\pathbipartite{(x)--(y)--(x')--(y')}}

A bound on $K_{1,2}$ is a second moment bound on the degree
distribution, so we can bound the number of vertices of low degree
using Chebyshev's inequality, as done in the next lemma. Recall the
notation $G(x,S) = \int_{y \in Y} G(x,y)1_S(y) \ dy$ for $S \subseteq
Y$ as the normalized degree.

\begin{lemma}
  \label{lem:DISC-deviate}
  Let $G$ be a bipartite graph between vertex sets $X$ and $Y$. Let
  $0 \leq q \leq 1$ and $\epsilon > 0$. Suppose $G(K_{1,1}) \geq (1 -
  \epsilon)q$ and $G(K_{1,2}) \leq (1+\epsilon)^2 q^2$. Let $X '
  =\setcond{x \in X}{G(x,Y) < (1 - 2\epsilon^{1/3})q}$. Then $\abs{X'}
  \leq 2\epsilon^{1/3}\abs{X}$.
\end{lemma}

\begin{proof}
  We have
  \begin{align*}
    \frac{\abs{X'}}{\abs{X}} \paren{2\epsilon^{1/3} q}^2
    &\leq \int_{x
      \in X} \paren{G(x,Y) - q}^2 \ dx
    \\
    & = G(K_{1,2}) - 2qG(K_{1,1}) + q^2
    \\
    &
    \leq (1+\epsilon)^2 q^2 - 2(1-\epsilon) q^2 + q^2
    \\
    & \leq 5\epsilon q^2.
  \end{align*}
  Thus $\abs{X'} \leq \frac54 \epsilon^{1/3} \abs{X}$.
\end{proof}

We write
\[
G\paren{\Pxyxy} = \int_{\substack{x,x' \in X \\ y,y' \in Y}} G(x,y)G(x',y)G(x',y') \ dxdx'dydy'.
\]
The next lemma proves a lower bound on $G\paren{\Pxyxy}$ by discarding
vertices of low degree.

\begin{lemma}
  \label{lem:P3}
  Let $G$ be a bipartite graph between vertex sets $X$ and $Y$. Let
  $0 \leq q \leq 1$ and $\epsilon > 0$. Suppose $G(K_{1,1})
  \geq (1 - \epsilon)q$, $G(K_{1,2}) \leq (1+\epsilon)^2 q^2$ and
  $G(K_{2,1}) \leq (1+\epsilon)^2 q^2$. Then $G\paren{\Pxyxy} \geq
  (1 - 14\epsilon^{1/9})q^3$.
\end{lemma}

\begin{proof}
  Let
  \[
  X' = \setcond{x \in X}{G(x,Y) < (1-2\epsilon^{1/3}) q}.
  \]
  Let $G'$ denote the subgraph of $G$ where we remove all edges with
  an endpoint in $X'$. Then $G'(K_{2,1}) \leq G(K_{2,1}) \leq
  (1+\epsilon)^2q^2$ and, by Lemma~\ref{lem:DISC-deviate}
  \[
  G'(K_{1,1}) \geq \frac{\abs{X\setminus X'}}{\abs{X}} (1 -
  2\epsilon^{1/3}) q \geq (1- 2\epsilon^{1/3})^2 q \geq (1-4\epsilon^{1/3})q.
  \]
  Let
  \[
  Y' = \setcond{y \in Y}{G(X\setminus X', y) < (1 - 4\epsilon^{1/9}) q}.
  \]
  So $\abs{Y'} \leq 4\epsilon^{1/9}$ by applying
  Lemma~\ref{lem:DISC-deviate} again. Restricting to paths
  with vertices in $X\setminus X', Y\setminus Y', X \setminus X', Y$,
  we find that
  \begin{align*}
  G\paren{\Pxyxy} &\geq \frac{\abs{Y\setminus
      Y'}}{\abs{Y}} \paren{\min_{y \in Y\setminus Y'}
      G(X\setminus X', y)}^2 \paren{\min_{x \in X\setminus X'} G(x,Y)}
    \\
    &\geq (1-4\epsilon^{1/9})^3(1-2\epsilon^{1/3})q^3
    \\
    & \geq (1 - 14\epsilon^{1/9})q^3.
  \end{align*}
\end{proof}

The above argument can be modified to show that a bound on $K_{1,2}$
implies one-sided counting for trees. We state the generalization and
omit the proof.

\begin{proposition}
  \label{prop:K12-tree}
  Let $H$ be a tree on vertices $\set{1, 2, \dots, m}$. For every
  $\theta > 0$ there exists $\epsilon > 0$ of size polynomial in $\theta$
  so that the following holds.

  Let $G$ be a weighted graph with vertex subsets $X_1, \dots, X_m$. For
  every edge $ab$ of $H$, assume there is some $q_{ab} \in [0,1]$ so that
  the bipartite graph $(X_a, X_b)_G$ satisfies $G(K_{1,1}) \geq
  (1-\epsilon)q_{ab}$, $G(K_{1,2}) \leq (1+\epsilon)^2 q_{ab}^2$ and
  $G(K_{2,1}) \leq (1+\epsilon)^2 q_{ab}^2$. Then $G(H) \geq (1 - \theta)q(H)$.
\end{proposition}

\begin{proof}
  [Proof of Proposition~\ref{prop:C4-DISC}]
  Using Lemma~\ref{lem:K11-K12-K22}, we have $G(K_{1,2}) \leq
  (1+\epsilon)^2 q^2$,  $G(K_{2,1}) \leq (1+\epsilon)^2 q^2$ and
  $G(K_{1,1}) \leq (1+\epsilon)q$. Let $f(x,y) = G(x,y) - q$. Applying
  Lemma~\ref{lem:P3}, we have
  \begin{align*}
    f(K_{2,2})
    &= G(K_{2,2}) - 4qG\paren{\Pxyxy} + 2q^2 G(K_{1,1})^2 +
    4q^2G(K_{1,2}) - 4q^3G(K_{1,1}) + q^4
    \\
    & \leq (1+\epsilon)^4 q^4 - 4(1-14\epsilon^{1/9})q^4 +
    2(1+\epsilon)^2 q^4 + 4(1+\epsilon)^2 q^4 - 4(1-\epsilon) q^4 +
    q^4
    \\
    &\leq 100\epsilon^{1/9}q^4.
  \end{align*}
  Thus, by Lemma~\ref{lem:C4-DISC}, $(X,Y)_G$ satisfies
  $\DISC(q,q,4\epsilon^{1/36})$.
\end{proof}

\subsection{$K_{1,2,2}$ implies inheritance}

We now prove Propositions~\ref{prop:inheritintro} and \ref{prop:inherit2}
using Lemma~\ref{lem:C4-DISC-subgraph}.

\newcommand{\graphinherit}[1]
{
  \begin{tikzpicture}[scale =.45, baseline=-4pt,font=\tiny, label distance = -3pt]
    \path[use as bounding box] (-1.2,-1.2) rectangle (1.2,1);
    \node[p] (x) at (-90:1) {};
    \node[p] (y) at (130:1) {};
    \node[p] (y1) at (170:1) {};
    \node[p] (z) at (50:1) {};
    \node[p] (z1) at (10:1) {};
   \draw #1;
    \end{tikzpicture}
}
\newcommand{\graphinheritA}{\graphinherit{(x)--(y)--(z)--(x)--(y1)--(z1)--(x)
    (y1)--(z) (y)--(z1)}}
\newcommand{\graphinheritB}{\graphinherit{(x)--(y) (z)--(x)--(y1)--(z1)--(x)
    (y1)--(z) (y)--(z1)}}
\newcommand{\graphinheritC}{\graphinherit{(x)--(y)--(z)--(x)--(y1)--(z1)--(x)}}
\newcommand{\graphinheritD}{\graphinherit{(x)--(y) (z)--(x)--(y1)--(z1)--(x)
    (y1)--(z)}}
\newcommand{\graphinheritE}{\graphinherit{(x)--(y1)--(z1)--(x)
    (y)--(x)--(z)}}
\newcommand{\graphinheritF}{\graphinherit{(y)--(x)--(z)
    (y1)--(x)--(z1)}}

\begin{proof}
  [Proof of Proposition~\ref{prop:inheritintro}] First we show that only a
  small fraction of vertices in $Z$ have very few neighbors in $X$ and
  $Y$. Let $Z_1$ be the set of all vertices in $Z$ with fewer than $(1
  - \xi)q_{XZ}\abs{X}$ neighbors in $X$. Applying discrepancy to $(X,
  Z_1)$ yields $\xi q_{XZ} \abs{Z_1} \leq \epsilon p \abs{Z}$. If we assume
  that $\epsilon \leq \frac13 \alpha \xi^2$, we have $\abs{Z_1} \leq
  \frac{\epsilon p}{\xi q_{XZ}} \abs{Z} \leq
  \frac{\xi}{3}\abs{Z}$. Similarly let $Z_2$ be the set of all
  vertices in $Z$ with fewer than $(1 - \xi)q_{YZ} \abs{Y}$ neighbors
  in $Y$, so that $\abs{Z_2} \leq \frac{\xi}{3}\abs{Z}$ as well.

  Define $f \colon V(G) \x V(G) \to \RR$ to be a function which agrees
  with $G$ on pairs $(X,Z)$ and $(Y,Z)$, and agrees with $G - q_{XY}$ on
  $(X,Y)$. Let us assign each vertex of $K_{2,2,1}$ to one of $\set{X,
    Y, Z}$ as follows (two vertices are assigned to each of $X$ and $Y$).
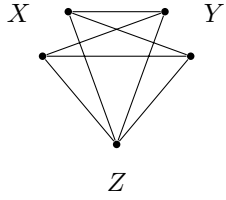
  \begin{center}
    \begin{tikzpicture}[scale =1, baseline=-4pt,font=\small, label distance = -3pt]
      \node[p] (z) at (-90:1) {};
      \node[p] (x) at (130:1) {};
      \node[p] (x') at (170:1) {};
      \node[p] (y') at (50:1) {};
      \node[p] (y) at (10:1) {};
      \node at (150:1.5) {$X$};
      \node at (30:1.5) {$Y$};
      \node at (-90:1.5) {$Z$};
      \draw (x)--(y)--(x')--(y')--(x)--(z)--(x') (y)--(z)--(y');
    \end{tikzpicture}
  \end{center}
  The stated jumbledness hypotheses suffice for counting $K_{1,2,2}$
  and its subgraphs; we refer to the tutorial in
  Section~\ref{sec:tut-jumble} for an explanation.

  By expanding all the $(G(x,y) - q_{XY})$ factors and using our
  counting lemma, we get
  \begin{align*}
    f\paren{\graphinheritA}
    &= G\paren{\graphinheritA}
    - 4 q_{XY} G\paren{\graphinheritB}
    + 2 q_{XY}^2 G\paren{\graphinheritC}
    \\
    &\qquad
    + 4 q_{XY}^2 G\paren{\graphinheritD}
    - 4 q_{XY}^3 G\paren{\graphinheritE}
    + q_{XY}^4 G\paren{\graphinheritF}
    \\
    &\approxmod_{c,\epsilon}^{p^8}
     q\paren{\graphinheritA}
    - 4 q_{XY} q\paren{\graphinheritB}
    + 2 q_{XY}^2 q\paren{\graphinheritC}
    \\
    &\qquad
    + 4 q_{XY}^2 q\paren{\graphinheritD}
    - 4 q_{XY}^3 q\paren{\graphinheritE}
    + q_{XY}^4 q\paren{\graphinheritF}
    \\
    &
    =0.
  \end{align*}
  Therefore, by choosing $\epsilon$ and $c$ to be sufficiently
  small (but polynomial in $\xi, \alpha, \epsilon'$), we can guarantee that
  \[
    f\paren{\graphinheritA} \leq \frac{1}{3}\xi (1-\xi)^4 \epsilon'^4 \alpha^4 p^8.
  \]
  Let $K_{2,2}$ denote the subgraph of the above $K_{1,2,2}$ that gets
  mapped between $X$ and $Y$. For each $z \in Z$, let $f_z \colon X \x
  Y \to \RR$ be defined by $(G(x,y) - q_{XY})1_{N_X(z)}(x) 1_{N_Y(z)}
  (y)$. We have
  \[
  f\paren{\graphinheritA}
  = \int_{z \in Z} f_z(K_{2,2}) \ dz.
  \]
  By Lemma~\ref{lem:K11-K12-K22}, $f_z(K_{2,2}) \geq 0$ for all $z \in
  Z$. Let $Z_3$ be the set of vertices $z$ in $Z$ such that
  $f_z(K_{2,2}) > \epsilon'^4 (1-\xi)^4 \alpha^4 p^8$. Then
  $\abs{Z_3} \leq \frac{\xi}{3}\abs{Z}$.

  Let $Z' = Z \setminus (Z_1 \cup Z_2 \cup Z_3)$. So $\abs{Z'} \geq (1
  - \xi)\abs{Z}$. Furthermore, for any $z \in Z_1$,
  \[
  f_z(K_{2,2}) \leq \epsilon'^4 (1-\xi)^4 \alpha^4p^8 \leq
   \epsilon'^4 (1-\xi)^4 p^4 q_{XZ}^2q_{YZ}^2
   \leq \epsilon'^4 p^4 \paren{\frac{\abs{N_X(z)}}{\abs{X}}}^2\paren{\frac{\abs{N_Y(z)}}{\abs{Y}}}^2.
  \]
  It follows by Lemma~\ref{lem:C4-DISC-subgraph} that $(N_X(z),
  N_Y(z))_G$ satisfies $\DISC(q_{XY}, p, \epsilon')$.
\end{proof}

\begin{proof}[Proof of Proposition~\ref{prop:inherit2}]
  The proof is essentially the same as the proof of
  Proposition~\ref{prop:inheritintro} with the difference being that we
  now use the following graph. We omit the details.
    
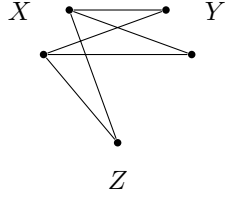
\begin{center}
    \begin{tikzpicture}[scale =1, baseline=-4pt,font=\small, label distance = -3pt]
      \node[p] (z) at (-90:1) {};
      \node[p] (x) at (130:1) {};
      \node[p] (x1) at (170:1) {};
      \node[p] (y) at (50:1) {};
      \node[p] (y1) at (10:1) {};
      \node at (150:1.5) {$X$};
      \node at (30:1.5) {$Y$};
      \node at (-90:1.5) {$Z$};
      \draw (z)--(x)--(y)(z)--(x1)--(y1)(z)
      (x1)--(y) (x)--(y1);
    \end{tikzpicture}
  \end{center}
\end{proof}

\section{One-sided counting} \label{sec:oneside}

We are now in a position to prove Theorem \ref{thm:onesidedintro},
which we now recall.

\vspace{2mm} \noindent
{\bf Theorem \ref{thm:onesidedintro}} {\it
For every fixed graph $H$ on vertex set $\{1,2, \dots, m\}$ and every
$\alpha, \theta > 0$, there exist constants $c > 0$ and $\epsilon > 0$ such that the following holds.

Let $\Gamma$ be a graph with vertex subsets $X_1, \dots, X_m$ and suppose that the bipartite graph $(X_i, X_j)_\Gamma$ is $(p, cp^{d_2(H) + 3} \sqrt{\abs{X_i}\abs{X_j}})$-jumbled for every $i < j$ with $ij \in E(H)$. Let $G$ be a subgraph of $\Gamma$, with the vertex $i$ of $H$ assigned to the vertex subset $X_i$ of $G$. For each edge $ij$ of $H$, assume that $(X_i, X_j)_G$ satisfies $\DISC(q_{ij}, p, \epsilon)$, where $\alpha p \leq q_{ij} \leq p$. Then $G(H) \geq (1 - \theta) q(H)$.}

\vspace{2mm}

The idea is to embed vertices of $H$ one at a time. At each step, the
set of potential targets for each unembedded vertex shrinks, but we can
choose our embedding so that it doesn't shrink too much and
discrepancy is inherited.

\begin{proof}
Suppose that $v_1, v_2, \dots, v_m$ is an ordering of the vertices of $H$ which yields the $2$-degeneracy $d_2(H)$ and that the vertex $v_i$ is to be embedded in $X_i$. Let $L(j) = \{v_1, v_2, \dots, v_j\}$. For $i > j$, let $N(i, j) = N(v_i) \cap L(j)$ be the set of neighbors $v_h$ of $v_i$ with $h \leq j$. Let $q(j) = \prod q_{ab}$, where the product is taken over all edges $v_a v_b$ of $H$ with $1 \leq a < b \leq j$ and $q(i, j) = \prod_{v_h \in N(i,j)} q_{hi}$. Note that $q(j) = q(j, j-1) q(j-1)$.

We need to define several constants. To begin, we let $\theta_m = \theta$ and $\epsilon_m = 1$. Given $\theta_j$ and $\epsilon_j$, we define $\xi_j = \frac{\theta_j}{6m^2}$ and $\theta_{j-1} = \frac{\theta_j}{2}$. We apply Propositions \ref{prop:inheritintro} and \ref{prop:inherit2} with $\alpha, \xi_j$ and $\epsilon_j$ to find constants $c_{j-1}$ and $\epsilon_{j-1}^*$ such that the conclusions of the two propositions hold. We let $\epsilon_{j-1} = \min(\epsilon_{j-1}^* , \frac{\alpha \theta_j^2}{72m})$, $c = \frac{1}{2} \alpha^{d_2(H)} c_0$ and $\epsilon = \epsilon_0$.

We will find many embeddings $f: V(H) \rightarrow V(G)$ by embedding the vertices of $H$ one by one in increasing order. We will prove by induction on $j$ that there are $(1 - \theta_j) q(j) |X_1| |X_2| \dots |X_j|$ choices for $f(v_1), f(v_2), \dots, f(v_j)$ such that the following conditions hold. Here, for each $i > j$, we let $T(i, j)$ be the set of vertices in $X_i$ which are adjacent to $f(v_h)$ for every $v_h \in N(i,j)$. That is, it is the set of possible vertices into which, having embedded $v_1, v_2, \dots, v_j$, we may embed $v_i$.
\begin{itemize}

\item
For $1 \leq a < b \leq j$, $(f(v_a), f(v_b))$ is an edge of $G$ if $(v_a, v_b)$ is an edge of $H$;

\item
$|T(i, j)| \geq (1 - \frac{\theta_j}{6}) q(i, j) |X_i|$ for every $i
> j$;

\item
For each $i_1, i_2 > j$ with $v_{i_1} v_{i_2}$ an edge of $H$, the graph $(T(i_1,j), T(i_2,j))_G$ satisfies the discrepancy condition $\DISC(q_{ab}, p, \epsilon_j)$.

\end{itemize}

The base case $j = 0$ clearly holds by letting $T(i, 0) = X_i$. We may therefore assume that there are $(1 - \theta_{j-1}) q(j-1) |X_1| |X_2| \dots |X_{j-1}|$ embeddings of $v_1, \dots, v_{j-1}$ satisfying the conditions above. Let us fix such an embedding $f$. Our aim is to find a set $W(j) \subseteq T(j,j-1)$ with $|W(j)| \geq (1 - \frac{\theta_j}{2}) q(j, j-1) |X_j|$ such that for every $w \in W(j)$ the following three conditions hold.

\begin{enumerate}

\item
For each $i > j$ with $v_i \in N(v_j)$, there are at least $(1 - \frac{\theta_j}{6}) q(i, j) |X_i|$ vertices in $T(i, j-1)$ which are adjacent to $w$;

\item
For each $i_1, i_2 > j$ with $v_{i_1} v_{i_2}$, $v_{i_1} v_{j}$ and $v_{i_2} v_{j}$ edges of $H$, the induced subgraph of $G$ between $N(w) \cap T(i_1,j)$ and $N(w) \cap T(i_2,j)$ satisfies the discrepancy condition $\DISC(q_{ab}, p, \epsilon_j)$;

\item
For each $i_1, i_2 > j$ with $v_{i_1} v_{i_2}$ and $v_{i_1} v_{j}$ edges of $H$ and $v_{i_2} v_{j}$ not an edge of $H$, the induced subgraph of $G$ between $N(w) \cap T(i_1,j)$ and $T(i_2,j-1)$ satisfies the discrepancy condition $\DISC(q_{ab}, p, \epsilon_j)$.
\end{enumerate}

Note that once we have found such a set, we may take $f(v_j) = w$ for any $w \in W(j)$. By using the induction hypothesis to count the number of embeddings of the first $j-1$ vertices, we see that there are at least
\[\left(1 - \frac{\theta_j}{2}\right) q(j, j-1) |X_j| (1 - \theta_{j-1}) q(j-1) |X_1| |X_2| \dots |X_{j-1}| \geq (1 - \theta_j) q(j) |X_1| |X_2| \dots |X_j|\]
ways of embedding $v_1, v_2, \dots, v_j$ satisfying the necessary conditions. Here we used that $q(j) = q(j, j-1) q(j-1)$ and $\theta_{j-1} = \frac{\theta_j}{2}$. The induction therefore follows by letting $T(i, j) = N(w) \cap T(i, j-1)$ for all $i > j$ with $v_i \in N(v_j)$ and $T(i, j) = T(i, j-1)$ otherwise.

It remains to show that there is a large subset $W(j)$ of $T(j, j-1)$ satisfying the required conditions. For each $i > j$, let $A_i(j)$ be the set of vertices in $T(j, j-1)$ for which $|N(w) \cap T(i, j-1)| \leq (1 - \frac{\theta_j}{12}) q_{ij} |T(i, j-1)|$. Then, since the graph between $T(i, j-1)$ and $T(j,j-1)$ satisfies $\DISC(q_{ji}, p, \epsilon_{j-1})$, we have that
$\epsilon_{j-1} p |T(j, j-1)| \geq \frac{\theta_j}{12} q_{ij} |A_i(j)|$. Hence, since $q_{ij} \geq \alpha p$,
\[|A_i(j)| \leq \frac{12 \epsilon_{j-1}}{\alpha \theta_j} |T(j, j-1)|.\]
Note that for any $w \in T(j,j-1) \setminus A_i(j)$,
\[|N(w) \cap T(i, j-1)| \geq \left(1 - \frac{\theta_j}{12}\right) q_{ij} |T(i, j-1)| \geq \left(1 - \frac{\theta_j}{6}\right) q(i,j) |X_i|.\]

For each $i_1, i_2 > j$ with $v_{i_1} v_{i_2}$, $v_{i_1} v_j$ and $v_{i_2} v_j$ edges of $H$, let $B_{i_1, i_2}(j)$ be the set of vertices $w$ in $T(j,j-1)$ for which the graph between $N(w) \cap T(i_1, j-1)$ and $N(w) \cap T(i_2, j-1)$ does not satisfy $\DISC(q_{i_1 i_2}, p, \epsilon_j)$. Note that
\begin{align*}
|T(i_1, j-1)| |T(i_2,j-1)| & \geq \left(1 - \frac{\theta_{j-1}}{6}\right) q(i_1, j-1) q(i_2, j-1) |X_{i_1}| |X_{i_2}|\\
& \geq \frac{\alpha^{2d_2(H) - 2}}{2} p^{2d_2(H) - 2} |X_{i_1}| |X_{i_2}|,
\end{align*}
where we get $2d_2(H) - 2$ because $j$ is a neighbor of both $i_1$ and $i_2$ with $j < i_1, i_2$. Similarly, $|T(i_1, j-1)| |T(j, j-1)|$ and $|T(i_2,j-1)| |T(j,j-1)|$ are at least $\frac{\alpha^{2d_2(H)}}{2} p^{2d_2(H)} |X_{j}|$ and $\frac{\alpha^{2d_2(H)}}{2} p^{2d_2(H)} |X_{i_2}|$, respectively.

Since
\begin{align*}
c p^{d_2(H) + 3} \sqrt{|X_{i_1}||X_{i_2}|} & \leq \frac{1}{2} \alpha^{d_2(H)} c_0 p^4 \sqrt{p^{2d_2(H) - 2} |X_{i_1}| |X_{i_2}|}\\
& \leq c_0 p^4 \sqrt{|T(i_1, j-1)||T(i_2, j-1)|},
\end{align*}
the induced subgraph of $\Gamma$ between $T(i_1,
j-1)$ and $T(i_2, j-1)$ is $(p, c_0 p^4 \sqrt{|T(i_1, j-1)||T(i_2,
  j-1)|})$-jumbled. Similarly, the induced subgraph of $\Gamma$
between the sets $T(j, j-1)$ and $T(i_1,j-1)$ is $(p, c_0 p^3 \sqrt{|T(j, j-1)||T(i_1,
  j-1)|})$-jumbled and the induced subgraph between $T(j, j-1)$ and $T(i_2,j-1)$ is $(p, c_0 p^3 \sqrt{|T(j, j-1)||T(i_2,
  j-1)|})$-jumbled. By our choice of $\epsilon_{j-1}$, we may therefore apply Proposition \ref{prop:inheritintro} to show that $|B_{i_1, i_2}(j)| \leq \xi_j |T(j, j -1)|$.

For each $i_1, i_2 > j$ with $v_{i_1} v_{i_2}$ and $v_{i_1} v_{j}$ edges of $H$ and $v_{i_2} v_{j}$ not an edge of $H$, let $C_{i_1, i_2}(j)$ be the set of vertices $w$ in $T(j, j-1)$ for which the graph between $N(w) \cap T(i_1, j-1)$ and $T(i_2, j-1)$ does not satisfy $\DISC(q_{i_1 i_2}, p, \epsilon_j)$. As with $B_{i_1,i_2}(j)$, we may apply Proposition \ref{prop:inherit2} to conclude that $|C_{i_1, i_2}(j)| \leq \xi_j |T(j, j -1)|$.

Counting over all possible bad events and using that $|T(j,j-1)| \geq (1 - \frac{\theta_{j-1}}{6}) q(j,j-1) |X_j|$, we see that the set $W(j)$ of good vertices has size at least $(1 - \sigma) q(j,j-1) |X_j|$, where
\[\sigma \leq \frac{\theta_{j-1}}{6} + \frac{12 m \epsilon_{j-1}}{\alpha \theta_j} + 2\binom{m}{2} \xi_j \leq \frac{\theta_j}{12} + \frac{\theta_j}{6} + \frac{\theta_j}{6} \leq \frac{\theta_{j}}{2},\]
as required. Here we used $\theta_j = \frac{\theta_{j-1}}{2}$, $\epsilon_{j-1} \leq \frac{\alpha \theta_j^2}{72 m}$ and $\xi_j = \frac{\theta_j}{6m^2}$. This completes the proof.
\end{proof}

Note that for the clique $K_t$, we have $d_2(K_t) + 3 = t + 1$. In this case, it is better to use the bound coming from
two-sided counting, which gives the exponent $t$.

Another case of interest is when the graph $H$ is triangle-free. Here it is sufficient to always apply the simpler inheritance theorem, Proposition \ref{prop:inherit2}, to maintain
discrepancy. Then, since
\[p^{d_2(H) + 2} \sqrt{|X_{i_1}||X_{i_2}|} = p^{\frac{5}{2}} \sqrt{p^{2d_2(H) - 1} |X_{i_1}| |X_{i_2}|},\]
we see that an exponent of $d_2(H) + 2$ is sufficient in this case. In particular, for $H = K_{s,t}$, we get an exponent of $d_2(K_{s,t}) + 2 = \frac{s-1}{2} + 2 = \frac{s+3}{2}$, as quoted in Table \ref{tab:k}.

It is also worth noting that a one-sided counting lemma for $\Gamma$ holds under the slightly weaker assumption that $\beta \leq c p^{d_2(H) + 1} n$. We omit the details since the proof is a simpler version of the previous one, without the necessity for tracking inheritance of discrepancy.

\begin{proposition} \label{prop:oneside-Gamma}
For every fixed graph $H$ on vertex set $\{1,2, \dots, m\}$ and every $\alpha, \theta > 0$ there exist constants $c > 0$ and $\epsilon > 0$ such that the following holds.

Let $\Gamma$ be a graph with vertex subsets $X_1, \dots, X_m$ where vertex $i$ of $H$ is assigned to the vertex subset $X_i$ of $\Gamma$ and suppose that the bipartite graph $(X_i, X_j)_\Gamma$ is $(p, cp^{d_2(H) + 1} \sqrt{\abs{X_i}\abs{X_j}})$-jumbled for every $i < j$ with $ij \in E(H)$. Then $\Gamma(H) \geq (1 - \theta) p(H)$.
\end{proposition}

\section{Counting cycles} \label{sec:cycles}

Using the tools of doubling and densification, we already know how to
count all cycles. For cycles of length $4$ or greater, $(p,
cp^2n)$-jumbledness suffices.

\begin{proposition}
  \label{prop:cycle-two-sided}
  Assume Setup~\ref{set:G} with $H = C_\ell$ and $k \geq 3$ if $\ell = 3$
  or $k \geq 2$ if $\ell \geq 4$. Then $G(C_\ell)
  \approxmod_{c,\epsilon}^{p(C_\ell)} q(C_\ell)$.
\end{proposition}

\begin{proof}
  When $\ell = 4$, see Proposition~\ref{prop:C4}. When $\ell = 3$, see
  Section~\ref{sec:countingstrat} for the doubling procedure. For $\ell
  \geq 5$, we can perform densification to reduce the problem to
  counting $C_{\ell-1}$ with at least one dense edge, so we proceed by
  induction.
\end{proof}
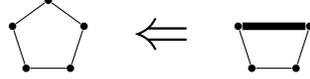
\[
\begin{tikzpicture}[scale=.5]
  \begin{scope}
    \foreach \i in {0,1,2,3,4}{
      \node[p] (\i) at (360/5 *\i + 90:1) {};
    }
    \draw (0)--(1)--(2)--(3)--(4)--(0);
  \end{scope}

  \node at (3,0) {\LargeLeftarrow};

  \begin{scope}[shift={(6,0)}]
    \foreach \i in {1,2,3,4}{
      \node[p] (\i) at (360/5 *\i + 90:1) {};
    }
    \draw (1)--(2)--(3)--(4);
    \draw[dense] (1)--(4);
  \end{scope}
\end{tikzpicture}
\]

The goal of this section is to prove a one-sided counting lemma for
cycles that requires much weaker jumbledness.

\begin{proposition}
  \label{prop:cycle-one-sided}
  Assume Setup~\ref{set:G} with $H = C_\ell$, where $\ell \geq 5$, and all
  edges sparse. Let $k
  \geq 1 + \frac{1}{\ell-3}$ if $\ell$ is odd and $1 + \frac{1}{\ell-4}$ if $\ell$ is even. Then $G(C_\ell)
  \geq q(C_\ell) - \theta p(C_\ell)$ with $\theta \leq 100 (\epsilon^{1/(2\ell)} + \ell c^{2/3})$.
\end{proposition}

The strategy is via subdivision densification, as outlined in
Section~\ref{sec:countingstrat}.

\subsection{Subdivision densification}\label{sec:subdivision}

In Section~\ref{sec:densification} we showed how to reduce a counting
problem by transforming a singly subdivided edge of $H$ into a dense
edge. In this section, we show how to transform a multiply subdivided
edge of $H$ into a dense edge, using much weaker hypotheses on
jumbledness, at least for one-sided counting. The
idea is that a long subdivision allows more room for mixing, and thus
requires less jumbledness at each step.

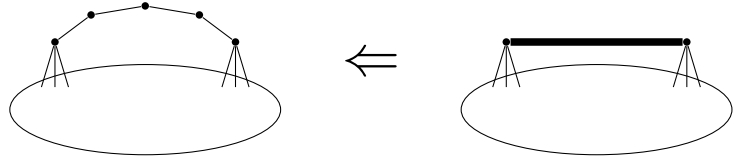
\begin{center}
  \begin{tikzpicture}[scale=.6]
    \begin{scope}
      \node[p] (a) at (-2,0) {};
      \node[p] (b1) at (-1.2,.6) {};
      \node[p] (b2) at (0,.8) {};
      \node[p] (b3) at (1.2,.6) {};
      \node[p] (c) at (2,0) {};
      \draw (a)--(b1)--(b2)--(b3)--(c);
      \draw (a) -- +(-.3,-1) (a) -- +(0,-1)  (a) -- +(.3,-1) ;
      \draw (c) -- +(-.3,-1) (c) -- +(0,-1) (c) -- +(.3,-1) ;
      \draw (0,-1.5) ellipse (3cm and 1cm);
    \end{scope}

    \node at (5,-.5) {\LargeLeftarrow};

    \begin{scope}[shift={(10,0)}]
      \node[p] (a) at (-2,0) {};
      \node[p] (c) at (2,0) {};
      \draw[dense] (a)--(c);
      \draw (a) -- +(-.3,-1) (a) -- +(0,-1)  (a) -- +(.3,-1) ;
      \draw (c) -- +(-.3,-1) (c) -- +(0,-1) (c) -- +(.3,-1) ;
      \draw (0,-1.5) ellipse (3cm and 1cm);
    \end{scope}
  \end{tikzpicture}
\end{center}

We introduce a weaker variant of discrepancy for one-sided counting.

\begin{definition}
  Let $G$ be a graph with vertex subsets $X$ and $Y$. We say that $(X,
  Y)_G$ satisfies $\DISC_{\geq} (q, p, \epsilon)$ if
  \begin{equation} \label{eq:super-disc}
    \int_{\substack{x \in X \\ y \in Y}} (G(x,y) - q)u(x)v(y) \ dxdy
    \geq -\epsilon p
  \end{equation}
  for all functions $u \colon X \to [0,1]$ and all $v \colon Y \to [0,1]$.
\end{definition}

In a graph $H$, we say that $a_0a_1a_2\cdots a_m$ is a \emph{subdivided
  edge} if the neighborhood of $a_i$ in $H$ is $\set{a_{i-1},
  a_{i+1}}$ for $1 \leq i \leq m-1$. Say that it is sparse if every
edge $a_ia_{i+1}$, $0 \leq i \leq m -1$, is sparse.

For a graph $\Gamma$ or $G$ with vertex subsets $X_0, X_1, \dots,
X_m$, $x_0 \in X_0$, $x_m \in X_m$ and $X'_i \subseteq X_i$, we write
\begin{align*}
  G(x_0, X'_1, X'_2, \dots, X'_m) &= \int_{\substack{x_1 \in X_1
      \\\cdots \\x_m \in X_m}} G(x_0, x_1) 1_{X'_1}(x_1) \cdots G(x_{m-1}, x_{m})
  1_{X'_m}(x_m) \ dx_1dx_2 \cdots dx_m,
  \\
  G(x_0, X'_1, X'_2, \dots, x_m) &= \int_{\substack{x_1 \in X_1
      \\\cdots \\x_{m-1} \in X_{m-1}}} G(x_0, x_1) 1_{X'_1}(x_1)
  \cdots  G(x_{m-1}, x_{m}) \ dx_1dx_2
  \cdots dx_{m-1}.
\end{align*}
These quantities can be interpreted probabilistically. The first
expression is the probability that a randomly chosen sequence of
vertices with one endpoint fixed is a path in $G$ with the vertices
landing in the chosen subsets. For the second expression, both
endpoints are fixed.

\begin{lemma}[Subdivision densification]
  \label{lem:subdiv-densify}
  Assume Setup~\ref{set:G}. Let $\ell \geq 2$ and let $a_0a_1 \cdots
  a_\ell$ be a sparse subdivided edge and assume that $a_0a_\ell$ is not an
  edge of $H$. Assume $k \geq 1 + \frac{1}{2\ell -
    2}$. Replace the induced bipartite graph $(X_{a_0}, X_{a_\ell})$
  by the weighted bipartite graph given by
  \[
  G(x_0, x_\ell) = \frac{1}{4p^\ell} \min\set{G(x_0, X_1, X_2, \dots,
    X_{\ell -1}, x_\ell), 4p^\ell }
  \]
  for $(x_0, x_\ell) \in X_0 \x X_\ell$.  Let $H'$ be $H$ with the
  path $a_0a_1\cdots a_\ell $ deleted and the edge $a_0a_\ell$
  added. Let $p_{a_0a_\ell} = 1$ and $q_{a_0a_\ell} =
  \frac{1}{4p^\ell} q_{a_0a_1}q_{a_1a_2}\cdots
  q_{a_{\ell-1}a_{\ell}}$. Then $G(H) \geq 4p^\ell G(H')$ and $(X_0,
  X_\ell)_G$ satisfies $\DISC_\geq(q_{a_0a_\ell}, 1,
  18(\epsilon^{1/(2\ell)} + \ell c^{2/3}))$.
\end{lemma}

\begin{remark}If there is at least one dense edge in the subdivision,
  then using arguments similar to the ones in
  Section~\ref{sec:densification}, modified for one-sided counting, we
  can show that $k \geq 1$ suffices for subdivision densification.
\end{remark}

The idea of the proof is very similar to densification in
Section~\ref{sec:densification}. The claim $G(H) \geq 4p^\ell G(H')$
follows easily from the new edge weights. It remains to show that
$(X_0, X_\ell)_G$ satisfies $\DISC_\geq$. So
Lemma~\ref{lem:subdiv-densify} follows from the next result.

\begin{lemma}
  \label{lem:path-densify}
  Let $m \geq 2$,
  $c, \epsilon, p \in (0,1]$, and $q_1, q_2, \dots, q_m \in [0,p]$. Let $\Gamma$
  be any weighted graph with vertex subsets $X_0, X_1, \dots, X_m$
  and let $G$ be a subgraph of $\Gamma$. Suppose that, for each $i = 1,
  \dots, m$, $(X_{i-1}, X_i)_\Gamma$ is $(p, cp^{1 +
    \frac{1}{2m-2}}\sqrt{\abs{X_{i-1}}\abs{X_i}})$-jumbled and
  $(X_{i-1},X_i)_G$ satisfies $\DISC_\geq (q_i, p, \epsilon)$. Then
  the weighted graph $G'$ on $(X_0, X_m)$ defined by
  \[
  G'(x_0, x_m) = \min\set{G(x_0, X_1, X_2, \dots, X_{m-1}, x_m),
    4p^m}
  \]
  satisfies $\DISC_\geq(q_1q_2 \cdots q_m, p^m, 72(\epsilon^{1/(2m)} + mc^{2/3}))$.
\end{lemma}

Here are the steps for the proof of Lemma~\ref{lem:path-densify}.
\begin{enumerate}
\item Show that the graph on $X_0 \x X_m$ with weights $G(x_0, X_1,
  X_2, \dots, X_{m-1}, x_m)$ satisfies $\DISC_\geq$.
\item Under the assumption that every vertex $X_i$ has roughly
  the same number of neighbors in $X_{i+1}$ for every $i$, show that
  capping of the edge weights has negligible effect on discrepancy.
\item Show that we can delete a small subset from each vertex subset
  $X_i$ so that the assumption in step 2 is satisfied.
\end{enumerate}
Step 2 is the most difficult. Since we are only proving lower bound
discrepancy, it is okay to delete vertices in step 3. This is also the
reason why this proof, without significant modification, cannot prove
two-sided discrepancy (which may require stronger hypotheses), as we
may have deleted too many edges in the process. Also, unlike the
densification in Section~\ref{sec:densification}, we do not have to
worry about the effect of the edge weight capping on the overall
$H$-count, as we are content with a lower bound.

The next two lemmas form step 1 of the program.

\begin{lemma}
  \label{lem:G-2-super-disc}
  Let $G$ be a weighted graph with vertex subsets $X, Y, Z$. Let $p_1,
  p_2,\epsilon \in (0,1]$ and $q_1 \in [0,p_1]$, $q_2 \in [0,p_2]$.
  If $(X, Y)_G$ satisfies $\DISC_{\geq}(q_1, p_1, \epsilon)$ and $(Y, Z)_G$
  satisfies $\DISC_{\geq}(q_2, p_2, \epsilon)$, then the induced
  weighted bipartite graph $G'$ on $(X,Z)$ whose weight is given by
  \[
  G'(x,z) = G(x,Y,z)
  \]
  satisfies $\DISC_{\geq}(q_1q_2, p_1p_2, 6\sqrt{\epsilon})$.
\end{lemma}

Note that no jumbledness hypothesis is needed for the lemma.

\begin{proof}
  Let $u \colon X \to [0,1]$ and $w \colon Z \to [0,1]$ be
  arbitrary functions. Let
  \[
  Y' = \setcond{y \in Y}{\int_{x \in X} (G(x,y) - q_1)u(x) \ dx \leq
    -\sqrt{\epsilon} p_1 }.
  \]
  Then applying \eqref{eq:super-disc} to $u$ and $1_{Y'}$ yields
  $\abs{Y'} \leq \sqrt{\epsilon} \abs{Y}$. Similarly, let
  \[
  Y'' = \setcond{y \in Y}{\int_{z \in Z} (G(y,z) - q_2)w(z) \ dx \leq
    -\sqrt{\epsilon} p_2 }.
  \]
  Then $\abs{Y''} \leq \sqrt{\epsilon} \abs{Y}$ as well. So
  \begin{align*}
    \int_{\substack{x \in X \\ z \in Z}} G'(x,z) u(x)w(z) \ dxdz
    &
    = \int_{\substack{x \in X \\ y \in Y \\ z \in Z}}
    u(x)G(x,y)G(y,z)w(z) \ dxdydz
    \\
    &\geq \int_{y \in Y\setminus(Y' \cup Y'')} \paren{\int_{x \in X}
      G(x,y)u(x)  \ dx}\paren{\int_{z \in Z}
      G(y,z)w(z) \ dz} \ dy
    \\
    &\geq \int_{y \in Y \setminus (Y' \cup Y'')} (q_1 \EE u -
    \sqrt{\epsilon} p_1) (q_2 \EE w -
    \sqrt{\epsilon} p_2) \ dy
    \\
    &\geq (1 - 2\sqrt{\epsilon}) (q_1 \EE u -
    \sqrt{\epsilon} p_1) (q_2 \EE w -
    \sqrt{\epsilon} p_2)
    \\
    &\geq q_1q_2\EE u \EE w - 6\sqrt{\epsilon}p_1p_2.
  \end{align*}
\end{proof}

The above proof can be extended to prove a one-sided counting lemma
for trees without any jumbledness hypotheses. We omit the details.

\begin{proposition}
  \label{prop:tree-oneside}
  Let $H$ be a tree on vertices $\set{1,2,\dots, m}$. For every
  $\theta > 0$, there exists $\epsilon > 0$ of size at least
  polynomial in $\theta$ such that the following holds.

  Let $G$ be a
  weighted graph with vertex subsets $X_1, \dots, X_m$. For each edge
  $ab$ of $H$, suppose that $(X_a, X_b)_G$ satisfies $\DISC_\geq(q_{ab}, p_{ab},
  \epsilon)$ for some $0 \leq q_{ab} \leq p_{ab} \leq 1$. Then $G(H)
  \geq q(H) - \theta p(H)$.
\end{proposition}

By Lemma~\ref{lem:G-2-super-disc} and induction, we obtain the
following lemma about counting paths in $G$.

\begin{lemma}
  \label{lem:G-path-super-disc}
  Let $G$ be a weighted graph with vertex subsets $X_0, X_1, \dots,
  X_m$. Let $0 < \epsilon < 1$. Suppose that for each $i = 1, 2, \dots,
  m$, $(X_{i-1}, X_i)_G$ satisfies $\DISC_{\geq} (q_i, p_i,
  \epsilon)$ for some numbers $0 \leq q_i \leq p_i \leq 1$. Then the
  induced weighted bipartite graph $G'$ on $X_0 \x X_m$ whose edge weights are given by
  \[
  G'(x_0,x_m) = G(x_0, X_1, X_2, \dots, X_{m-1}, x_m)
  \]
  satisfies $\DISC_{\geq}(q_1q_2\cdots q_m, p_1p_2\cdots p_m, 36 \epsilon^{1/(2m)})$.
\end{lemma}

\begin{proof}
  Applying Lemma~\ref{lem:G-2-super-disc}, we see that the auxiliary weighted
  graphs on $(X_0, X_2), (X_2, X_4), \dots$ satisfy
  $\DISC_{\geq}(q_1q_2, p_1p_2, 36\epsilon^{1/2})$, etc. Applying
  Lemma~\ref{lem:G-2-super-disc} again, we find that the auxiliary weighted
  graph on $(X_0, X_4), (X_4, X_8)$ satisfy
  $\DISC_{\geq}(q_1q_2q_3q_4, p_1p_2p_3p_4, 36\epsilon^{1/4})$,
  etc. Continuing, we find that $(X_0, X_m)_{G'}$ satisfies
  $\DISC_{\geq}(q_1q_2\cdots q_m, p_1p_2\cdots p_m,\epsilon')$ with
  $\epsilon' = 36\epsilon^{2^{-(\log_2 m + 1)}} = 36\epsilon^{1/(2m)}$.
\end{proof}

For step 2 of the proof, we need to assume some degree-regularity
between the parts. We note that the order of $X$ and $Y$ is important in the
following definition.

\begin{definition}
  Let $\Gamma$ be a weighted graph with vertex subsets $X$ and $Y$. We say that
  $(X, Y)_\Gamma$ is \emph{$(p, \xi, \eta)$-bounded} if $\abs{\Gamma(x, Y) -
    p} \leq \xi p$ for all $x \in X$ and $\Gamma(x,y) \leq \eta$ for
  all $x \in X$ and $y \in Y$.
\end{definition}

Here is the idea of the proof. Fix a vertex $x_0 \in X_0$, and
consider its successive neighborhoods in $X_1, X_2, \dots$. Let us
keep track of the number of paths from $x_0$ to each
endpoint. We expect the number of paths to be somewhat evenly
distributed among vertices in the successive neighborhoods and,
therefore, we do not expect many vertices in $X_i$ to have
disproportionately many paths to $x_0$. In particular,
capping the weights of $\Gamma(x_0, X_1, \dots, X_{m-1}, x_m)$ has a
negligible effect.

Here is a back-of-the-envelope calculation. Suppose every pair $(X_i,
X_{i+1})_\Gamma$ is $(p,
\gamma\sqrt{\abs{X_i}\abs{X_{i+1}}})$-jumbled. First we remove a small
fraction of vertices from each vertex subset $X_i$ so in the remaining
graph $\Gamma$ is bounded, i.e., every vertex has roughly the expected
number of neighbors in the next vertex subset. Let $S \subseteq X_i$,
and let $N(S)$ be its neighborhood in $X_{i+1}$. Then the number of
edges $e(S, N(S))$ between $S$ and $N(S)$ is roughly
$p\abs{S}\abs{X_{i+1}}$ by the degree assumptions on $X_i$. On the
other hand, by jumbledness, $e(S, N_{i+1}(S)) \leq
\gamma\sqrt{\abs{X_i}\abs{X_{i+1}}\abs{S}\abs{N(S)}} + p
\abs{S}\abs{N(S)}$. When $S$ is small, the first term dominates, and
by comparing the two estimates we get that
$\frac{\abs{N(S)}}{\abs{X_{i+1}}}$ is at least roughly $p^2\gamma^{-2}
\frac{\abs{S}}{\abs{X_i}}$. Now fix a vertex $x_0 \in X_0$. It has
about $p\abs{X_1}$ neighbors in $X_1$. At each step, the fraction of
$X_i$ occupied by the successive neighborhood of $x_0$ expands by a
factor of about $p^2\gamma^{-2}$, until the successive neighborhood
saturates some $X_i$. Note that for $\gamma = cp^{1+ \frac{1}{2m-2}}$,
we have $p(p^2\gamma^{-2})^{m-1} \gg 1$, so the successive
neighborhood of $x_0$ in $X_m$ is essentially all of $X_m$. So we can
expect the resulting weighted graph to be dense.

We will use induction. We show that from a fixed $x_0 \in X_0$,
if we can bound the number of paths to each vertex in $X_i$, then we
can do so for $X_{i+1}$ as well.

The next result is the key technical lemma. It is an induction step for the
lemma that follows. One should think of $X, Y$ and $Z$ as $X_0, X_i$ and $X_{i+1}$, respectively.

\begin{lemma}
  \label{lem:Gamma-bounded-step}
  Let $p_1, p_2, \xi_1, \xi_2, \xi_3 \in (0,1]$, and $\eta_1, \gamma_2 > 0$. Let $\Gamma$ be a
  weighted graph with vertex subsets $X, Y, Z$. Assume that
  $(X,Y)_\Gamma$ is $(p_1, \xi_1, \eta_1)$-bounded and $(Y,Z)_\Gamma$
  is $(p_2, \xi_2, 1)$-bounded and $(p_2,
  \gamma_2\sqrt{\abs{Y}\abs{Z}})$-jumbled. Let $\eta' =
  \max\set{4\gamma_2^2p_2^{-1}\xi_3^{-1}\eta_1, 4p_1p_2}$ and $\xi' =
  \xi_1 + 2\xi_2 + 2\xi_3$. Then the weighted graph
  $\Gamma'$ on $(X, Z)$ given by
  \[
  \Gamma'(x,z) = \min\set{\Gamma(x,Y,z), \eta'}
  \]
  is $(p_1p_2, \xi', \eta')$-bounded.
\end{lemma}

\begin{proof}
  We have $\Gamma'(x,z) \leq \eta'$ for all $x \in X, z \in Z$. Also,
  by the boundedness assumptions, we have $\Gamma'(x,Z) \leq
  \Gamma(x,Y,Z) \leq (1+\xi_1)(1+\xi_2)p_1p_2 \leq (1+\xi')p_1p_2$. It
  only remains to prove that $\Gamma'(x,Z) \geq (1-\xi')p_1p_2$ for
  all $x \in X$.

  Fix any $x \in X$. Let
  \[
  Z'_x = \setcond{z \in Z}{\Gamma(x,Y,z) > \eta'}.
  \]
  Note that $\Gamma'(x,Z) \geq \Gamma(x,Y,Z) - \Gamma(x,Y,Z'_x)$, so
  we would like to find an upper bound for $\Gamma(x,Y,Z'_x)$.

  Apply the jumbledness criterion~\eqref{eq:jumbled} to $(Y,Z)_\Gamma$ with the functions $u(y) =
  \Gamma(x,y) \eta_1^{-1}$ and $v(z) = 1_{Z'_x}$. Note that $0 \leq u
  \leq 1$ due to boundedness. We have
  \[
  \int_{\substack{y \in Y \\ z \in Z}}
  \Gamma(x,y) \eta_1^{-1}  (\Gamma(y,z) - p_2) 1_{Z'_x}(z) \ dydz
  \leq \gamma_2\sqrt{\Gamma(x,Y)\eta_1^{-1} \frac{\abs{Z'_x}}{\abs{Z}}}
  \leq \gamma_2\sqrt{(1+\xi_1)p_1\eta_1^{-1} \frac{\abs{Z'_x}}{\abs{Z}}}.
  \]
  The integral equals $\eta_1^{-1}\paren{\Gamma(x,Y,Z'_x) -
    p_2\Gamma(x,Y) \frac{\abs{Z'_x}}{\abs{Z}}}$, so we have
  \begin{equation}
    \label{eq:Gamma-step-jumb}
    \Gamma(x,Y,Z'_x) -
    p_2\Gamma(x,Y) \frac{\abs{Z'_x}}{\abs{Z}}
    \leq \gamma_2\sqrt{(1+\xi_1)p_1\eta_1 \frac{\abs{Z'_x}}{\abs{Z}}}.
  \end{equation}
  On the other hand, we have
  \begin{equation} \label{eq:Gamma-step-YZ-lower}
    \Gamma(x,Y,Z'_x) -
    p_2\Gamma(x,Y) \frac{\abs{Z'_x}}{\abs{Z}}
    \geq \eta'\frac{\abs{Z'_x}}{\abs{Z}} - (1+\xi_1)p_1p_2
    \frac{\abs{Z'_x}}{\abs{Z}}
    \geq \frac{\eta'}{2} \frac{\abs{Z'_x}}{\abs{Z}}.
  \end{equation}
  Combining \eqref{eq:Gamma-step-YZ-lower} with
  \eqref{eq:Gamma-step-jumb}, we get
  \begin{equation} \label{eq:Gamma-step-Z} \frac{\abs{Z'_x}}{\abs{ Z}}
    \leq \frac{4\gamma_2^2 (1+\xi_1)p_1\eta_1}{\eta'^2}.
  \end{equation}
  Substituting \eqref{eq:Gamma-step-Z} back into \eqref{eq:Gamma-step-jumb}, we have
  \begin{align*}
    \Gamma(x,Y,Z'_x)
    &\leq (1+\xi_1)p_1p_2 \frac{\abs{Z'_x}}{\abs{Z}}
    + \gamma_2\sqrt{(1+\xi_1)p_1\eta_1 \frac{\abs{Z'_x}}{\abs{Z}}}
    \\
    & \leq \frac{4\gamma_2^2 (1+\xi_1)^2p_1^2p_2\eta_1}{\eta'^2}
    + \frac{2 \gamma_2^2 (1+\xi_1)p_1\eta_1}{\eta'}
    \\
    & \leq \frac{4\gamma_2^2 (1+\xi_1)^2p_1^2p_2\eta_1}{(4\gamma_2^2p_2^{-1}\xi_3^{-1}\eta_1)(4p_1p_2)}
    + \frac{2 \gamma_2^2
      (1+\xi_1)p_1\eta_1}{4\gamma_2^2p_2^{-1}\xi_3^{-1}\eta_1} \\
    & = \frac{1}{4}(1+\xi_1)^2\xi_3p_1p_2 + \frac{1}{2}(1+\xi_1)\xi_3p_1p_2
    \\
    &\leq 2\xi_3 p_1p_2.
  \end{align*}
  Therefore,
  \[
  \Gamma'(x,Z) \geq \Gamma(x,Y,Z) - \Gamma(x,Y,Z'_x)
  \geq (1-\xi_1)(1-\xi_2)p_1p_2 - 2\xi_3 p_1p_2
  \geq (1 - \xi')p_1p_2.
  \]
  This completes the proof that $\Gamma'$ is $(p_1p_2, \xi', \eta')$-bounded.
\end{proof}

By repeated applications of Lemma~\ref{lem:Gamma-bounded-step}, we
obtain the following lemma for embedding paths in $\Gamma$.

\begin{lemma}
  \label{lem:Gamma-bounded-path}
  Let $0 < 4c^2 < \xi < \frac{1}{4m}$ and $0 < p \leq 1$.
  Let $\Gamma$ be a graph with vertex subsets $X_0, X_1, \dots,
  X_m$. Suppose that, for each $i = 1, \dots, m$, $(X_{i-1},
  X_i)_\Gamma$ is $(p, \xi, 1)$-bounded and $(p,
  cp^{1+\frac{1}{2m-2}}\sqrt{\abs{X_{i-1}}\abs{X_i}})$-jumbled. Then
  the weighted bipartite graph $\Gamma'$ on $(X_0, X_m)$ defined by
  \[
  \Gamma'(x_0, x_m) = \min\set{\Gamma(x_0, X_1, X_2, \dots, X_{m-1},
    x_m), 4p^m}
  \]
  is $(p^m, 4m\xi, 4p^m)$-bounded.
\end{lemma}

\begin{proof}
  Since $\Gamma'(x_0, X_m) \leq \Gamma(x_0, X_1, X_2, \dots, X_m) \leq
  (1+\xi)^mp^m \leq e^{m\xi}p^m \leq (1 + 4m\xi)p^m$ for all $x_0 \in X_0$, it remains to show
  that $\Gamma'(x_0, X_m) \geq (1-4m\xi)p^m$ for all $x_0 \in X_0$.

  For every $i = 1, \dots, m$, define a weighted graph $\Gamma^{(i)}$
  on vertex sets $X_0, X_i, X_{i+1}$ (with $\Gamma^{(m)}$ only defined
  on $X_0$ and $X_m$) as follows. Set $(X_i,
  X_{i+1})_{\Gamma^{(i)}} = (X_i, X_{i+1})_\Gamma$ for each $1 \leq i
  \leq m-1$. Set
  $(X_0, X_1)_{\Gamma^{(1)}} = (X_0, X_1)_\Gamma$ and
  \[
  \Gamma^{(i+1)}(x_0, x_{i+1}) = \min\set{\Gamma^{(i)}(x_0, X_i,
    x_{i+1}), \eta_{i+1}}
  \]
  for each $1 \leq i
  \leq m-1$, where
  \[
  \eta_i = \max\set{(4c^2\xi^{-1})^{i-1} p^{(i-1)\paren{1 +
        \frac{1}{m-1}}}, 4p^i}
  \]
  for every $i$. So $\Gamma^{(i)}(x_0, x_i) \leq \Gamma(x_0, X_1,
  \dots, X_{i-1}, x_i)$ for every $i$ and every $x_0 \in X_0, x_i \in
  X_i$. Let $\gamma = cp^{1+\frac{1}{2m-2}}$. Note that
  $\eta_{i+1} = \max\set{4\gamma^2p^{-1}\xi^{-1}\eta_i, 4p^{i+1}}$ for
  every $i$. So it follows by Lemma~\ref{lem:Gamma-bounded-step} and
  induction that
  $(X_0, X_i)_{\Gamma^{(i)}}$ is $(p^i, 4i\xi, \eta_i)$-bounded for every
  $i$. Since $\eta_m = 4p^m$, $\Gamma'(x_0, X_m) \geq
  \Gamma^{(m)}(x_0, X_m) \geq (1 - 4m\xi)p^m$, as desired.
\end{proof}

To complete step 2 of the proof, we show that the boundedness
assumptions imply that the edge weight capping has negligible effect
on discrepancy.

\begin{lemma}
  \label{lemma:path-bounded-densify}
  Let $0 < 4c^2 < \xi$ and $0 < p \leq 1$.
  Let $\Gamma$ be a graph with vertex subsets $X_0, X_1, \dots, X_m$
  and let $G$ be a subgraph of $\Gamma$. Suppose that, for each $i = 1,
  \dots, m$, $(X_{i-1}, X_i)_\Gamma$ is $(p, \xi, 1)$-bounded and $(p, cp^{1 +
    \frac{1}{2m-2}}\sqrt{\abs{X_{i-1}}\abs{X_i}})$-jumbled and $(X_{i-1},X_i)_G$ satisfies
  $\DISC_\geq (q_i, p_i, \epsilon)$. Then the weighted graph $G'$ on
  $(X_0, X_m)$ defined by
  \[
  G'(x_0, x_m) = \min\set{G(x_0, X_1, X_2, \dots, X_{m-1}, x_m),
    4p^m}
  \]
  satisfies $\DISC_\geq(q_1q_2 \cdots q_m, p^m, 36\epsilon^{1/(2m)} + 8m\xi)$.
\end{lemma}

\begin{proof}
  We may assume that $\xi < \frac{1}{4m}$ since otherwise the claim is
  trivial as every graph satisfies $\DISC_\geq(q,p,\epsilon)$ when
  $\epsilon \geq 1$.
  Let $\Gamma'$ be constructed as in
  Lemma~\ref{lem:Gamma-bounded-path}. To simplify notation, let us
  write
  \begin{align*}
    G(x_0, x_m) &= G(x_0, X_1, \cdots, X_{m-1}, x_m) \\
    \text{and } \quad \Gamma(x_0, x_m) &= \Gamma(x_0, X_1, \cdots, X_{m-1}, x_m)
  \end{align*}
  for $x_0 \in X_0$, $x_m \in X_m$.
  We have
  \begin{align*}
  G(x_0, x_m) - G'(x_0, x_m)
  &= \max\set{0, G(x_0, x_m) - 4p^m}\\
  &\leq \max\set{0, \Gamma(x_0, x_m) - 4p^m}
  = \Gamma(x_0, x_m) - \Gamma'(x_0, x_m).
  \end{align*}
  Let $q = q_1q_2\cdots q_m$.
  For any functions $u \colon X \to [0,1]$ and $v \colon Y \to
  [0,1]$, we have
  \begin{align*}
    \int_{\substack{x_0 \in X_0 \\ x_m \in X_m}} (G'(x_0,x_m) -
    q)u(x_0)v(x_m) & \ dx_0dx_m
    \geq \int_{\substack{x_0 \in X_0 \\ x_m \in X_m}} (G(x_0,x_m) -
    q)u(x_0)v(x_m) \ dx_0dx_m \\ &- \int_{\substack{x_0 \in X_0 \\ x_m \in
        X_m}} (\Gamma(x_0,x_m) - \Gamma'(x_0, x_m))u(x_0)v(x_m) \ dx_0dx_m.
  \end{align*}
  The first term is at least $-36\epsilon^{1/(2m)}p^m$ by
  Lemma~\ref{lem:G-path-super-disc}. For the second term, we use the
  boundedness of $\Gamma$ and $\Gamma'$ to get
  \begin{align*}
    \int_{\substack{x_0 \in X_0 \\ x_m \in
        X_m}} (\Gamma(x_0,x_m) - \Gamma'(x_0, x_m))u(x_0)v(x_m) \
    dx_0dx_m
    &\leq \int_{\substack{x_0 \in X_0 \\ x_m \in X_m}}
    (\Gamma(x_0,x_m) - \Gamma'(x_0, x_m) \ dx_0dx_m
    \\
    &\leq (1+\xi)^mp^m
    - (1-4\xi m)p^m
    \\
    &\leq 8\xi m p^m.
  \end{align*}
  It follows that $G'$ satisfies $\DISC_\geq (q,p^m,36\epsilon^{1/(2m)} + 8m\xi)$.
\end{proof}

This completes step 2 of the program. Finally, we need to show that we
have a large subgraph of $\Gamma$ satisfying boundedness, so that we can
apply Lemma~\ref{lemma:path-bounded-densify} and then transfer the
results back to the original graph.

\begin{lemma}
  \label{lem:make-bounded}
  Let $0 < \delta, \gamma, \xi, p < 1$ satisfy
  $2\gamma^2 \leq \delta \xi^2p^2$.
  Let $\Gamma$ be a graph with vertex subsets $X_0, X_1, \dots, X_m$
  and suppose that, for each $i = 1,
  \dots, m$, $(X_{i-1}, X_i)_\Gamma$ is $(p,
  (1-\delta)\gamma\sqrt{\abs{X_{i-1}}\abs{X_i}})$-jumbled. Then we can find
  $\wt X_i \subseteq X_i$ with $\abs{\wt X_i} \geq (1-\delta)\abs{X_i}$
  for every $i$ such that, for every $0 \leq i \leq m-1$, the induced bipartite
  graph $(\wt X_i, \wt X_{i+1})_\Gamma$ is $(p, \xi, 1)$-bounded and
  $(p, \gamma\sqrt{\abs{\wt X_{i}}\abs{\wt X_{i+1}}})$-jumbled.
\end{lemma}

\begin{proof}
  The jumbledness condition follows directly from the size of $\abs{X_i}$, so it
  suffices to make the bipartite graphs bounded. Let $\wt X_m =
  X_m$. For each $i = m-1, m-2, \dots, 0$, in this order, set $\wt
  X_i$ to be the vertices in $X_i$ with $(1 \pm \xi)p\abs{\wt
    X_{i+1}}$ neighbors in $X_{i+1}$. So
  $(\wt X_i, \wt X_{i+1})$ is $(p, \xi, 1)$-bounded. Lemma~\ref{lem:deg-est} gives us $\abs{X_i \setminus \wt X_i} \leq
  \frac{2\gamma^2}{\xi^2 p^2}\abs{X_i} \leq \delta \abs{X_i}$.
\end{proof}

\begin{lemma}
  \label{lem:disc-pass-up}
  Let $0 \leq q \leq p \leq 1$ and $\epsilon, \delta, \delta' > 0$.
  Let $G$ be a weighted bipartite graph with vertex sets $X$ and $Y$. Let
  $\wt X \subseteq X$ and $\wt Y \subseteq Y$ satisfy $\abs{\wt X} \geq (1
  -\delta)\abs{X}$ and $\abs{\wt Y} \geq (1- \delta)\abs{Y}$. Let $\wt
  G$ be a weighted bipartite graph on $(\wt X, \wt Y)$ such that
  $G(x,y) \geq (1 - \delta')\wt G (x,y)$ for all $x \in \wt X, y \in
  \wt Y$. If $(\wt
  X, \wt Y)_{\wt G}$ satisfies $\DISC_\geq (q, p, \epsilon)$, then $(X,Y)_G$
  satisfies $\DISC_\geq(q,p,\epsilon + 2\delta + \delta')$.
\end{lemma}

\begin{proof}
  For this proof we use sums instead of integrals since the integrals
  corresponding to $(X,Y)_G$ and $(\wt X, \wt Y)_G$ have different
  normalizations and can be somewhat confusing. Let $u \colon X \to
  [0,1]$ and $v \colon Y \to [0,1]$. We have
  \begin{align*}
    \sum_{x \in X}\sum_{y \in Y} G(x,y)u(x)v(y)
    &\geq (1 - \delta')\sum_{x \in \wt X} \sum_{y \in \wt Y}\wt G(x,y)u(x)v(y) \\
    &\geq q (1-\delta') \paren{\sum_{x \in \wt X} u(x)}\paren{\sum_{y \in \wt Y} v(y)} - \epsilon
    p \abs{\wt X} \abs{\wt Y}
    \\
    &\geq q (1 - \delta') \paren{\sum_{x \in X}
      u(x) - \delta\abs{X}}\paren{\sum_{y \in Y} v(y) - \delta\abs{Y}} - \epsilon
    p \abs{X} \abs{Y}
    \\
    & \geq q u(X) v(Y) - (\epsilon + 2\delta + \delta') p\abs{X}\abs{Y}.
  \end{align*}
\end{proof}

\begin{proof}[Proof of Lemma~\ref{lem:path-densify}]
  We apply Lemma~\ref{lem:make-bounded} to find large subsets of
  vertices for which the induced subgraph of $\Gamma$ is bounded and then apply
  Lemma~\ref{lemma:path-bounded-densify} to show that $G$ restricted to this subgraph satisfies $\DISC_\geq$.
  Finally, we use Lemma~\ref{lem:disc-pass-up} to pass the result back
  to the original graph.

  Here are the details. Let $\xi = 8c^{2/3}$ and $\delta = \frac{1}{4}c^{2/3}$,
  so that the hypotheses of Lemma~\ref{lem:make-bounded} are satisfied
  with $\gamma = \frac{c}{1-\delta}p^{1 + \frac{1}{2m-2}}$. Therefore, we
  can find $\wt X_i \subseteq X_i$ with $\abs{\wt X_i} \geq (1-
  \delta){X_i}$ for each $i$ so that $(\wt X_i, \wt X_{i+1})_{\Gamma}$
  is $(p, \xi, 1)$-bounded and $(p, \frac{c}{1-\delta}
  p^{1+\frac{1}{2m-2}}\sqrt{\abs{\wt X_i}{\abs{\wt
        X_{i+1}}}})$-jumbled for every $0 \leq i \leq m-1$. Let $\wt
  G$ denote the graph $G$ restricted to $\wt X_0, \dots, \wt
  X_m$. Note that the normalizations of $G$ and $\wt G$ are
  different. For instance, for any $S \subseteq \wt X_1$ and any $x_0
  \in \wt X_0$ and $x_2 \in \wt X_2$, we write
  \[
  G(x_0, S, x_2) = \frac{1}{\abs{X_1}} \sum_{x_1 \in S} G(x_0, x_1)G(x_1,x_2)
  \]
  while
  \[
  \wt G(x_0, S, x_2) = \frac{1}{\abs{\wt X_1}} \sum_{x_1 \in S} G(x_0, x_1)G(x_1,x_2).
  \]
  So $(\wt X_{i-1}, \wt X_i)_{\wt G}$ satisfies $\DISC_{\geq}(q_i, p,
  \epsilon')$ with $\epsilon' \leq \frac{\epsilon}{(1 - \delta)^2}
  \leq 2\epsilon$.
  Let $\wt G'$ denote the weighted bipartite graph on $(\wt X_0, \wt
  X_m)$ given by
  \[
  \wt G'(x_0, x_m) = \min\set{\wt G(x_0, \wt X_1, \dots, \wt X_{m-1},
    x_m), 4p^m}.
  \]
  Since $4(\frac{c}{1-\delta})^2 \leq 8c^2 \leq \xi$, we can apply
  Lemma~\ref{lemma:path-bounded-densify} to $\wt G$ to find
  that $(\wt X_0, \wt X_m)_{\wt G'}$ satisfies $\DISC_\geq (q_1\cdots
  q_m, p^m, 72\epsilon^{1/(2m)} + 8m\xi)$. To pass the result back to $G'$, we
  note that
  \begin{align*}
    G'(x_0, x_m)
    &= \min\set{G(x_0, X_1, \dots, X_{m-1}, x_m), 4p^m}
    \\
    &\geq \min\set{G(x_0, \wt X_1, \dots, \wt X_{m-1}, x_m), 4p^m}
    \\
    &= \min\set{\frac{\abs{\wt X_1} \cdots \abs{\wt
          X_{m-1}}}{\abs{X_1} \cdots \abs{X_{m-1}}}\wt G(x_0, \wt X_1,
      \dots, \wt X_{m-1}, x_m), 4p^m}
    \\
    &\geq (1 - \delta)^{m-1} \wt G'(x_0, x_m)
    \\
    &\geq (1- (m-1)\delta) \wt G'(x_0, x_m).
  \end{align*}
  It follows by Lemma~\ref{lem:disc-pass-up} that $(X_0, X_m)_{G'}$
  satisfies $\DISC_\geq(q_1\cdots q_m, p^m, \epsilon')$ with $\epsilon'
  \leq 72\epsilon^{1/(2m)} + 8m\xi + 2\delta + (m-1)\delta \leq
  72(\epsilon^{1/(2m)}  + mc^{2/3})$.
\end{proof}

\subsection{One-sided cycle counting}

If we can perform densification to reduce $H$ to a triangle with two
dense edges, then we have a counting lemma for $H$, as shown by the
following lemma. Note that we do not even need
any jumbledness assumptions on the remaining sparse edge.

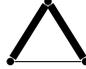
\begin{center}
  \begin{tikzpicture}[scale=.5]
    \node[p] (a) at (-1,0) {};
    \node[p] (b) at (1,0) {};
    \node[p] (c) at (0,1.6) {};
    \draw[dense] (a)--(c)--(b);
    \draw (a)--(b);
  \end{tikzpicture}
\end{center}

\begin{lemma}
  \label{lem:triangle-single-sparse-edge}
  Let $K_3$ denote the triangle with vertex set $\set{1,2,3}$. Let $G$
  be a weighted graph with vertex subsets $X_1, X_2, X_3$ such that,
  for all $i \neq j$,
  $(X_i, X_j)_G$ satisfies $\DISC_\geq (q_{ij}, p_{ij}, \epsilon)$,
  where $p_{13} = p_{23} = 1$, $0 \leq p_{12} \leq 1$, and $0 \leq
  q_{ij} \leq p_{ij}$. Then $G(K_3) \geq q_{12}q_{13}q_{23} -
  3\epsilon p_{12}$.
\end{lemma}

\begin{proof}
We have
  \begin{multline*}
  G(K_3) - q_{12}q_{13}q_{23}
  = \int_{x_1,x_2,x_3} (G(x_1,x_2) - q_{12}) G(x_1,x_3) G(x_2,x_3) \
  dx_1dx_2dx_3 \\
  + q_{12} \int_{x_1,x_2,x_3} (G(x_1,x_3)-q_{13}) G(x_2,x_3) \
  dx_1dx_2dx_3
  + q_{12} q_{13}\int_{x_1,x_2,x_3} (G(x_2,x_3) - q_{13}) \
  dx_1dx_2dx_3.
  \end{multline*}
  The first integral can be bounded below by $-\epsilon p_{12}$ and the
  latter two integrals by $-\epsilon q_{12}$. This gives the desired bound.
\end{proof}

The one-sided counting lemma can be proved by performing
subdivision densification as shown below.
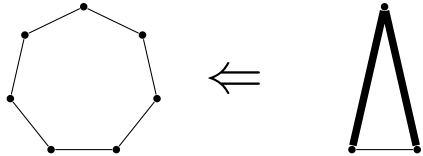
\begin{center}
  \begin{tikzpicture}
    \path[use as bounding box] (-2,-1) rectangle (6,1);
    \begin{scope}
      \foreach \i in {0,1,2,3,4,5,6}{
        \node[p] (\i) at (360/7 *\i + 90:1) {};
        }
      \draw (0)--(1)--(2)--(3)--(4)--(5)--(6)--(0);
    \end{scope}

    \node at (2,0) {\LargeLeftarrow};

    \begin{scope}[shift={(4,0)}]
      \foreach \i in {0,3,4}{
        \node[p] (\i) at (360/7 *\i + 90:1) {};
        }
      \draw[dense] (3)--(0)--(4);
      \draw (3)--(4);
    \end{scope}
  \end{tikzpicture}
\end{center}

\begin{proof}[Proof of Proposition~\ref{prop:cycle-one-sided}]
Let the vertices of $C_\ell$ be $\set{1, 2, \dots, \ell}$ in that
order. Apply subdivision densification
(Lemma~\ref{lem:subdiv-densify}) to the subdivided edge
$(1,2,\dots,\ceil{\ell/2})$, as well as to the subdivided edge
$(\ceil{\ell/2}, \ceil{\ell/2}+1, \dots, \ell)$. Conclude with Lemma~\ref{lem:triangle-single-sparse-edge}.
\end{proof}

\section{Applications} \label{sec:applications}

It is now relatively straightforward to prove our sparse pseudorandom analogues of Tur\'an's theorem, Ramsey's theorem and the graph removal lemma. All of the proofs have essentially the same flavour. We begin by applying the sparse regularity lemma for jumbled graphs, Theorem \ref{thm:sparsereg}. We then apply the dense version of the theorem we are considering to the reduced graph to find a copy of our graph $H$. The counting lemma then implies that our original sparse graph must also contain many copies of $H$.

In order to apply the counting lemma, we will always need to clean up our regular partition, removing all edges which are not contained in a dense regular pair. The following lemma is sufficient for our purposes.

\begin{lemma} \label{lem:cleaning}
For every $\e, \alpha > 0$ and positive integer $m$, there exists $c > 0$ and a positive integer $M$ such that if $\Gamma$ is a $(p, cpn)$-jumbled graph on $n$ vertices then any subgraph $G$ of $\Gamma$ is such that there is a subgraph $G'$ of $G$ with $e(G') \geq e(G) - 4\alpha e(\Gamma)$ and an equitable partition of the vertex set into $k$ pieces $V_1, V_2, \dots, V_k$ with $m \leq k \leq M$  such that the following conditions hold.
\begin{enumerate}
\item
There are no edges of $G'$ within $V_i$ for any $1\leq i \leq k$.

\item
Every non-empty subgraph $(V_i, V_j)_{G'}$ has $d_{G'}(V_i, V_j) = q_{ij} \geq \alpha p$ and satisfies $\DISC(q_{ij}, p, \epsilon)$.
\end{enumerate}
\end{lemma}

\begin{proof}
Let $m_0 = \max(32 \alpha^{-1}, m)$ and $\theta =
\frac{\alpha}{32}$. An application of Theorem \ref{thm:sparsereg}, the
sparse regularity lemma for jumbled graphs, using $\min\set{\theta,
  \epsilon}$ as the parameter $\epsilon$ in the regularity lemma, tells us that there exists an $\eta > 0$ and a positive integer $M$ such that if $\Gamma$ is $(p, \eta p n)$-jumbled then there is an equitable partition of the vertices of $G$ into $k$ pieces with $m_0 \leq k \leq M$ such that all but $\theta k^2$ pairs of vertex subsets $(V_i, V_j)_G$ satisfy $\DISC(q_{ij}, p, \epsilon)$. Let $c = \min(\eta, \frac{1}{8 M^2})$.

Since $\Gamma$ is $(p, \beta)$-jumbled with $\beta \leq c p n$, $c \leq \frac{1}{8 M^2}$ and $n \leq 2 M |V_i|$ for all $i$, the number of edges between $V_i$ and $V_j$ satisfies
\[|e(V_i, V_j) - p|V_i||V_j|| \leq c p n^2 \leq \frac{1}{2} p |V_i||V_j|\]
and thus lies between $\frac{1}{2} p|V_i||V_j|$ and $\frac{3}{2}  p |V_i||V_j|$. Note that this also holds for $i = j$, allowing for the fact that we will count all edges twice.

Therefore, if we remove all edges contained entirely within any $V_i$, we remove at most
$2 p k \left(\frac{2 n}{k}\right)^2 = \frac{8 p n^2}{k}  \leq \frac{\alpha}{4} p n^2$
edges. Here we used that $|V_i| \leq \lceil \frac{n}{k}\rceil \leq \frac{2n}{k}$ for all $i$. If we remove all edges contained within pairs which do not satisfy the discrepancy condition, the number of edges we are removing is at most $2 p \theta k^2 \left(\frac{2n}{k}\right)^2 = 8 p \theta n^2 = \frac{\alpha}{4}  p n^2$. Finally, if we remove all edges contained within pairs whose density is smaller than $\alpha p$, we remove at most $\alpha p \binom{n}{2} \leq \frac{\alpha}{2} p n^2$ edges. Overall, we have removed at most $\alpha p n^2 \leq 4 \alpha e(\Gamma)$ edges. We are left with a graph $G'$ with $e(G') \geq e(G)  - 4 \alpha e(\Gamma)$ edges, as required.
\end{proof}

\subsection{Erd\H{o}s-Stone-Simonovits theorem} \label{sec:app-ESS}

We are now ready to prove the Erd\H{o}s-Stone-Simonovits theorem in
jumbled graphs. We first recall the statement.
Recall that a graph $\Gamma$ is \emph{$(H, \e)$-Tur\'an} if any
subgraph of $\Gamma$ with at least $\paren{1 - \frac{1}{\chi(H) - 1} +
  \e} e(\Gamma)$ edges contains a copy of $H$.

\vspace{2mm} \noindent
{\bf Theorem \ref{TuranIntro}} {\it For every graph $H$ and every $\e > 0$, there exists $c > 0$ such that if $\beta \leq c p^{d_2(H) + 3}n$ then any $(p, \beta)$-jumbled graph on $n$ vertices is $(H, \e)$-Tur\'an.}

\begin{proof}
Suppose that $H$ has vertex set $\{1,2, \dots, m\}$, $\Gamma$ is a $(p, \beta)$-jumbled graph on $n$ vertices, where $\beta \leq c p^{d_2(H) + 3}n$, and $G$ is a subgraph of $\Gamma$ containing at least $\left(1 - \frac{1}{\chi(H) - 1} + \epsilon\right) e(\Gamma)$ edges.

We will need to apply the one-sided counting lemma, Lemma \ref{thm:onesidedintro}, with $\alpha = \frac{\epsilon}{8}$ and $\theta$. We get constants $c_0$ and $\epsilon_0 > 0$ such that if $\Gamma$ is $(p, c_0 p^{d_2(H) + 3} \sqrt{\abs{X_i}\abs{X_j}})$-jumbled and $G$ satisfies $\DISC(q_{ij}, p, \epsilon_0)$, where $\alpha p \leq q_{ij} \leq p$, between sets $X_i$ and $X_j$ for every $1 \leq i < j \leq m$ with $ij \in E(H)$, then $G(H) \geq (1 - \theta) q(H)$.

Apply Lemma \ref{lem:cleaning} with $\alpha = \epsilon/8$ and $\epsilon_0$. This yields constants $c_1$ and $M$ such that if $\Gamma$ is $(p, c_1 p n)$-jumbled then there is a subgraph $G'$ of $G$ with
\[e(G') \geq \left(1 - \frac{1}{\chi(H) - 1} + \epsilon - 4 \alpha\right) e(\Gamma) \geq \left(1 - \frac{1}{\chi(H) - 1} + \frac{\epsilon}{2}\right) e(\Gamma),\]
where we used that $\alpha = \frac{\epsilon}{8}$. Moreover, there is an equitable partition of the vertex set into $k \leq M$ pieces $V_1, \dots, V_k$ such that every non-empty subgraph $(V_i, V_j)_{G'}$ has $d(V_i, V_j) = q_{ij} \geq \alpha p$ and satisfies $\DISC(q_{ij}, p, \epsilon_0)$.

We now consider the reduced graph $R$, considering each piece $V_i$ of the partition as a vertex $v_i$ and placing an edge between $v_i$ and $v_j$ if and only if the graph between $V_i$ and $V_j$ is non-empty. Since $\Gamma$ is $(p, cpn)$-jumbled and $n \leq 2 M |V_i|$, the number of edges between any two pieces differs from $p|V_i||V_j|$ by at most $c p n^2 \leq \frac{\epsilon}{20} p |V_i||V_j|$ provided that $c \leq \frac{\epsilon}{80M^2}$. Note, moreover, that $|V_i| \leq \lceil \frac{n}{k} \rceil \leq (1 + \frac{\epsilon}{20}) \frac{n}{k}$ provided that $n \geq \frac{20 M}{\epsilon}$. Therefore, the number of edges in the reduced graph $R$ is at least
\[e(R) \geq \frac{e(G')}{(1 + \frac{\epsilon}{20}) p \lceil \frac{n}{k} \rceil^2} \geq \frac{(1 - \frac{1}{\chi(H) - 1} + \frac{\epsilon}{2}) e(\Gamma)}{(1 + \frac{\epsilon}{20})^3 p (\frac{n}{k})^2} \geq \left(1 - \frac{1}{\chi(H) - 1} + \frac{\epsilon}{4}\right) \binom{k}{2},\]
where the final step follows from $e(\Gamma) \geq (1 - \frac{\epsilon}{20}) p \binom{n}{2}$.

Applying the Erd\H{o}s-Stone-Simonovits theorem to the reduced graph implies that it contains a copy of $H$. But if this is the case then we have a collection of vertex subsets $X_1, \dots, X_m$ such that, for every edge $ij \in E(H)$, the induced subgraph $(X_i, X_j)_{G'}$ has $d(X_i, X_j) = q_{ij} \geq \alpha p$ and satisfies $\DISC(q_{ij}, p, \epsilon_0)$. By the counting lemma, provided $c \leq \frac{c_0}{2M}$, we have $G(H) \geq G'(H) \geq (1 - \theta) (\alpha p)^{e(H)} (2M)^{-v(H)}$. Therefore, for $c = \min(\frac{c_0}{2M}, c_1, \frac{\epsilon}{80M^2})$, we see that $G$ contains a copy of $H$.
\end{proof}

The proof of the stability theorem, Theorem \ref{StabIntro}, is
similar to the proof of Theorem \ref{TuranIntro}, so we confine
ourselves to a sketch. Suppose that $\Gamma$ is a $(p, \beta)$-jumbled
graph on $n$ vertices, where $\beta \leq c p^{d_2(H) + 3}n$, and $G$ is
a subgraph of $\Gamma$ containing $\left(1 - \frac{1}{\chi(H) - 1} -
  \delta\right) e(\Gamma)$ edges. An application of Lemma
\ref{lem:cleaning} as in the proof above allows us to show that there
is a subgraph $G'$ of $G$ formed by removing at most $\frac{\delta}{4}
p n^2$ edges and a regular partition of $G'$ into $k$ pieces such that
the reduced graph has at least $\left(1 - \frac{1}{\chi(H) - 1} - 2
  \delta\right) \binom{k}{2}$ edges. This graph can contain no copies
of $H$ - otherwise the original graph would have many copies of $H$ as
in the last paragraph above. From the dense version of the stability
theorem~\cite{Si68} it follows that if $\delta$ is sufficiently small then we may
make $R$ into a $(\chi(H) - 1)$-partite graph by removing at most
$\frac{\epsilon}{16} k^2$ edges. We imitate this removal process in
the graph $G'$. That is, if we remove edges between $v_i$ and $v_j$ in
$R$ then we remove all of the edges between $V_i$ and $V_j$ in
$G'$. Since the number of edges between $V_i$ and $V_j$ is at most $2
p |V_i||V_j|$, we will remove at most
\[\frac{\epsilon}{16} k^2 2 p \left\lceil \frac{n}{k} \right\rceil^2 \leq \frac{\epsilon}{2} p n^2\]
edges in total from $G'$. Since we have already removed all edges
which are contained within any $V_i$ the resulting graph is clearly
$(\chi(H) - 1)$-partite. Moreover, the total number of edges removed
is at most $\frac{\delta}{4} p n^2 + \frac{\epsilon}{2} p n^2 \leq
\epsilon p n^2$, as required.

\subsection{Ramsey's theorem} \label{sec:app-ramsey}

In order to prove that the Ramsey property also holds in sparse jumbled graphs, we need the following lemma which says that we may remove a small proportion of edges from any sufficiently large clique and still maintain the Ramsey property.

\begin{lemma} \label{RobustRamsey}
For any graph $H$ and any positive integer $r \geq 2$, there exist $a, \eta > 0$ such that if $n$ is sufficiently large and $G$ is any subgraph of $K_n$ of density at least $1 - \eta$, any $r$-coloring of the edges of $G$ will contain at least $a n^{v(H)}$ monochromatic copies of $H$.
\end{lemma}

\begin{proof}
Suppose first that the edges of $K_n$ have been $r$-colored. Ramsey's theorem together with a standard averaging argument tells us that for $n$ sufficiently large there exists $a_0$ such that there are at least $a_0 n^{v(H)}$ monochromatic copies of $H$. Since $G$ is formed from $K_n$ by removing at most $\eta n^2$ edges, this deletion process will force us to delete at most $\eta n^{v(H)}$ copies of $H$. Therefore, provided that $\eta \leq \frac{a_0}{2}$, the result follows with $a = \frac{a_0}{2}$.
\end{proof}

We also need a slight variant of the sparse regularity lemma, Theorem \ref{thm:sparsereg}, which allows us to take a regular partition which works for more than one graph.

\begin{lemma} \label{lem:colorreg}
For every $\epsilon > 0$ and integers $\ell, m_0 \geq 1$, there exist $\eta > 0$ and a positive integer $M$ such that if $\Gamma$ is a $(p, \eta p n)$-jumbled graph on $n$ vertices and $G_1, G_2, \dots G_\ell$ is a collection of weighted subgraphs of $\Gamma$ then there is an equitable partition into $m_0 \leq k \leq M$ pieces such that for each $G_i$, $1 \leq i \leq \ell$, all but at most $\e k^2$ pairs of vertex subsets $(V_a, V_b)_{G_i}$ satisfy $\DISC(q_{ab}^{(i)}, p, \epsilon)$ for some $q_{ab}^{(i)}$.
\end{lemma}

There is also an appropriate analogue of Lemma \ref{lem:cleaning} to go with this regularity lemma.

\begin{lemma} \label{lem:colorcleaning}
For every $\e, \alpha > 0$ and positive integer $m$, there exist $c > 0$ and a positive integer $M$ such that if $\Gamma$ is a $(p, cpn)$-jumbled graph on $n$ vertices then any collection of subgraphs $G_1, G_2, \dots, G_\ell$ of $\Gamma$ will be such that there are subgraphs $G'_i$ of $G_i$ with $e(G'_i) \geq e(G_i) - 4\alpha e(\Gamma)$ and an equitable partition of the vertex set into $k$ pieces $V_1, V_2, \dots, V_k$ with $m \leq k \leq M$  such that the following conditions hold.
\begin{enumerate}
\item
There are no edges of $G'_i$ within $V_a$ for any $1 \leq i \leq \ell$ and any $1\leq a \leq k$.

\item
Every subgraph $(V_a, V_b)_{G'_i}$ containing any edges from $G'_i$ has $d_{G'_i}(V_a, V_b) = q_{ab}^{(i)} \geq \alpha p$ and satisfies $\DISC(q_{ab}^{(i)}, p, \epsilon)$.
\end{enumerate}
\end{lemma}

The proof of the sparse analogue of Ramsey's theorem now follows along the lines of the proof of Theorem \ref{TuranIntro} above.

\vspace{2mm}\noindent
{\bf Theorem \ref{RamseyIntro}}
{\it For every graph $H$ and every positive integer $r \geq 2$, there exists $c > 0$ such that if $\beta \leq c p^{d_2(H) + 3}n$ then any $(p, \beta)$-jumbled graph on $n$ vertices is $(H, r)$-Ramsey.}

\begin{proof}
Suppose that $H$ has vertex set $\{1,2, \dots, m\}$, $\Gamma$ is a $(p, \beta)$-jumbled graph on $n$ vertices, where $\beta \leq c p^{d_2(H) + 3}n$, and $G_1, G_2, \dots, G_r$ are subgraphs of $\Gamma$ where $G_i$ is the subgraph whose edges have been colored in color $i$.

Let $a, \eta$ be the constants given by Lemma \ref{RobustRamsey}. That is, for $n \geq n_0$, any subgraph of $K_n$ of density at least $1-\eta$ is such that any $r$-coloring of its edges contains at least $a n^{v(H)}$ monochromatic copies of $H$. We will need to apply the one-sided counting lemma, Theorem~\ref{thm:onesidedintro}, with $\alpha = \frac{\eta}{8r}$ and $\theta$. We get constants $c_0$ and $\epsilon_0 > 0$ such that if $\Gamma$ is $(p, c_0 p^{d_2(H) + 3} \sqrt{\abs{X_i}\abs{X_j}})$-jumbled and $G$ satisfies $\DISC(q_{ij}, p, \epsilon_0)$, where $\alpha p \leq q_{ij} \leq p$, between sets $X_i$ and $X_j$ for every $1 \leq i < j \leq m$ with $ij \in E(H)$, then $G(H) \geq (1 - \theta) q(H)$.

We apply Lemma \ref{lem:colorcleaning} to the collection $G_i$ with $\alpha = \frac{\eta}{8r}$,  $\epsilon_0$ and $m = n_0$. This yields $c_1 > 0$ and a positive integer $M$ such that if $\Gamma$ is $(p, c_1pn)$-jumbled then there is a collection of graphs $G'_i$ such that $e(G'_i) \geq e(G_i) - 4\alpha e(\Gamma)$ and every subgraph $(V_a, V_b)_{G'_i}$ containing any edges from $G'_i$ has $d_{G'_i}(V_a, V_b) = q_{ab}^{(i)} \geq \alpha p$ and satisfies $\DISC(q_{ab}^{(i)}, p, \epsilon)$. Adding over all $r$ graphs, we will have removed at most $4r \alpha e(\Gamma) = \frac{\eta}{2} e(\Gamma)$ edges. Let $G'$ be the union of the $G'_i$. This graph has density at least $1 - \frac{\eta}{2}$ in $\Gamma$.

We now consider the colored reduced (multi)graph $R$, considering each piece $V_a$ of the partition as a vertex $v_a$ and placing an edge of color $i$ between $v_a$ and $v_b$ if the graph between $V_a$ and $V_b$ contains an edge of color $i$. Since $\Gamma$ is $(p, cpn)$-jumbled and $n \leq 2 M |V_i|$, the number of edges between any two pieces differs from $p|V_i||V_j|$ by at most $c p n^2 \leq \frac{\eta}{20} p |V_i||V_j|$ provided that $c \leq \frac{\eta}{80M^2}$. Note, moreover, that $|V_i| \leq \lceil \frac{n}{k} \rceil \leq (1 + \frac{\eta}{20}) \frac{n}{k}$ provided that $n \geq \frac{20 M}{\eta}$. Therefore, the number of edges in the reduced graph $R$ is at least
\[e(R) \geq \frac{e(G')}{(1 + \frac{\eta}{20}) p \lceil \frac{n}{k} \rceil^2} \geq \frac{(1 - \frac{\eta}{2}) e(\Gamma)}{(1 + \frac{\eta}{20})^3 p (\frac{n}{k})^2} \geq \left(1 - \eta\right) \binom{k}{2},\]
where the final step follows from $e(\Gamma) \geq (1 - \frac{\eta}{20}) p \binom{n}{2}$.

We now apply Lemma \ref{RobustRamsey} to the reduced graph. Since $k \geq m =n_0$, there exists a monochromatic copy of $H$ in the reduced graph, in color $i$, say. But if this is the case then we have a collection of vertex subsets $X_1, \dots, X_m$ such that, for every edge $ab \in E(H)$, the induced subgraph $(X_a, X_b)_{G'_i}$ has $d_{G'_i}(X_a, X_b) = q^{(i)}_{ab} \geq \alpha p$ and satisfies $\DISC(q^{(i)}_{ab}, p, \epsilon_0)$. By the counting lemma, provided $c \leq \frac{c_0}{2M}$, we have $G(H) \geq G'_i(H) \geq (1 - \theta) (\alpha p)^{e(H)} (2M)^{-v(H)}$. Therefore, for $c = \min(\frac{c_0}{2M}, c_1, n_0^{-1}, \frac{\eta}{80M^2})$, we see that $G$ contains a copy of $H$.
\end{proof}

\subsection{Graph removal lemma} \label{sec:app-removal}

We prove that the graph removal lemma also holds in sparse jumbled graphs. The proof is much the same as the proof for Tur\'an's theorem, though we include it for completeness.

\vspace{2mm}\noindent
{\bf Theorem \ref{RemovalIntro}}
{\it For every graph $H$ and every $\e > 0$, there exist $\delta > 0$ and $c > 0$ such that if $\beta \leq c p^{d_2(H) + 3}n$ then any $(p, \beta)$-jumbled graph $\Gamma$ on $n$ vertices has the following property. Any subgraph of $\Gamma$ containing at most $\delta p^{e(H)} n^{v(H)}$ copies of  $H$ may be made $H$-free by removing at most $\e p n^2$ edges.}

\begin{proof}
Suppose that $H$ has vertex set $\{1,2, \dots, m\}$, $\Gamma$ is a $(p, \beta)$-jumbled graph on $n$ vertices, where $\beta \leq c p^{d_2(H) + 3}n$, and $G$ is a subgraph of $\Gamma$ containing at most $\delta p^{e(H)} n^{v(H)}$ copies of $H$.

We will need to apply the one-sided counting lemma, Lemma \ref{thm:onesidedintro}, with $\alpha = \frac{\epsilon}{16}$ and $\theta = \frac{1}{2}$. We get constants $c_0$ and $\epsilon_0 > 0$ such that if $\Gamma$ is $(p, c_0 p^{d_2(H) + 3} \sqrt{\abs{X_i}\abs{X_j}})$-jumbled and $G$ satisfies $\DISC(q_{ij}, p, \epsilon_0)$, where $\alpha p \leq q_{ij} \leq p$, between sets $X_i$ and $X_j$ for every $1 \leq i < j \leq m$ with $ij \in E(H)$ then $G(H) \geq \frac{1}{2} q(H)$.

Apply Lemma \ref{lem:cleaning} with $\alpha = \epsilon/16$ and $\epsilon_0$. This yields constants $c_1$ and $M$ such that if $\Gamma$ is $(p, c_1 p n)$-jumbled then there is a subgraph $G'$ of $G$ with
\[e(G') \geq e(G) - 4 \alpha e(\Gamma) \geq e(G) - \frac{\epsilon}{4} e(\Gamma) \geq e(G) - \epsilon p n^2,\]
where we used that $\alpha = \frac{\epsilon}{16}$. Moreover, there is an equitable partition into $k \leq M$ pieces $V_1, \dots, V_k$ such that every non-empty subgraph $(V_i, V_j)_{G'}$ has $d(V_i, V_j) = q_{ij} \geq \alpha p$ and satisfies $\DISC(q_{ij}, p, \epsilon_0)$.

Suppose now that there is a copy of $H$ left in $G'$. If this is the case then we have a collection of vertex subsets $X_1, \dots, X_m$ such that, for every edge $ij \in E(H)$, the induced subgraph $(X_i, X_j)_{G'}$ has $d_{G'}(X_i, X_j) = q_{ij} \geq \alpha p$ and satisfies $\DISC(q_{ij}, p, \epsilon_0)$. By the counting lemma, provided $c \leq \frac{c_0}{2M}$, we have $G(H) \geq G'(H) \geq \frac{1}{2} (\alpha p)^{e(H)} (2M)^{-v(H)}$. Therefore, for $c = \min(\frac{c_0}{2M}, c_1)$ and $\delta = \frac{1}{2} \alpha^{e(H)} (2M)^{-v(H)}$, we see that $G$ contains at least $\delta p^{e(H)} n^{v(H)}$ copies of $H$, contradicting our assumption about $G$.
\end{proof}

\subsection{Removal lemma for groups} \label{sec:app-group-removal}

We recall the following removal lemma for groups. Its proof is a
straightforward adaption of the proof of the dense version given by
Kr\'al, Serra and Vena~\cite{KSV09}.

For the rest of this section, let $k_3=3$, $k_4=2$, $k_{m}=1+\frac{1}{m-3}$ if $m \geq 5$ is odd, and $k_m=1+\frac{1}{m-4}$ if $m \geq 6$ is even.

\vspace{2mm}\noindent
{\bf Theorem \ref{thm:removal-groups}}
{\it   For each $\epsilon>0$
  and positive integer $m$, there are $c,\delta>0$ such that the
  following holds. Suppose $B_1,\ldots,B_m$ are subsets of a group $G$ of order $n$
  such that each $B_i$ is $(p,\beta)$-jumbled with $\beta \leq
  cp^{k_m}n$. If subsets $A_i \subseteq B_i$ for $i=1,\ldots,m$ are
  such that there are at most $\delta |B_1|\cdots|B_m|/n$ solutions to
  the equation $x_1x_2 \cdots x_m=1$ with $x_i \in A_i$ for all $i$,
  then it is possible to remove at most $\epsilon |B_i|$ elements from
  each set $A_i$ so as to obtain sets $A_i'$ for which there are no
  solutions to $x_1x_2 \cdots x_m=1$ with $x_i \in A'_i$ for all $i$.
}
\vspace{2mm}

We saw above that the one-sided counting lemma gives the graph removal
lemma. For cycles, the removal lemma follows from
Proposition~\ref{prop:cycle-one-sided}. The version we need is
stated below.

\begin{proposition}
  \label{prop:removal-cycles}
  For every $m \geq 3$ and $\epsilon > 0$, there exist $\delta > 0$ and $c > 0$
  so that any graph $\Gamma$ with vertex subsets $X_1, \dots, X_m$,
  each of size $n$,
  satisfying $(X_i, X_{i+1})_\Gamma$ being $(p, \beta)$-jumbled with
  $\beta \leq c p^{1 + k_m}n$ for each $i = 1, \dots,  m$ (index taken
  mod $m$) has the following
  property. Any subgraph of $\Gamma$ containing at most $\delta p^m
  n^m$ copies of $C_m$ may be made $C_m$-free by removing at most $\e
  p n^2$ edges, where we only consider embeddings of $C_m$ into
  $\Gamma$ where the $i$-th vertex of $C_m$ embeds into $X_i$.
\end{proposition}

\begin{proof}
  [Proof of Theorem \ref{thm:removal-groups}] Let
  $\Gamma$ denote the graph with vertex set $G \x \set{1, \dots, m}$,
  the second coordinate taken modulo $m$, and with vertex $(g, i)$
  colored $i$. Form an edge from $(y,i)$ to $(z,i+1)$ in $\Gamma$ if
  and only if $z = yx_i$ for some $x_i \in B_i$, and let $G_0$ be a
  subgraph of $\Gamma$ consisting of those edges with $x_i \in
  A_i$. Observe that colored $m$-cycles in the graph $G_0$ correspond exactly
  to $(m+1)$-tuples $(y, x_1, x_2, \dots, x_m)$ with $y \in G$ and
  $x_i \in A_i$ for each $i$ satisfying $x_1x_2\dots x_m = 1$. The
  hypothesis implies that there are at most $\delta \abs{B_1}\cdots
  \abs{B_m} \leq \delta 2^mp^m n^m$ colored $m$-cycles in the graph
  $G_0$, where we assumed that $c < \frac12$ so that $\frac{1}{2} pn
  \leq \abs{B_i} \leq \frac{3}{2} pn$ by jumbledness. Then by the
  cycle removal lemma (Proposition~\ref{prop:removal-cycles})
  we can choose $c$ and $\delta$ so that $G_0$ can be made $C_m$-free
  by removing at most $\frac{\epsilon}{2m} pn^2$ edges.

  In $A_i$, remove the element $x_i$ if at least
  $\frac{n}{m}$ edges of the form $(y, i)(yx_i,i+1)$ have been
  removed. Since we removed at most $\frac{\epsilon}{2m} pn^2$ edges,
  we remove at most $\frac{\epsilon}{2} pn \leq
  \epsilon\abs{B_i}$ elements from each $A_i$. Let $A'_i$ denote the
  remaining elements of $A_i$. For any solution to $x_1 x_2
  \cdots x_m =1$ for $x_i \in A_i$, consider the $n$ edge-disjoint
  $m$-cycles $(g,1)(gx_1,2)(gx_1x_2,3)\cdots (gx_1\cdots x_m, m)$ in
  the graph $G_0$ for
  $g \in G$. We must have removed at least one edge from each of the
  $n$ cycles, and so we must have removed at least $\frac{n}{m}$ edges
  of the form $(y,i)(yx_i,i+1)$ for some $i$, which implies that $x_i
  \notin A'_i$. It follows that there is no solution to $x_1x_2\cdots
  x_m = 1$ with $x_i \in A_i$ for all $i$.
\end{proof}

In \cite{KSV09}, the authors also proved removal lemmas for systems of
equations which are \emph{graph representable}. For instance, the
system
\begin{align*}
  x_1x_2x_4^{-1}x_3^{-1} &= 1 \\
  x_1x_2x_5^{-1} &= 1
\end{align*}
can be represented by the graph below, in the sense that solutions to the
above system correspond to embeddings of this graph into some larger
graph with vertex set $G \x \set{1,\dots, 4}$, similar to how
solutions to $x_1 x_2\cdots x_n = 1$ correspond to cycles in the proof
of Theorem~\ref{thm:removal-groups}. We refer to the paper
\cite{KSV09} for the precise statements. These results
can be adapted to the sparse setting in a manner similar to
Proposition~\ref{prop:removal-cycles}.
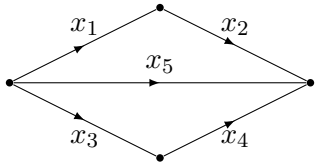
\begin{center}
  \begin{tikzpicture}[decoration={markings,mark=at position .5 with
      {\arrow{latex}}}, d/.style={postaction={decorate}}]
    \node[p] (a) at (-2,0) {};
    \node[p] (b) at (2,0) {};
    \node[p] (c) at (0,1) {};
    \node[p] (d) at (0,-1) {};
    \draw[d] (a) -- node[above]{$x_1$} (c);
    \draw[d] (c) -- node[above]{$x_2$} (b);
    \draw[d] (a) -- node[below]{$x_3$} (d);
    \draw[d] (d) -- node[below]{$x_4$} (b);
    \draw[d] (a) -- node[above]{$x_5$} (b);
  \end{tikzpicture}
\end{center}

\section{Concluding remarks} \label{sec:conclude}

We conclude with discussions on the sharpness of our results,  a sparse extension of quasirandom graphs, induced extensions of the various counting and extremal results, other sparse regularity lemmas, algorithmic applications and sparse Ramsey and Tur\'an-type multiplicity results.

\subsection{Sharpness of results}

We have already noted in the introduction that for every $H$ there are
$(p, \beta)$-jumbled graphs $\Gamma$ on $n$ vertices, with $\beta =
O(p^{d(H) + 2)/4} n)$, such that $\Gamma$ does not contain a copy of
$H$.
On the other hand, the results of Section \ref{sec:oneside} tell us
that we can always find copies of $H$ in $\Gamma$ provided that $\beta
\leq c p^{d_2(H) + 1}$ and in $G$ provided that $\beta \leq c
p^{d_2(H) + 3}$. So, since $d_2(H)$ and $d(H)$ differ by at most a constant factor, our results are sharp up to a
multiplicative constant in the exponent for all $H$. However, we
believe that our results are likely to be sharp up to an additive
constant for the exponent of $p$ in the jumbledness parameter, with some caveats.

An old conjecture of Erd\H{o}s \cite{E67} asserts that if $H$ is a $d$-degenerate bipartite graph then there exists $C > 0$ such that every graph $G$ on $n$ vertices with at least $C n^{2 - \frac{1}{d}}$ edges contains a copy of $H$. This conjecture is known to hold for some bipartite graphs such as $K_{t,t}$ but remains open  in general. The best result to date, due to Alon, Krivelevich and Sudakov \cite{AKS03}, states that if $G$ has $C n^{2 - \frac{1}{4d}}$ edges then it contains a copy of $H$.

If Erd\H{o}s' conjecture is true then this would mean that copies of bipartite $H$ begin to appear already when the density is around $n^{-1/d(H)}$, without any need for a jumbledness condition. If $d_2(H) = d(H) - \frac{1}{2}$ then, even for optimally jumbled graphs, our results only apply down to densities of about $n^{-1/(2d(H) - 1)}$.

However, we considered embeddings of $H$ into $\Gamma$ such that each vertex $\{1,2, \dots, m\}$ of $H$ is to be embedded into a separate vertex subset $X_i$. We believe that in this setting our results are indeed sharp up to an additive constant, even in the case $H$ is bipartite. Without this caveat of embedding each vertex of $H$ into a separate vertex subset in $\Gamma$, we still believe that our results should be sharp for many classes of graphs. In particular, we believe the conjecture \cite{FLS12, KS06, SSV05} that there is a $(p, cp^{t-1} n)$-jumbled graph which does not contain a copy of $K_t$.

One thing which we have left undecided is whether the jumbledness condition for appearance of copies of $H$ in regular subgraphs $G$ of $\Gamma$ should be the same as that for the appearance of copies of $H$ in $\Gamma$ alone. For this question, it is natural to consider the case of triangles where we know that there are $(p, c p^2 n)$-jumbled graphs on $n$ vertices which do not contain any triangles. That is, we know the embedding result for $\Gamma$ is best possible. The result of Sudakov, Szab\'o and Vu \cite{SSV05} mentioned in the introduction also gives us a sharp result for the $(K_3, \epsilon)$-Tur\'an property. In the next subsection, we will obtain a similar sharp bound for the $(K_3,2)$-Ramsey property.

While these Tur\'an and Ramsey-type results are suggestive, we believe that the jumbledness condition for counting in $G$ should be stronger than that for counting in $\Gamma$. The fact that the results mentioned above are sharp is because there are alternative proofs of Tur\'an's theorem for cliques and Ramsey's theorem for the triangle which only need counting results in $\Gamma$ rather than within some regular subgraph $G$. Such a workaround seems unlikely to work for the triangle removal lemma. Kohayakawa, R\"odl, Schacht and Skokan \cite{KRSS10} conjecture that the jumbledness condition in the sparse triangle removal lemma, Theorem \ref{KRSSS}, can be improved from $\beta=o(p^3 n)$ to $\beta=o(p^2n)$. We conjecture that the contrary holds.

\subsection{Relative quasirandomness} \label{sec:rel-quasi}

The study of quasirandom graphs began in the pioneering work of Thomason \cite{T85,T87} and  Chung, Graham, and Wilson \cite{CGW89}.  As briefly discussed in Section \ref{sec:pseudos},
they showed that a large number of interesting graph properties satisfied by random graphs are all equivalent.
Perhaps the most surprising aspect of this work is that if the number of cycles of length 4 in a graph is as one would expect in a binomial random graph of the same density, then this is enough to imply that the edges are very well-spread and the number of copies of {\it any} fixed graph is as one would expect in a binomial random graph of the same  density.

There has been a considerable amount of research aimed at extending quasirandomness to sparse graphs, see \cite{CG02,CG08,KR03,KRS04}. However, the key property of counting small subgraphs was missing from previous results in this area. The following theorem  extends the fundamental results in this area to the setting of subgraphs of (possibly sparse) pseudorandom graphs. The case $p=1$ and $\Gamma$ is the complete graph corresponds to the original setting. We prove that the natural analogues of the original quasirandom properties in this more general setting are all equivalent.  Of particular note is the inclusion of the count of small subgraphs, a key property missing from previous results in this area. The proof of some of the implications extend easily from the dense case. However, to imply the notable counting properties, we use the counting lemma, Theorem \ref{thm:sparse-counting}, which acts as a transference principle from the sparse setting to the dense setting.

Such quasirandomness of a structure within a sparse but pseudorandom structure is known as {\it relative quasirandomness}. This concept has been instrumental in the development of the hypergraph regularity and counting lemma \cite{G07, NRS06, RS04, T06}. In the $3$-uniform case, for example, one repeatedly has to deal with $3$-uniform hypergraphs which are subsets of the triangles of a very pseudorandom graph.

To keep the theorem statement simple, we first describe some notation. The co-degree $d_G(v,v')$ of two vertices $v,v'$ in a graph $G$ is the number of vertices which are adjacent to both $v$ and $v'$.
For a graph $H$, we let $s(H)=  \min\set{\frac{\Delta(L(H)) + 4}{2}, \frac{d(L(H)) + 6}{2}}$. For a graph $H$ and another graph $G$, let $N_H(G)$ denote the number of labeled copies of $H$ (as a subgraph) in $G$.

\begin{theorem}\label{thm:relquas} Let $k \geq 2$ be a positive integer.
For $n \geq 1$, let $\Gamma=\Gamma_n$ be a $(p,\beta)$-jumbled graph on $n$ vertices with $p=p(\Gamma)$ and $\beta=\beta(\Gamma)=o(p^k n)$, and $G=G_n$ be a spanning subgraph of $\Gamma_n$.
The following are equivalent.
\begin{itemize}
\item[$P_1$:] For all vertex subsets $S$ and $T$, $$\left|e_{G}(S,T)-q|S||T|\right|=o(pn^2).$$
\item[$P_2$:] For all vertex subsets $S$, $$\left|e_{G}(S)-q\frac{|S|^2}{2}\right|=o(pn^2).$$
\item[$P_3$:] For all vertex subsets $S$ with $|S|=\lfloor \frac{n}{2} \rfloor$,
  $$\left|e_{G}(S)-q\frac{n^2}{8}\right|=o(pn^2).$$
\item[$P_4$:] For each graph $H$ with $k \geq s(H)$, $$N_H(G)=q^{e(H)}n^{v(H)}+o(p^{e(H)}n^{v(H)}).$$
\item[$P_5$:] $e(G) \geq q\frac{n^2}{2}+o(pn^2)$ and $$N_{C_4}(G) \leq q^{4}n^{4}+o(p^{4}n^{4}).$$
\item[$P_6$:] $e(G) \geq (1+o(1))q\frac{n^2}{2}$, $\lambda_1=(1+o(1))qn$, and $\lambda_2=o(pn)$, where $\lambda_i$ is the $i$th largest eigenvalue, in absolute value, of the adjacency matrix of $G$. \\
\item[$P_7$:] $$\sum_{v,v' \in V(G)} |d_{G}(v,v')-q^2n|=o(p^2n^3).$$
\end{itemize}
\end{theorem}

We briefly describe how to prove the equivalences between the various
properties in Theorem~\ref{thm:relquas}, with a flow chart shown
below.
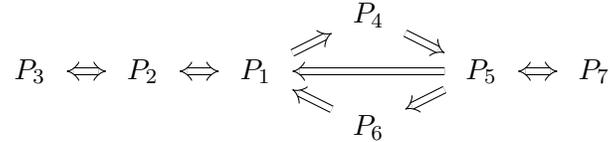
\begin{center}
  \begin{tikzpicture}
    [scale=1.5, every node/.style={minimum width = 2.5em}, every path/.style={double}, >=implies, double distance = .2em]
    \node (1) at (0,0) {$P_1$};
    \node (2) at (-1,0) {$P_2$};
    \node (3) at (-2,0) {$P_3$};
    \node (4) at (1,.5) {$P_4$};
    \node (5) at (2,0) {$P_5$};
    \node (6) at (1,-.5) {$P_6$};
    \node (7) at (3,0) {$P_7$};
    \draw[<->] (1) -- (2);
    \draw[<->] (2) -- (3);
    \draw[->] (1)--(4);
    \draw[->] (4)--(5);
    \draw[->] (5)--(1);
    \draw[->] (5)--(6);
    \draw[->] (6)--(1);
    \draw[<->] (5)--(7);
  \end{tikzpicture}
\end{center}
The equivalence between the discrepancy properties  $P_1$,
$P_2$, $P_3$ is fairly straightforward and similar to the dense case. Theorem \ref{thm:sparse-counting} shows that $P_1$ implies $P_4$. As $P_5$ is a special case of $P_4$, we have that $P_4$ implies $P_5$. Proposition \ref{prop:C4-DISC} shows that $P_5$ implies $P_1$. The fact $P_5$ implies  $P_6$ follows easily from the identity that the trace of the fourth power of the adjacency matrix of $G$ is both the number of closed walks in $G$ of length
$4$, and the sum of the fourth powers of the eigenvalues of the adjacency matrix of $G$. The fact that $P_6$ implies $P_1$ is the standard proof of the expander mixing lemma. The fact $P_5$ implies $P_7$ follows easily from the identity  \begin{equation}\label{eq:C4codeg}N_{C_4}(G)=4\sum_{v,v' \in V(G)}{d_{G}(v,v') \choose 2},\end{equation}
where the sum is over all ${n \choose 2}$ pairs of distinct vertices, as well as the identity
$$\sum_{v,v'}d_G(v,v')=\sum_{v} {d_G(v) \choose 2},$$
and two applications of the Cauchy-Schwarz inequality. Finally, we have $P_7$ implies $P_5$ for the following reason. From (\ref{eq:C4codeg}), we have $P_5$ is equivalent to  \begin{equation}\label{p7p5}\sum_{v,v'} d_G(v,v')^2 = \frac{1}{2}q^4n^4+o(p^4n^4).\end{equation}
To verify (\ref{p7p5}), we split up the sum on the left into three sums. The first sum is over pairs $v,v'$ with $|d_G(v,v')-q^2n|=o(p^2n)$, the second sum is over pairs $v,v'$ with $d_G(v,v')>2p^2 n$, and the third sum is over the remaining pairs $v,v'$. From $P_7$, almost all pairs $v,v'$ of vertices satisfy $|d_{G}(v,v')-q^2 n|=o(p^2n)$, and so the first sum is $\frac{1}{2}q^4n^4+o(p^4n^4)$. The second sum satisfies
$$\sum_{v,v':d_G(v,v')>2p^2n}d_G(v,v')^2 \leq \sum_{v,v':d_{\Gamma}(v,v')>2p^2n}d_{\Gamma}(v,v')^2=o(p^4n^4),$$
where the first inequality follows from $G$ is a subgraph of $\Gamma$, and the second inequality follows from pseudorandomness in $\Gamma$. Finally, as $P_7$ implies there are $o(n^2)$ pairs $v,v'$ not satisfying $|d_G(v,v')-q^2n|=o(p^2n)$, and the terms in the third sum are at most $2p^2 n$, the third sum is $o(p^4n^4)$. This completes the proof sketch of the equivalences between the various properties in Theorem \ref{thm:relquas}.

\subsection{Induced extensions of counting lemmas and extremal
  results}

With not much extra effort, we can establish induced versions of the various counting lemmas and extremal results for graphs. We assume that we are in Setup \ref{set:G} with the additional condition that, in Setup \ref{set:Gamma}, the graph $\Gamma$ satisfies the jumbledness condition for all pairs $ab$ of vertices and not just the sparse edges of $H$. Define a {\it strongly induced copy} of $H$ in $G$ to be a copy of $H$ in $G$ such that the nonedges of the copy of $H$ are nonedges of $\Gamma$. Since $G$ is a subgraph of $\Gamma$, a strongly induced copy of $H$ is an induced copy of $H$. Define
  \[
  G^*(H) := \int_{x_1 \in X_1, \dots, x_m \in X_m} \prod_{(i,j) \in E(H)} G(x_i,x_j)\prod_{(i,j) \not \in E(H)} (1-\Gamma(x_i,x_j)) \ dx_1 \cdots dx_m
  \]
  and
  \[
  q^*(H) := \prod_{(i,j) \in E(H)} q_{ij}\prod_{(i,j) \not \in E(H)}(1-p_{ij})
  \]
Note that $G^*(H)$ is the probability that a random compatible map forms a strongly induced copy of $H$, and $q^*(H)$ is the idealized version. Also note that  if $\Gamma$ is $(p,\beta)$-jumbled, then its complement $\bar \Gamma$ is $(1-p,\beta)$-jumbled. Hence, for $p$ small, we expect that most copies of $H$  guaranteed by Theorem \ref{thm:onesidedintro} are strongly induced. This is formalized in the following theorem, which is an induced analogue of the one-sided counting lemma, Theorem \ref{thm:onesidedintro}.

\begin{theorem}\label{induced-oneside}
For every fixed graph $H$ on vertex set $\{1,2, \dots, m\}$ and every $\theta > 0$, there exist constants $c > 0$ and $\epsilon > 0$ such that the following holds.

Let $\Gamma$ be a graph with vertex sets $X_1, \dots, X_m$ and suppose that $p \leq \frac{1}{m}$ and the bipartite graph $(X_i, X_j)_\Gamma$ is $(p, cp^{d(H) + \frac{5}{2}} \sqrt{\abs{X_i}\abs{X_j}})$-jumbled for every $i < j$. Let $G$ be a subgraph of $\Gamma$, with the vertex $i$ of $H$ assigned to the vertex set $X_i$ of $G$. For each edge $ij$ in $H$, assume that $(X_i, X_j)_G$ satisfies $\DISC(q_{ij}, p, \epsilon)$. Then
 \[ G^*(H) \geq (1 - \theta) q^*(H). \]
\end{theorem}

We next discuss how the proof of Theorem \ref{induced-oneside} is a minor modification of the proof of Theorem \ref{thm:onesidedintro}.  As in the proof of Theorem \ref{thm:onesidedintro}, after $j-1$ steps in the embedding, we have picked $f(v_1),\ldots,f(v_{j-1})$ and have subsets $T(i,j-1) \subset X_i$ for $j \leq i \leq m$ which consist of the possible vertices for $f(v_i)$ given the choice of the first $j-1$ embedded vertices. We are left with the task of picking a good set $W(j) \subset T(j,j-1)$ of possible vertices $w=f(v_j)$ to continue the embedding with the desired properties. We will guarantee, in addition to the three properties stated there (which may be maintained since $d_2(H) + 3 \leq d(H) + \frac{5}{2}$), that

\vspace{0.2cm}
4. $|N_{\bar \Gamma}(w) \cap T(i,j-1)| \geq (1-p-\sqrt{c})|T(i,j-1)|$
for each $i>j$ which is not adjacent to $j$.
\vspace{0.2cm}

As for each such $i$, if $w$ is chosen for $f(v_j)$, $T(i,j)=N_{\bar
  \Gamma}(w) \cap T(i,j-1)$, this will guarantee that $|T(i,j)| \geq
(1-p-\sqrt{c})|T(i,j-1)|$. As for each such $i$, the set $T(i,j)$ is
only slightly smaller than $T(i,j-1)$, this will affect the
discrepancy between each pair of sets by at most a factor
$(1-p-\sqrt{c})^2$. This additional fourth property makes the set
$W(j)$ only slightly smaller. Indeed, to guarantee this property, we
need that for each of the nonneighbors $i>j$ of $j$, the vertices $w$
with fewer than $(1-p-\sqrt{c})|T(i,j-1)|$ nonneighbors in $T(i,j-1)$
in graph $\bar \Gamma$ are not in $W(j)$, and there are at most
$\frac{\beta^2}{c |T(i,j-1)|}$ such vertices for each $i$ by
Lemma~\ref{lem:deg-est}. As there are at most $m$ choices of $i$, and
$|T(i,j-1)| \geq \left(1-\frac{\theta_{j-1}}{6}\right)q(i,j-1)|X_i|$,
we get that satisfying this additional fourth property requires that
the number of additional vertices deleted to obtain $W(j)$ is at most
\[
m\frac{\beta^2}{c|T(i,j-1)|} \leq \frac{mcp^{2d(H)+5}}{
  \left(1-\frac{\theta_{j-1}}{6}\right)q(i,j-1)}|X_j| \leq \frac{mcp^{2d(H)+5}}{
  \left(1-\frac{\theta_{j-1}}{6}\right)^2q(i,j-1) q(j, j-1)}|T(j, j-1)|,
\]
which is neglible since both $q(i, j-1)$ and $q(j, j-1)$ are at most $p^{d(H)}$. We therefore see that, after changing the various
parameters in the proof of Theorem \ref{thm:onesidedintro} slightly,
the simple modification of the proof sketched above completes the
proof of Theorem \ref{induced-oneside}. We remark that the assumption
$p \leq \frac{1}{m}$ can be replaced by $p$ is bounded away from $1$,
which is needed as we must guarantee that the nonedges of the induced
copy of $H$ must be nonedges of $\Gamma$. We also note that the exponent of $p$ in the jumbledness assumption in Theorem \ref{induced-oneside} is $d(H)+\frac{5}{2}$ and not $d_2(H)+3$ for the following reason. In addition to using the inheritance of regularity to get that edges of $H$ map to edges of $G$ in the strongly induced copies of $H$ we are counting, we also use jumbledness of $\bar \Gamma$ to get that the nonedges of $H$ map to nonedges of $\Gamma$.

In the greedy proof sketched above, to conclude that the good set $W(j) \subset T(j,j-1)$ of possible vertices $w=f(v_j)$ is large, we use the jumbledness of $\bar \Gamma$ and that, for the vertices $i>j$ not adjacent to $j$, the product $|T(j,j-1)||T(i,j-1)|$ is large. The sizes of these sets are related to the number of neighbors of $j$ less than $j$ and the number of neighbors of $i$ less $j$, respectively.

The induced graph removal lemma was proved by Alon, Fischer,
Krivelevich, and Szegedy \cite{AFKS00}. It states that for each graph
$H$ and $\epsilon>0$ there is $\delta>0$ such that every graph on $n$
vertices with at most $\delta n^{v(H)}$ induced copies of $H$ can be
made induced $H$-free by adding or deleting at most $\epsilon n^2$
edges. This clearly extends the original graph removal lemma. To prove
the induced graph removal lemma, they developed the strong regularity
lemma, whose proof involves iterating Szemer\'edi's regularity lemma
many times. A new proof of the induced graph removal lemma which gives
an improved quantitative estimate was recently obtained in \cite{CF12}.

The first application of Theorem \ref{induced-oneside} we discuss is an induced extension of the sparse graph removal, Theorem \ref{RemovalIntro}. It does not imply the induced graph removal lemma above.

\begin{theorem} \label{inducedremoval}
For every graph $H$ and every $\e > 0$, there exist $\delta > 0$ and $c > 0$ such that if $\beta \leq c p^{d(H) + \frac{5}{2}}n$ then any $(p, \beta)$-jumbled graph $\Gamma$ on $n$ vertices with $p \leq \frac{1}{v(H)}$ has the following property. Any subgraph of $\Gamma$ containing at most $\delta p^{e(H)} n^{v(H)}$ (strongly) induced copies of $H$ may be made $H$-free by removing at most $\e p n^2$ edges.
\end{theorem}

The proof of Theorem \ref{inducedremoval} is the same as the proof of Theorem \ref{RemovalIntro}, except we replace the one-sided counting lemma, Theorem \ref{thm:onesidedintro}, with its induced variant, Theorem \ref{induced-oneside}. Note that unlike the standard induced graph removal lemma, here it suffices only to delete edges. Furthermore, all copies of $H$, not just induced copies, are removed by the deletion of few edges.

The induced Ramsey number $r_{\textrm{ind}}(H;r)$ is the smallest natural
number $N$ for which there is a graph $G$ on $N$ vertices such that in every $r$-coloring of the edges of
$G$ there is an induced monochromatic copy of $H$. The existence of these numbers was independently
proven in the early 1970s by Deuber \cite{D75}, Erd\H{o}s, Hajnal
and Posa \cite{EHP75}, and R\"odl \cite{R73}. The bounds that these
original proofs give on $r_{\textrm{ind}}(H,r)$ are enormous. However,
Trotter conjectured that the induced Ramsey number of bounded degree
graphs is at most polynomial in the number of vertices. That is, for
each $\Delta$ there is $c(\Delta)$ such that $r_{\textrm{ind}}(H;2)
\leq n^{c(\Delta)}$. This was proved by \L uczak and R\"odl
\cite{LR96}, who gave an enormous upper bound on $c(\Delta)$, namely, a
tower of twos of height $O(\Delta^2)$. More recently, Fox and Sudakov
\cite{FS08} proved an upper bound on $c(\Delta)$ which is $O(\Delta
\log \Delta)$. These results giving a polynomial bound on the induced
Ramsey number of graphs of bounded degree do not appear to extend to more than two colors.

A graph $G$ is \emph{induced Ramsey $(\Delta,n,r)$-universal} if, for every $r$-edge-coloring of $G$, there is a color for which there is a monochromatic induced copy in that color of every graph on $n$ vertices with maximum degree $\Delta$. Clearly, if $G$ is induced Ramsey $(\Delta,n,r)$-universal, then $r_{\textrm{ind}}(H;r) \leq |G|$ for every graph $H$ on $n$ vertices with maximum degree $\Delta$.

\begin{theorem}\label{inducedramseyuniversal}
For each $\Delta$ and $r$ there is $C=C(\Delta,r)$ such that for every $n$ there is an induced Ramsey $(\Delta,n,r)$-universal graph on at most $Cn^{2\Delta+8}$ vertices.
\end{theorem}

The exponent of $n$ in the above result is best possible up to a multiplicative factor. This is because even for the much weaker condition that $G$ contains an induced copy of all graphs on $n$ vertices with maximum degree $\Delta$, the number of vertices of $G$ has to be $\Omega(n^{\Delta/2})$ (see, e.g., \cite{B09}).

We have the following immediate corollary of Theorem \ref{inducedramseyuniversal}, improving the bound for induced Ramsey numbers of bounded degree graphs. It is also the first polynomial upper bound which works for more than two colors.

\begin{corollary}
For each $\Delta$ and $r$ there is $C=C(\Delta,r)$ such that $r_{\textrm{ind}}(H;r) \leq Cn^{2\Delta+8}$ for every $n$-vertex graph $H$ of maximum degree $\Delta$.
\end{corollary}

We next sketch the proof of Theorem \ref{inducedramseyuniversal}. The proof builds on ideas used in the proof of Chvatal, R\"odl, Szemer\'edi, and Trotter \cite{CRST83} that Ramsey numbers of bounded degree graphs grow linearly in the number of vertices. We claim that any graph $G$ on $N=Cn^{2\Delta+8}$ vertices which is $(p,\beta)$-jumbled with $p=\frac{1}{n}$ and $\beta=O(\sqrt{pN})$ is
the desired induced Ramsey $(\Delta,n,r)$-universal graph. Such a
graph exists as almost surely $G(N,p)$ has this jumbledness
property. Note that $\beta = cp^{d(H)+\frac{5}{2}} N$ with $c =
O(p^2)$. We consider an $r$-coloring of the edges of $G$ and apply
the multicolor sparse regularity lemma so that each color satisfies a
discrepancy condition between almost all pairs of parts. Using
Tur\'an's theorem and Ramsey's theorem in the reduced graph, we find
$\Delta+1$ parts $X_1,\ldots,X_{\Delta+1}$, each pair of which has
density at least $\frac{p}{2r}$ in the same color, say red, and
satisfies a discrepancy condition. Let $H$ be a graph on $n$ vertices
with maximum degree $\Delta$, so $H$ has chromatic number at most
$\Delta+1$. Assign each vertex $a$ of $H$ to some part so that the
vertices assigned to each part form an independent set in $H$.  We
then use the induced counting lemma, Theorem~\ref{induced-oneside}, to
get an induced monochromatic red copy of $H$. We make a couple of
observations which are vital for this proof to work, and one must look
closely into the proof of the induced counting lemma to verify these
claims. First, we can choose the constants in the regularity lemma and
the counting lemma so that they only depend on the maximum degree $\Delta$ and the
number of colors $r$ and not on the number $n$ of vertices. Indeed, in
addition to the at most $2\Delta$ times that we apply inheritance of
regularity, the discrepancy-parameter increases by a factor of at most
$(1 - p- \sqrt{c})^{-2n} = (1-O(p))^{-2n} = (1 - O(\frac{1}{n}))^{-2n}
= O(1)$ due to the restrictions imposed by the non-edges of $H$. So we
lose a total of at most a constant factor in the discrepancy, which
does not affect the outcome.  Second, as we assigned some vertices of
$H$ to the same part, they may get embedded to the same
vertex. However, one easily checks that almost all the embeddings of
$H$ in the proof of the induced counting lemma are one-to-one, and
hence there is a monochromatic induced copy of $H$. Indeed, as there
are less than $n$ vertices which are previously embedded at each step
of the proof of the induced counting lemma, and $W(j) \gg n$, then
there is always a vertex $w \in W(j)$ to pick for $f(v_j)$ to continue
the embedding. This completes the proof sketch.

In the proof sketched above, the use of the sparse regularity lemma forces an enormous upper bound on $C(\Delta,r)$, of tower-type. However, all we needed was $\Delta+1$ parts such that the graph between each pair of parts has density at least $\frac{p}{2r}$ in  the same color and satisfies a discrepancy condition. To guarantee this, one does not need the full strength of the regularity lemma, and the sparse version of the Duke-Lefmann-R\"odl weak regularity lemma discussed in Subsection \ref{othersparsereg} is sufficient. This gives a better bound on $C(\Delta,r)$, which is an exponential-tower of constant height.

The last application we mention is an induced extension of the sparse Erd\H{o}s-Stone-Simonovits theorem, Theorem \ref{TuranIntro}. We say that a graph $\Gamma$ is {\it induced $(H,\epsilon)$-Tur\'an} if any subgraph of $\Gamma$ with at least $(1-\frac{1}{\chi(H)-1}+\epsilon)e(\Gamma)$ edges contains a strongly  induced copy of $H$.

\begin{theorem} \label{inducedturan}
For every graph $H$ and every $\e > 0$, there exists $c > 0$ such that if $\beta \leq c p^{d(H) + \frac{5}{2}}n$ then any $(p, \beta)$-jumbled graph on $n$ vertices with $p \leq \frac{1}{v(H)}$ is induced $(H, \e)$-Tur\'an.
\end{theorem}

The proof of Theorem \ref{inducedturan} is the same as the proof of Theorem \ref{TuranIntro}, except we replace the one-sided counting lemma, Theorem \ref{thm:onesidedintro}, with its induced variant, Theorem \ref{induced-oneside}.

\subsection{Other sparse regularity lemmas}\label{othersparsereg}

The sparse regularity lemma, in the form due to Scott \cite{Sc11},
states that for every $\epsilon > 0$ and positive integer $m$, there
exists a positive integer $M$ such that every graph $G$ has an
equitable partition into $k$ pieces $V_1, V_2, \dots, V_k$ with $m
\leq k \leq M$ such that all but $\epsilon k^2$ pairs $(V_i, V_j)_G$
satisfy $\DISC(p_{ij}, p_{ij}, \epsilon)$ for some $p_{ij}$. The
additional condition of jumbledness which we imposed in our regularity
lemma, Theorem \ref{thm:sparsereg}, was there so as to force all of
the $p_{ij}$ to be $p$. If this were not the case, it could easily be
that all of the edges of the graph bunch up within a particular bad
pair, so the result would tell us nothing.

In our results, we made repeated use of sparse regularity. While
convenient, this does have its limitations. In particular, the bounds
which the regularity lemma gives on the number of pieces $M$ in the
regular partition is (and is necessarily \cite{CF12, Gow97}) of
tower-type in $\epsilon$. This means that the constants $c^{-1}$ which
this method produces for Theorems \ref{RemovalIntro},
\ref{TuranIntro}, \ref{RamseyIntro}, and \ref{StabIntro} are also of tower-type.

In the dense setting, there are other sparse regularity lemmas which prove sufficient for many of our applications. One such example is the cylinder regularity lemma of Duke, Lefmann and R\"odl \cite{DLR95}. This lemma says that for a $k$-partite graph, between sets $V_1, V_2, \dots, V_k$, there is an $\epsilon$-regular partition of the cylinder $V_1 \times \cdots \times V_k$ into a relatively small number of cylinders $K = W_1 \times \cdots \times W_k$, with $W_i \subseteq V_i$ for $1 \leq i \leq k$. The definition of an $\epsilon$-regular partition here is that all but an $\epsilon$-fraction of the $k$-tuples $(v_1, \dots, v_k) \in V_1 \times \cdots \times V_k$ are in $\epsilon$-regular cylinders, where a cylinder $W_1\times \cdots \times W_k$ is $\epsilon$-regular if all $\binom{k}{2}$ pairs $(W_i,W_j)$, $1\leq i < j \leq k$, are $\epsilon$-regular in the usual sense.

For sparse graphs, a similar theorem may be proven by appropriately adapting the proof of Duke, Lefmann and R\"odl using the ideas of Scott. Consider a $k$-partite graph, between sets  $V_1, V_2, \dots, V_k$. We will say that a cylinder $K = W_1 \times \cdots \times W_k$, with $W_i \subseteq V_i$ for $1 \leq i \leq k$, satisfies $\DISC(q_K, p_K, \epsilon)$ with $q_K = (q_{ij})_{1 \leq i < j \leq k}$ and $p_K = (p_{ij})_{1 \leq i < j \leq k}$ if all $\binom{k}{2}$ pairs $(W_i,W_j)$, $1\leq i < j \leq k$, satisfy $\DISC(q_{ij}, p_{ij}, \epsilon)$. The sparse version of the cylinder regularity lemma is now as follows.

\begin{proposition} \label{prop:cylinderreg}
For every $\epsilon > 0$ and positive integer $k$, there exists $\gamma > 0$ such that if G = (V,E) is a $k$-partite graph with $k$-partition $V = V_1 \cup \dots \cup V_k$ then there exists an $\epsilon$-regular partition $\mathcal{K}$ of $V_1\times \cdots \times V_k$ into at most $\gamma^{-1}$ cylinders such that, for each $K \in \mathcal{K}$ with $K = W_1 \times \dots \times W_k$ and $1 \leq i \leq k$,  $|W_i| \geq \gamma|V_i|$.
\end{proposition}

The constant $\gamma$ is at most exponential in a power of $k \epsilon^{-1}$. Moreover, this theorem is sufficient for our applications to Tur\'an's theorem and Ramsey's theorem. This results in constants $c^{-1}$ which are at most double exponential in the parameters $|H|$, $\epsilon$ and $r$ for Theorems \ref{TuranIntro} and \ref{RamseyIntro}.

For the graph removal lemma, we may also make some improvement, but it is of a less dramatic nature. As in the dense case \cite{F11}, it shows that in Theorem \ref{RemovalIntro} we may take $\delta^{-1}$ and $c^{-1}$ to be a tower of twos of height logarithmic in $\epsilon^{-1}$. The proof essentially transfers to the sparse case using the sparse counting lemma, Theorem \ref{thm:onesidedintro}.

\subsection{Algorithmic applications}

The algorithmic versions of Szemer\'edi's regularity lemma and its variants have applications to fundamental algorithmic problems such as max-cut, max $k$-sat, and property testing (see \cite{ASS09} and its references).  The result of Alon and Naor \cite{AN06} approximating the cut-norm of a graph via Grothendieck's inequality  allows one to obtain algorithmic versions of Szemeredi's regularity lemma \cite{ACHKRS10}, the Frieze-Kannan weak regularity lemma \cite{CCF09}, and the Duke-Lefmann-R\"odl weak regularity lemma. Many of these algorithmic applications can be transferred to the sparse setting using algorithmic versions of the sparse regularity lemmas, allowing one to  substantially improve the error approximation in this setting. Our new counting lemmas allows for further sparse extensions. We describe one such extension below.

Suppose we are given a graph $H$ on $h$ vertices, and we want to compute the number of copies of $H$ in a graph $G$ on $n$ vertices. The brute force approach considers all possible $h$-tuples of vertices and computes the desired number in time $O(n^h)$. The Duke-Lefmann-R\"odl regularity lemma was originally introduced in order to obtain a much faster algorithm, which runs in polynomial time with an absolute constant exponent, at the expense of some error. More precisely, for each $\epsilon>0$, they found an algorithm which, given a graph on $n$ vertices, runs in polynomial time and approximates the number of copies of $H$ as a subgraph to within $\epsilon n^h$. The running time is of the form $C(h,\epsilon)n^{c}$, where $c$ is an absolute constant and $C(h,\epsilon)$ is exponential in a power of $h\epsilon^{-1}$. We have the following extension of this result to the sparse setting. The proof transfers from the dense setting using  the algorithmic version of the sparse Duke-Lefmann-R\"odl regularity lemma, Proposition \ref{prop:cylinderreg}, and the sparse counting lemma, Theorem \ref{thm:sparse-counting}. For a graph $H$, we let $s(H)=  \min\set{\frac{\Delta(L(H)) + 4}{2}, \frac{d(L(H)) + 6}{2}}$.

\begin{proposition}
Let $H$ be a graph on $h$ vertices with $s(H) \leq k$ and $\epsilon>0$. There is an absolute constant $c$ and another constant $C=C(\epsilon,h)$ depending only exponentially on $h\epsilon^{-1}$ such that the following holds. Given a graph $G$ on $n$ vertices which is known to be a spanning subgraph of a $(p,\beta)$-pseudorandom graph with $\beta \leq C^{-1}p^k n$, the number of copies of $H$ in $G$ can be computed up to an error $\epsilon p^{e(H)}n^{v(H)}$ in running time $Cn^c$.
\end{proposition}

\subsection{Multiplicity results}

There are many problems and results in graph Ramsey theory and extremal graph theory on the multiplicity of subgraphs. These results can be naturally extended to sparse pseudorandom graphs using the tools developed in this paper. Indeed, by applying the sparse regularity lemma and the new counting lemmas, we get extensions of these results to sparse graphs. In this subsection, we discuss a few of these results.

Recall that Ramsey's theorem states that every $2$-edge-coloring of a
sufficiently large complete graph $K_n$ contains at least one
monochromatic copy of a given graph $H$. Let $c_{H,n}$ denote the
fraction of copies of $H$ in $K_n$ that must be monochromatic in any
$2$-edge-coloring of $G$. By an averaging argument, $c_{H,n}$ is a
bounded, monotone increasing function in $n$, and therefore has a
positive limit $c_H$ as $n \to \infty$. The constant $c_H$ is known as
the {\it Ramsey multiplicity constant} for the graph $H$. It is simple
to show that $c_H \leq 2^{1-m}$ for a graph $H$ with $m=e(H)$ edges,
where this bound comes from considering a random $2$-edge-coloring of
$K_n$ with each coloring equally likely.

Erd\H{o}s \cite{E62} and, in a more general form, Burr and Rosta \cite{BR80} suggested that the Ramsey multiplicity constant
is achieved by a random coloring. These conjectures are false as was demonstrated by Thomason \cite{T89}
even for $H$ being any complete graph $K_t$ with $t \geq 4$. Moreover, as shown in \cite{F08}, there are graphs $H$ with $m$ edges and $c_H \leq m^{-m/2+o(m)}$, which demonstrates that the random coloring is far from being optimal for some graphs.

For bipartite graphs the situation seems to be very different. The edge density of a graph is the
fraction of pairs of vertices that are edges. The conjectures of Erd\H{o}s-Simonovits \cite{Si84} and Sidorenko \cite{Si93}
suggest that for any bipartite $H$ the number of copies of $H$ in any graph $G$ on $n$ vertices with edge
density $p$ bounded away from $0$ is asymptotically at least the same as in the $n$-vertex random graph with edge density $p$. This conjecture implies that
$c_H=2^{1-m}$ if $H$ is bipartite with $m$ edges. The most general results on this problem were obtained in \cite{CFS10} and \cite{LS12}, where it is shown that every bipartite graph $H$ which has a vertex in one part complete to the other part satisfies the conjecture.

More generally, let $c_{H,\Gamma}$ denote the fraction of copies of $H$ in $\Gamma$ that must be monochromatic in any $2$-edge-coloring of $\Gamma$. For a graph $\Gamma$ with $n$ vertices, by averaging over all copies of $\Gamma$ in $K_n$, we have $c_{H,\Gamma} \leq c_{H,n} \leq c_H$. It is natural to try to find conditions on $\Gamma$ which imply that $c_{H,\Gamma}$ is close to $c_H$.  The next theorem shows that if $\Gamma$ is sufficiently jumbled, then $c_{H,\Gamma}$ is close to $c_H$. The proof follows by noting that the proportion of monochromatic copies of $H$ in the weighted reduced graph $R$ is at least $c_{H, |R|}$. This count then transfers back to $\Gamma$ using the one-sided counting lemma. We omit the details.

\begin{theorem}\label{thm:Ram-mult}
For each $\epsilon>0$ and graph $H$, there is $c>0$ such that if $\Gamma$ is a $(p,\beta)$-jumbled graph on $n$ vertices with $\beta \leq cp^{d_2(H)+3}n$ then
every $2$-edge-coloring of $\Gamma$ contains at least $(c_H-\epsilon)p^{e(H)}n^{v(H)}$ labeled monochromatic copies of $H$.
\end{theorem}

Maybe the earliest result on Ramsey multiplicity is Goodman's theorem \cite{G59}, which determines $c_{K_3,n}$ and, in particular, implies $c_{K_3}=\frac{1}{4}$. The next theorem shows an extension of Goodman's theorem to pseudorandom graphs, giving an optimal jumbledness condition to imply $c_{H,\Gamma}=\frac{1}{4}-o(1)$.

\begin{theorem}
  \label{thm:Goodman}
  If $\Gamma$ is a $(p, \beta)$-jumbled graph on $n$ vertices with $\beta \leq \frac{1}{10}p^2n$, then every $2$-edge-coloring of $\Gamma$ contains at least $(p^3-10p\frac{\beta}{n})\frac{n^3}{24}$ monochromatic triangles.
\end{theorem}

The proof of this theorem follows by first noting that $T = A + 2M$, where $A$ denotes the number of triangles in $\Gamma$, $M$ the number of monochromatic triangles in $\Gamma$, and $T$ the number of ordered triples $(a,b,c)$ of vertices of $\Gamma$ which form a triangle such that $(a,b)$ and $(a,c)$ are the same color. We then give an upper bound for $A$ and a lower bound for $T$ using the jumbledness conditions and standard inequalities. We omit the precise details.

The previous theorem has the following immediate corollary, giving an optimal jumbledness condition to imply that a graph is $(K_3,2)$-Ramsey.

\begin{corollary}
  If $\Gamma$ is a $(p, \beta)$-jumbled graph on $n$ vertices with $\beta < \frac{p^2}{10}n$, then $\Gamma$ is $(K_3,2)$-Ramsey.
\end{corollary}

Define the Tur\'an multiplicity $\rho_{H,d,n}$ to be the minimum, over all graphs $G$ on $n$ vertices with edge density at least $d$, of the fraction of copies of $H$ in $K_n$ which are also in $G$. Let $\rho_{H,d}$ be the limit of $\rho_{H,d,n}$ as $n \to \infty$. This limit exists by an averaging argument. The conjectures of Erd\H{o}s-Simonovits \cite{Si84} and Sidorenko \cite{Si93} mentioned earlier can be stated as $\rho_{H,d}=d^{e(H)}$ for bipartite $H$.
Recently, Reiher \cite{Re12}, extending work of Razborov \cite{R08} and Nikiforov \cite{N11} for $t = 3$ and $4$, determined $\rho_{K_t, d}$ for all $t \geq 3$.

We can similarly extend these results to the sparse setting. Let $\rho_{H,d,\Gamma}$ be the minimum, over all subgraphs $G$ of $\Gamma$ with at least $d e(\Gamma)$ edges, of the fraction of copies of $H$ in $\Gamma$ which are also in $G$. We have the following result, which gives jumbledness conditions on $\Gamma$ which imply that $\rho_{H,d,\Gamma}$ is close to $\rho_{H,d}$.

\begin{theorem}\label{thm:Tur-mult}
For each $\epsilon>0$ and graph $H$ there is $c>0$ such that if $\Gamma$ is a $(p,\beta)$-jumbled graph on $n$ vertices with $\beta \leq cp^{d_2(H)+3}n$ then
every subgraph of $\Gamma$ with at least $d e(\Gamma)$ edges contains at least $(\rho_{H,d}-\epsilon)p^{e(H)}n^{v(H)}$ labeled copies of $H$.
\end{theorem}

\bibliographystyle{abbrv}
\bibliography{references}

\begin{thebibliography}{10}

\bibitem{AHS12}
E.~Aigner-Horev, H.~H\`an, and M.~Schacht.
\newblock Extremal results for odd cycles in sparse pseudorandom graphs.
\newblock {\em Submitted}.

\bibitem{A94}
N.~Alon.
\newblock Explicit {R}amsey graphs and orthonormal labelings.
\newblock {\em Electron. J. Combin.}, 1:Research Paper 12, approx.\ 8 pp.\
  (electronic), 1994.

\bibitem{ACHKRS10}
N.~Alon, A.~Coja-Oghlan, H.~H{{\`a}}n, M.~Kang, V.~R{{\"o}}dl, and M.~Schacht.
\newblock Quasi-randomness and algorithmic regularity for graphs with general
  degree distributions.
\newblock {\em SIAM J. Comput.}, 39(6):2336--2362, 2010.

\bibitem{ADLRY94}
N.~Alon, R.~A. Duke, H.~Lefmann, V.~R{\"o}dl, and R.~Yuster.
\newblock The algorithmic aspects of the regularity lemma.
\newblock {\em J. Algorithms}, 16(1):80--109, 1994.

\bibitem{AFKS00}
N.~Alon, E.~Fischer, M.~Krivelevich, and M.~Szegedy.
\newblock Efficient testing of large graphs.
\newblock {\em Combinatorica}, 20(4):451--476, 2000.

\bibitem{AK98}
N.~Alon and N.~Kahale.
\newblock Approximating the independence number via the {$\theta$}-function.
\newblock {\em Math. Programming}, 80(3, Ser. A):253--264, 1998.

\bibitem{AKS03}
N.~Alon, M.~Krivelevich, and B.~Sudakov.
\newblock Tur\'an numbers of bipartite graphs and related {R}amsey-type
  questions.
\newblock {\em Combin. Probab. Comput.}, 12(5-6):477--494, 2003.
\newblock Special issue on Ramsey theory.

\bibitem{AN06}
N.~Alon and A.~Naor.
\newblock Approximating the cut-norm via {G}rothendieck's inequality.
\newblock {\em SIAM J. Comput.}, 35(4):787--803 (electronic), 2006.

\bibitem{ASS09}
N.~Alon, A.~Shapira, and B.~Sudakov.
\newblock Additive approximation for edge-deletion problems.
\newblock {\em Ann. of Math. (2)}, 170(1):371--411, 2009.

\bibitem{BMS12}
J.~Balogh, R.~Morris, and W.~Samotij.
\newblock Independent sets in hypergraphs.
\newblock {\em Submitted}.

\bibitem{BL06}
Y.~Bilu and N.~Linial.
\newblock Lifts, discrepancy and nearly optimal spectral gap.
\newblock {\em Combinatorica}, 26(5):495--519, 2006.

\bibitem{BCLSV08}
C.~Borgs, J.~T. Chayes, L.~Lov{{\'a}}sz, V.~T. S{{\'o}}s, and K.~Vesztergombi.
\newblock Convergent sequences of dense graphs. {I}. {S}ubgraph frequencies,
  metric properties and testing.
\newblock {\em Adv. Math.}, 219(6):1801--1851, 2008.

\bibitem{BR80}
S.~A. Burr and V.~Rosta.
\newblock On the {R}amsey multiplicities of graphs---problems and recent
  results.
\newblock {\em J. Graph Theory}, 4(4):347--361, 1980.

\bibitem{B09}
S.~Butler.
\newblock Induced-universal graphs for graphs with bounded maximum degree.
\newblock {\em Graphs Combin.}, 25(4):461--468, 2009.

\bibitem{C05}
F.~R.~K. Chung.
\newblock A spectral {T}ur\'an theorem.
\newblock {\em Combin. Probab. Comput.}, 14(5-6):755--767, 2005.

\bibitem{CG91}
F.~R.~K. Chung and R.~L. Graham.
\newblock Quasi-random set systems.
\newblock {\em J. Amer. Math. Soc.}, 4(1):151--196, 1991.

\bibitem{CG02}
F.~R.~K. Chung and R.~L. Graham.
\newblock Sparse quasi-random graphs.
\newblock {\em Combinatorica}, 22(2):217--244, 2002.
\newblock Special issue: Paul Erd{\H{o}}s and his mathematics.

\bibitem{CG08}
F.~R.~K. Chung and R.~L. Graham.
\newblock Quasi-random graphs with given degree sequences.
\newblock {\em Random Structures Algorithms}, 32(1):1--19, 2008.

\bibitem{CGW89}
F.~R.~K. Chung, R.~L. Graham, and R.~M. Wilson.
\newblock Quasi-random graphs.
\newblock {\em Combinatorica}, 9(4):345--362, 1989.

\bibitem{CRST83}
V.~Chvat{\'a}l, V.~R{\"o}dl, E.~Szemer{\'e}di, and W.~T. Trotter, Jr.
\newblock The {R}amsey number of a graph with bounded maximum degree.
\newblock {\em J. Combin. Theory Ser. B}, 34(3):239--243, 1983.

\bibitem{CCF09}
A.~Coja-Oghlan, C.~Cooper, and A.~Frieze.
\newblock An efficient sparse regularity concept.
\newblock In {\em Proceedings of the {T}wentieth {A}nnual {ACM}-{SIAM}
  {S}ymposium on {D}iscrete {A}lgorithms}, pages 207--216, Philadelphia, PA,
  2009. SIAM.

\bibitem{C09}
D.~Conlon.
\newblock A new upper bound for diagonal {R}amsey numbers.
\newblock {\em Ann. of Math. (2)}, 170(2):941--960, 2009.

\bibitem{CF12}
D.~Conlon and J.~Fox.
\newblock Bounds for graph regularity and removal lemmas.
\newblock {\em to appear in GAFA}.

\bibitem{CFS10}
D.~Conlon, J.~Fox, and B.~Sudakov.
\newblock An approximate version of {S}idorenko's conjecture.
\newblock {\em Geom. Funct. Anal.}, 20(6):1354--1366, 2010.

\bibitem{CG12}
D.~Conlon and W.~T. Gowers.
\newblock Combinatorial theorems in sparse random sets.
\newblock {\em Submitted}.

\bibitem{CGSS12}
D.~Conlon, W.~T. Gowers, W.~Samotij, and M.~Schacht.
\newblock On the {K\L R} conjecture in random graphs.
\newblock {\em Submitted}.

\bibitem{D75}
W.~Deuber.
\newblock Generalizations of {R}amsey's theorem.
\newblock In {\em Infinite and finite sets ({C}olloq., {K}eszthely, 1973;
  dedicated to {P}. {E}rd{\H o}s on his 60th birthday), {V}ol. {I}}, pages
  323--332. Colloq. Math. Soc. J{\'a}nos Bolyai, Vol. 10. North-Holland,
  Amsterdam, 1975.

\bibitem{DR08}
A.~Dudek and V.~R{\"o}dl.
\newblock On the {F}olkman number {$f(2,3,4)$}.
\newblock {\em Experiment. Math.}, 17(1):63--67, 2008.

\bibitem{DLR95}
R.~A. Duke, H.~Lefmann, and V.~R{\"o}dl.
\newblock A fast approximation algorithm for computing the frequencies of
  subgraphs in a given graph.
\newblock {\em SIAM J. Comput.}, 24(3):598--620, 1995.

\bibitem{ES46}
P.~Erd\H{o}s and A.~H. Stone.
\newblock On the structure of linear graphs.
\newblock {\em Bull. Amer. Math. Soc.}, 52:1087--1091, 1946.

\bibitem{E62}
P.~Erd{\H{o}}s.
\newblock On the number of complete subgraphs contained in certain graphs.
\newblock {\em Magyar Tud. Akad. Mat. Kutat\'o Int. K\"ozl.}, 7:459--464, 1962.

\bibitem{E67}
P.~Erd{\H{o}}s.
\newblock Some recent results on extremal problems in graph theory. {R}esults.
\newblock In {\em Theory of {G}raphs ({I}nternat. {S}ympos., {R}ome, 1966)},
  pages 117--123 (English); pp. 124--130 (French). Gordon and Breach, New York,
  1967.

\bibitem{EGPS88}
P.~Erd{\H{o}}s, M.~Goldberg, J.~Pach, and J.~Spencer.
\newblock Cutting a graph into two dissimilar halves.
\newblock {\em J. Graph Theory}, 12(1):121--131, 1988.

\bibitem{EHP75}
P.~Erd{\H{o}}s, A.~Hajnal, and L.~P{{\'o}}sa.
\newblock Strong embeddings of graphs into colored graphs.
\newblock In {\em Infinite and finite sets ({C}olloq., {K}eszthely, 1973;
  dedicated to {P}. {E}rd{\H o}s on his 60th birthday), {V}ol. {I}}, pages
  585--595. Colloq. Math. Soc. J{\'a}nos Bolyai, Vol. 10. North-Holland,
  Amsterdam, 1975.

\bibitem{ES71}
P.~Erd{\H{o}}s and J.~Spencer.
\newblock Imbalances in $k$-colorations.
\newblock {\em Networks}, 1:379--385, 1971/72.

\bibitem{F08}
J.~Fox.
\newblock There exist graphs with super-exponential {R}amsey multiplicity
  constant.
\newblock {\em J. Graph Theory}, 57(2):89--98, 2008.

\bibitem{F11}
J.~Fox.
\newblock A new proof of the graph removal lemma.
\newblock {\em Ann. of Math. (2)}, 174(1):561--579, 2011.

\bibitem{FLS12}
J.~Fox, C.~Lee, and B.~Sudakov.
\newblock Chromatic number, clique subdivisions, and the conjectures of
  {H}aj\'os and {E}rd{\H o}s-{F}ajtlowicz.
\newblock {\em Submitted}.

\bibitem{FS08}
J.~Fox and B.~Sudakov.
\newblock Induced {R}amsey-type theorems.
\newblock {\em Adv. Math.}, 219(6):1771--1800, 2008.

\bibitem{FS09}
J.~Fox and B.~Sudakov.
\newblock Two remarks on the {B}urr-{E}rd{\H o}s conjecture.
\newblock {\em European J. Combin.}, 30(7):1630--1645, 2009.

\bibitem{FRS10}
E.~Friedgut, V.~R{\"o}dl, and M.~Schacht.
\newblock Ramsey properties of random discrete structures.
\newblock {\em Random Structures Algorithms}, 37(4):407--436, 2010.

\bibitem{GKRS07}
S.~Gerke, Y.~Kohayakawa, V.~R{\"o}dl, and A.~Steger.
\newblock Small subsets inherit sparse {$\epsilon$}-regularity.
\newblock {\em J. Combin. Theory Ser. B}, 97(1):34--56, 2007.

\bibitem{GMS07}
S.~Gerke, M.~Marciniszyn, and A.~Steger.
\newblock A probabilistic counting lemma for complete graphs.
\newblock {\em Random Structures Algorithms}, 31(4):517--534, 2007.

\bibitem{GS05}
S.~Gerke and A.~Steger.
\newblock The sparse regularity lemma and its applications.
\newblock In {\em Surveys in Combinatorics 2005}, volume 327 of {\em London
  Math. Soc. Lecture Note Ser.}, pages 227--258. Cambridge Univ. Press,
  Cambridge, 2005.

\bibitem{G59}
A.~W. Goodman.
\newblock On sets of acquaintances and strangers at any party.
\newblock {\em Amer. Math. Monthly}, 66:778--783, 1959.

\bibitem{Gow97}
W.~T. Gowers.
\newblock Lower bounds of tower type for {S}zemer\'edi's uniformity lemma.
\newblock {\em Geom. Funct. Anal.}, 7(2):322--337, 1997.

\bibitem{G06}
W.~T. Gowers.
\newblock Quasirandomness, counting and regularity for 3-uniform hypergraphs.
\newblock {\em Combin. Probab. Comput.}, 15(1-2):143--184, 2006.

\bibitem{G07}
W.~T. Gowers.
\newblock Hypergraph regularity and the multidimensional {S}zemer{\'e}di
  theorem.
\newblock {\em Ann. of Math. (2)}, 166(3):897--946, 2007.

\bibitem{G08}
W.~T. Gowers.
\newblock Quasirandom groups.
\newblock {\em Combin. Probab. Comput.}, 17(3):363--387, 2008.

\bibitem{GRR00}
R.~L. Graham, V.~R{\"o}dl, and A.~Ruci{\'n}ski.
\newblock On graphs with linear {R}amsey numbers.
\newblock {\em J. Graph Theory}, 35(3):176--192, 2000.

\bibitem{G05a}
B.~Green.
\newblock A {S}zemer{\'e}di-type regularity lemma in abelian groups, with
  applications.
\newblock {\em Geom. Funct. Anal.}, 15(2):340--376, 2005.

\bibitem{GT08}
B.~Green and T.~Tao.
\newblock The primes contain arbitrarily long arithmetic progressions.
\newblock {\em Ann. of Math. (2)}, 167(2):481--547, 2008.

\bibitem{HKL95}
P.~E. Haxell, Y.~Kohayakawa, and T.~{\L}uczak.
\newblock Tur\'an's extremal problem in random graphs: forbidding even cycles.
\newblock {\em J. Combin. Theory Ser. B}, 64(2):273--287, 1995.

\bibitem{K97}
Y.~Kohayakawa.
\newblock Szemer\'edi's regularity lemma for sparse graphs.
\newblock In {\em Foundations of computational mathematics ({R}io de {J}aneiro,
  1997)}, pages 216--230. Springer, Berlin, 1997.

\bibitem{KLR97}
Y.~Kohayakawa, T.~{\L}uczak, and V.~R{\"o}dl.
\newblock On {$K^4$}-free subgraphs of random graphs.
\newblock {\em Combinatorica}, 17(2):173--213, 1997.

\bibitem{KR03}
Y.~Kohayakawa and V.~R{\"o}dl.
\newblock Regular pairs in sparse random graphs. {I}.
\newblock {\em Random Structures Algorithms}, 22(4):359--434, 2003.

\bibitem{KR03a}
Y.~Kohayakawa and V.~R{{\"o}}dl.
\newblock Szemer{\'e}di's regularity lemma and quasi-randomness.
\newblock In {\em Recent advances in algorithms and combinatorics}, volume~11
  of {\em CMS Books Math./Ouvrages Math. SMC}, pages 289--351. Springer, New
  York, 2003.

\bibitem{KRSSS07}
Y.~Kohayakawa, V.~R{\"o}dl, M.~Schacht, P.~Sissokho, and J.~Skokan.
\newblock Tur\'an's theorem for pseudo-random graphs.
\newblock {\em J. Combin. Theory Ser. A}, 114(4):631--657, 2007.

\bibitem{KRSS10}
Y.~Kohayakawa, V.~R\"odl, M.~Schacht, and J.~Skokan.
\newblock On the triangle removal lemma for subgraphs of sparse pseudorandom
  graphs.
\newblock In {\em An Irregular Mind (Szemer\'edi is 70)}, volume~21 of {\em
  Bolyai Soc. Math. Stud.}, pages 359--404. Springer Berlin, 2010.

\bibitem{KRSSz11}
Y.~Kohayakawa, V.~R{\"o}dl, M.~Schacht, and E.~Szemer{\'e}di.
\newblock Sparse partition universal graphs for graphs of bounded degree.
\newblock {\em Adv. Math.}, 226(6):5041--5065, 2011.

\bibitem{KRS04}
Y.~Kohayakawa, V.~R{\"o}dl, and P.~Sissokho.
\newblock Embedding graphs with bounded degree in sparse pseudorandom graphs.
\newblock {\em Israel J. Math.}, 139:93--137, 2004.

\bibitem{KS96}
J.~Koml{\'o}s and M.~Simonovits.
\newblock Szemer\'edi's regularity lemma and its applications in graph theory.
\newblock In {\em Combinatorics, {P}aul {E}rd{\H o}s is eighty, {V}ol.\ 2
  ({K}eszthely, 1993)}, volume~2 of {\em Bolyai Soc. Math. Stud.}, pages
  295--352. J\'anos Bolyai Math. Soc., Budapest, 1996.

\bibitem{KSV09}
D.~Kr{{\'a}}l, O.~Serra, and L.~Vena.
\newblock A combinatorial proof of the removal lemma for groups.
\newblock {\em J. Combin. Theory Ser. A}, 116(4):971--978, 2009.

\bibitem{KS06}
M.~Krivelevich and B.~Sudakov.
\newblock Pseudo-random graphs.
\newblock In {\em More sets, graphs and numbers}, volume~15 of {\em Bolyai Soc.
  Math. Stud.}, pages 199--262. Springer, Berlin, 2006.

\bibitem{LS12}
J.~Li and B.~Szegedy.
\newblock On the logarithmic calculus and {S}idorenko's conjecture.
\newblock {\em Submitted}.

\bibitem{LS06}
L.~Lov{\'a}sz and B.~Szegedy.
\newblock Limits of dense graph sequences.
\newblock {\em J. Combin. Theory Ser. B}, 96(6):933--957, 2006.

\bibitem{Lu07}
L.~Lu.
\newblock Explicit construction of small {F}olkman graphs.
\newblock {\em SIAM J. Discrete Math.}, 21(4):1053--1060, 2007.

\bibitem{LR96}
T.~{\L}uczak and V.~R{{\"o}}dl.
\newblock On induced {R}amsey numbers for graphs with bounded maximum degree.
\newblock {\em J. Combin. Theory Ser. B}, 66(2):324--333, 1996.

\bibitem{NRS06}
B.~Nagle, V.~R{\"o}dl, and M.~Schacht.
\newblock The counting lemma for regular {$k$}-uniform hypergraphs.
\newblock {\em Random Structures Algorithms}, 28(2):113--179, 2006.

\bibitem{N11}
V.~Nikiforov.
\newblock The number of cliques in graphs of given order and size.
\newblock {\em Trans. Amer. Math. Soc.}, 363(3):1599--1618, 2011.

\bibitem{R30}
F.~P. Ramsey.
\newblock On a problem of formal logic.
\newblock {\em Proc. London Math. Soc. Ser. 2}, 30:264--286, 1930.

\bibitem{R08}
A.~A. Razborov.
\newblock On the minimal density of triangles in graphs.
\newblock {\em Combin. Probab. Comput.}, 17(4):603--618, 2008.

\bibitem{Re12}
C.~Reiher.
\newblock The clique density theorem.
\newblock {\em Submitted}.

\bibitem{R73}
V.~R{{\"o}}dl.
\newblock The dimension of a graph and generalized ramsey theorems.
\newblock Master's thesis, Charles University, 1973.

\bibitem{RR95}
V.~R{\"o}dl and A.~Ruci{\'n}ski.
\newblock Threshold functions for {R}amsey properties.
\newblock {\em J. Amer. Math. Soc.}, 8(4):917--942, 1995.

\bibitem{RS10}
V.~R{\"o}dl and M.~Schacht.
\newblock Regularity lemmas for graphs.
\newblock In {\em Fete of Combinatorics and Computer Science}, volume~20 of
  {\em Bolyai Soc. Math. Stud.}, pages 287--325. J\'anos Bolyai Math. Soc.,
  Budapest, 2010.

\bibitem{RS04}
V.~R{\"o}dl and J.~Skokan.
\newblock Regularity lemma for {$k$}-uniform hypergraphs.
\newblock {\em Random Structures Algorithms}, 25(1):1--42, 2004.

\bibitem{RS78}
I.~Z. Ruzsa and E.~Szemer{\'e}di.
\newblock Triple systems with no six points carrying three triangles.
\newblock In {\em Combinatorics ({P}roc. {F}ifth {H}ungarian {C}olloq.,
  {K}eszthely, 1976), {V}ol. {II}}, volume~18 of {\em Colloq. Math. Soc.
  J\'anos Bolyai}, pages 939--945. North-Holland, Amsterdam, 1978.

\bibitem{S12}
M.~Schacht.
\newblock Extremal results for discrete random structures.
\newblock {\em Submitted}.

\bibitem{Sc11}
A.~Scott.
\newblock Szemer\'edi's regularity lemma for matrices and sparse graphs.
\newblock {\em Combin. Probab. Comput.}, 20(3):455--466, 2011.

\bibitem{Si93}
A.~Sidorenko.
\newblock A correlation inequality for bipartite graphs.
\newblock {\em Graphs Combin.}, 9(2):201--204, 1993.

\bibitem{Si68}
M.~Simonovits.
\newblock A method for solving extremal problems in graph theory, stability
  problems.
\newblock In {\em Theory of {G}raphs ({P}roc. {C}olloq., {T}ihany, 1966)},
  pages 279--319. Academic Press, New York, 1968.

\bibitem{Si84}
M.~Simonovits.
\newblock Extremal graph problems, degenerate extremal problems, and
  supersaturated graphs.
\newblock In {\em Progress in graph theory ({W}aterloo, {O}nt., 1982)}, pages
  419--437. Academic Press, Toronto, ON, 1984.

\bibitem{SSV05}
B.~Sudakov, T.~Szab{\'o}, and V.~H. Vu.
\newblock A generalization of {T}ur\'an's theorem.
\newblock {\em J. Graph Theory}, 49(3):187--195, 2005.

\bibitem{Sz75}
E.~Szemer{\'e}di.
\newblock On sets of integers containing no {$k$} elements in arithmetic
  progression.
\newblock {\em Acta Arith.}, 27:199--245, 1975.

\bibitem{Sz78}
E.~Szemer{\'e}di.
\newblock Regular partitions of graphs.
\newblock In {\em Probl\`emes combinatoires et th\'eorie des graphes ({C}olloq.
  {I}nternat. {CNRS}, {U}niv. {O}rsay, {O}rsay, 1976)}, volume 260 of {\em
  Colloq. Internat. CNRS}, pages 399--401. CNRS, Paris, 1978.

\bibitem{T06}
T.~Tao.
\newblock A variant of the hypergraph removal lemma.
\newblock {\em J. Combin. Theory Ser. A}, 113(7):1257--1280, 2006.

\bibitem{T85}
A.~Thomason.
\newblock Pseudorandom graphs.
\newblock In {\em Random graphs '85 ({P}ozna\'n, 1985)}, volume 144 of {\em
  North-Holland Math. Stud.}, pages 307--331. North-Holland, Amsterdam, 1987.

\bibitem{T87}
A.~Thomason.
\newblock Random graphs, strongly regular graphs and pseudorandom graphs.
\newblock In {\em Surveys in Combinatorics 1987 ({N}ew {C}ross, 1987)}, volume
  123 of {\em London Math. Soc. Lecture Note Ser.}, pages 173--195. Cambridge
  Univ. Press, Cambridge, 1987.

\bibitem{T89}
A.~Thomason.
\newblock A disproof of a conjecture of {E}rd{\H o}s in {R}amsey theory.
\newblock {\em J. London Math. Soc. (2)}, 39(2):246--255, 1989.

\bibitem{T41}
P.~Tur{\'a}n.
\newblock Eine {E}xtremalaufgabe aus der {G}raphentheorie.
\newblock {\em Mat. Fiz. Lapok}, 48:436--452, 1941.

\end{thebibliography}

\end{document}